\documentclass{amsart}

\IfFileExists{srcltx.sty}{\usepackage{srcltx}}

\DeclareMathSizes{10}{8}{5}{4}

\usepackage{tikz}
\usetikzlibrary{arrows,matrix}

\IfFileExists{srcltx.sty}{\usepackage{srcltx}}

\numberwithin{equation}{section}

\usepackage[latin1]{inputenc}
\usepackage{xspace,amssymb,amsfonts,euscript}
\usepackage{amsthm,amsmath}
\usepackage{euscript}
\input xy \xyoption {all}

\RequirePackage{color}
\definecolor{myred}{rgb}{0.75,0,0}
\definecolor{mygreen}{rgb}{0,0.5,0}
\definecolor{myblue}{rgb}{0,0,0.65}

\usepackage{caption}
\usepackage{subcaption}

\RequirePackage{ifpdf}
\ifpdf
  \IfFileExists{pdfsync.sty}{\RequirePackage{pdfsync}}{}
  \RequirePackage[pdftex,
   colorlinks = true,
   urlcolor = myblue, 
   citecolor = mygreen, 
   linkcolor = myred, 
   pagebackref,
   bookmarksopen=true]{hyperref}
\else
  \RequirePackage[hypertex]{hyperref}
\fi

\RequirePackage{ae, aecompl, aeguill} 

  \def\ag{{\mathfrak a}}  
  \def\bg{{\mathfrak b}}  
  \def\cg{{\mathfrak c}}  \def\CM{{\mathbb{C}}}

  \def\gg{{\mathfrak g}}  
  \def\hg{{\mathfrak h}}

  \def\lg{{\mathfrak l}}  
  \def\mg{{\mathfrak m}}  
  \def\ng{{\mathfrak n}}  
    
    \def\PM{{\mathbb{P}}}
  \def\qg{{\mathfrak q}}  \def\QM{{\mathbb{Q}}}
    
\def\SG{{\mathfrak S}}  \def\sg{{\mathfrak s}}

  \def\zg{{\mathfrak z}}  
\def\0{{\mathcal O}}

    \def\BC{{\mathcal{B}}}
    \def\CC{{\mathcal{C}}}
    
    \def\EC{{\mathcal{E}}}

    \def\JC{{\mathcal{J}}}

    \def\NC{{\mathcal{N}}}
    \def\OC{{\mathcal{O}}}

    \def\SC{{\mathcal{S}}}

    \def\YC{{\mathcal{Y}}}
    
\def\MATHFRAK{\mathfrak}
\def\cit{\mathbb C}
\def\P{{\mathcal P}}
\def\qit{\mathbb Q}
\def\pit{\mathbb P}

\DeclareMathOperator{\diag}{diag}

\def\SL{\rm{SL}}
\def\Sp{\rm{Sp}}
\def\Frac{\rm{Frac}}
\DeclareMathOperator{\Int}{Int}
\def\GL{\rm{GL}}

\def\slice{\SC_{\0,e}}

\newcommand{\nc}{\newcommand} \newcommand{\renc}{\renewcommand}

\newcommand{\rdots}{\mathinner{ \mkern1mu\raise1pt\hbox{.}
    \mkern2mu\raise4pt\hbox{.}
    \mkern2mu\raise7pt\vbox{\kern7pt\hbox{.}}\mkern1mu}}

\DeclareMathOperator{\Ad}{Ad} \DeclareMathOperator{\ad}{ad}
\DeclareMathOperator{\Tr}{Tr} \DeclareMathOperator{\Irr}{Irr}

\DeclareMathOperator{\Res}{Res}

\def\ov{\overline}

\def\to{\rightarrow}

\nc{\triright}{\stackrel{[1]}{\to}}

\nc{\Br}{\mathcal{B}}
\nc{\HotRR}{{}_R\mathcal{K}_R}
\nc{\HotR}{\mathcal{K}_R}
\nc{\excise}[1]{}
\nc{\defect}{\text{df}}
\nc{\h}[1]{\underline{H}_{#1}}
\nc{\cl}{\ov\OC}

\nc{\Ga}{\mathbb{G}_a} 
\nc{\Gm}{\mathbb{G}_m} 

\nc{\Perv}{{\mathbf{P}}}

\nc{\IH}{{\mathrm{IH}}}

\nc{\ic}{\mathbf{IC}}

\nc{\gl}{{\mathfrak{gl}}}
\renc{\sl}{{\mathfrak{sl}}}
\renc{\sp}{{\mathfrak{sp}}}
\nc{\so}{{\mathfrak{so}}}

\def\SG{{\mathfrak S}}

\def\Res{\mathop{\rm Res}\nolimits}
\def\Pic{\mathop{\rm Pic}\nolimits}
\def\Lie{\mathop{\rm Lie}\nolimits}

\def\codim{\mathop{\rm codim}\nolimits}
\def\Sing{\mathop{\rm Sing}\nolimits}

\DeclareMathOperator{\im}{{\mathrm{Im}}}

\newtheorem{theorem}{Theorem}[section]
\newtheorem{lemma}[theorem]{Lemma}
\newtheorem{proposition}[theorem]{Proposition}
\newtheorem{corollary}[theorem]{Corollary}

\theoremstyle{definition}

\theoremstyle{remark}
\newtheorem{remark}[theorem]{Remark}

\newtheorem{example}[theorem]{Example}

\DeclareMathOperator{\Aut}{Aut}

\DeclareMathOperator{\Spec}{Spec}

\title[Generic singularities of nilpotent orbit closures]{Generic singularities of nilpotent orbit closures}

\author{Baohua Fu}
\address{Hua Loo-Keng Key Laboratory  of  Mathematics  and  AMSS, Chinese Academy of Sciences,
55 ZhongGuanCun East Road, Beijing, 100190, P. R. China}
\email{bhfu@math.ac.cn}

\author{Daniel Juteau}
\address{LMNO, Universite de Caen Basse-Normandie, CNRS, BP 5186, 14032 
Caen Cedex, France}
\email{daniel.juteau@unicaen.fr}

\author{Paul Levy}
\address{Department of Mathematics and Statistics
Fylde College,
Lancaster University
Lancaster
LA1 4YF, United Kingdom}
\email{p.d.levy@lancaster.ac.uk}

\author{Eric Sommers}
\address{Department of Mathematics and Statistics, University of Massachusetts, Amherst,
MA 01003-4515, USA}
\email{esommers@math.umass.edu}

\begin{document}

\begin{abstract}
According to a theorem of Brieskorn and Slodowy, the
intersection of the nilpotent cone of a simple Lie algebra with a
transverse slice to the subregular nilpotent orbit is a simple surface
singularity.  At the opposite extremity of the poset of nilpotent orbits, 
the closure of the minimal
nilpotent orbit is also an isolated symplectic singularity, called a
minimal singularity.  For classical Lie algebras, Kraft and Procesi
showed that these two types of singularities suffice to describe all
generic singularities of nilpotent orbit closures: specifically, any
such singularity is either a simple surface singularity, a minimal
singularity, or a union of two simple surface singularities of type
$A_{2k-1}$.  In the present paper, we complete the picture by
determining the generic singularities of all nilpotent orbit closures
in {\it exceptional} Lie algebras (up to normalization in a few
cases).  We summarize the results in some graphs at the end of the paper.

In most cases, we also obtain simple surface singularities or minimal singularities, though
often with more complicated branching than occurs in
the classical types.  
There are, however, six singularities that do 
not occur in the classical types.  Three of these are unibranch non-normal singularities:
an $\SL_2(\CM)$-variety whose normalization is ${\mathbb A}^2$,   
an $\Sp_4(\CM)$-variety whose normalization is ${\mathbb A}^4$,   
and a two-dimensional variety whose normalization is the simple surface singularity $A_3$.
In addition, there are three 4-dimensional isolated singularities each appearing once.   
We also study an intrinsic symmetry action on the singularities, 
extending Slodowy's work for the singularity of the nilpotent cone at a point in the subregular orbit.
\end{abstract}

\maketitle

\section{Introduction}

\subsection{Generic singularities of nilpotent orbit closures} \label{first part}

Let $G$ be a connected, simple algebraic group of adjoint type
over the complex numbers $\cit$, with Lie algebra $\gg$.
A nilpotent orbit $\OC$ in $\gg$ is the orbit of a nilpotent element
under the adjoint action of $G$.  
Its closure $\cl$ is a union of finitely many nilpotent orbits.
The partial order on nilpotent orbits is defined to be the closure ordering.

We are interested in the singularities of $\cl$ at points of maximal orbits of its singular locus.
Such singularities are known as the {\it generic singularities} of $\cl$.
Kraft and Procesi determined the generic singularities in the classical types, 
while Brieskorn and Slodowy determined the generic
singularities of the whole nilpotent cone $\NC$ for $\gg$ of any type.
The goal of this paper is to determine the generic singularities of $\cl$
when $\gg$ is of exceptional type.  

In fact, the singular locus of $\cl$ coincides with the boundary of $\OC$ in $\cl$, 
as was shown by Namikawa using results of Kaledin \cite{Nam3}, \cite{Kaledin}.  This result also 
follows from the main theorem in this paper in the exceptional types and from 
Kraft and Procesi's work in the classical types \cite{Kraft-Procesi:GLn}, \cite{Kraft-Procesi:classical}.
Therefore to study generic singularities of $\cl$, 
it suffices to consider each maximal orbit $\0'$ in the boundary of $\OC$ in $\cl$.
We call such an $\0'$ a {\em minimal degeneration} of $\OC$.

The local geometry of $\cl$ at a point $e\in\0'$ 
is determined by the intersection of $\cl$ with a
transverse slice in $\gg$ to $\0'$ at $e$.   
Such a transverse slice in $\gg$ always exists and 
is provided by the affine space $\SC_{e} = e + \gg^f$, known as the Slodowy slice.  
Here, $e$ and $f$ are the nilpotent parts of an $\sl_2$-triple and $\gg^f$ is the centralizer of $f$ in $\gg$.
The local geometry of $\cl$ at a point $e$ is therefore encoded in 
$\SC_{\0, e} = \cl \cap \SC_e$, which we call a nilpotent Slodowy slice.
If $\0'$ is a minimal degeneration of $\0$, then
$\SC_{\0, e}$ has an isolated singularity at $e$.  The generic singularities
of $\cl$ can therefore be determined by studying the various 
$\SC_{\0, e}$, as $\0'$ runs over all minimal degenerations and $e \in \0'$.
The isomorphism type of the variety $\SC_{\0, e}$ is independent of the choice of $e$.

The main theorem of this paper is a classification of $\SC_{\0, e}$ up to algebraic isomorphism for
each minimal degeneration $\0'$ of $\0$ in the exceptional types.  
In a few cases, however, we are only able to determine 
the normalization of $\SC_{\0, e}$ and in a few others, we have determined 
$\SC_{\0, e}$ only up to local analytic isomorphism.

\subsection{Symplectic varieties}\label{Gorenstein}\label{Symplectic_singularity}
Recall from \cite{Beauville} that a symplectic variety
is a normal
variety $W$ with a holomorphic symplectic form $\omega$ on its smooth
locus such that for any resolution $\pi: Z \to W$, the pull-back
$\pi^* \omega$ extends to a regular 2-form on $Z$. If this $2$-form is symplectic
(i.e. if it is non-degenerate everywhere), then $\pi$ is called a
symplectic resolution. By a result of
Namikawa \cite{Namikawa}, a normal variety is symplectic if and only
if its singularities are rational Gorenstein and its smooth part
carries a holomorphic symplectic form.

The normalization of a nilpotent orbit $\cl$ is a symplectic variety:
it is well-known that $\OC$ admits a holomorphic non-degenerate closed 2-form 
(see \cite[Ch.\,1.4]{C-M}) and by work of Hinich \cite{Hinich} and Panyushev \cite{Panyushev}, 
the normalization of $\cl$ has only rational Gorenstein singularities.
Hence the normalization of $\cl$ is a symplectic variety.

Since the normalization of $\cl$ has rational Gorenstein singularities,  
the normalization $\widetilde{\SC}_{\0, e}$  of $\SC_{\0,e}$ also has 
rational Gorenstein singularities.
The smooth locus of $\SC_{\0,e}$ admits a symplectic form  by
restriction of the symplectic form on $\0$ (\cite[Corollary 7.2]{Gan-Ginzburg}),
and this yields a symplectic form on the smooth locus of $\widetilde{\SC}_{\0, e}$ since $\slice$ is smooth in codimension one.
Thus by the aforementioned result of Namikawa, 
$\widetilde{\SC}_{\0, e}$ is also a symplectic variety.

The term {\it symplectic singularity} refers to a singularity of a symplectic variety.
A better understanding of isolated symplectic singularities could shed light on the
long-standing conjecture (e.g. \cite{LeBrun}) that a Fano contact manifold is
homogeneous.  The importance of finding new examples of isolated symplectic singularities was stressed in \cite{Beauville}.
It is therefore of interest to determine generic singularities of nilpotent orbits, as a means to find new examples of
isolated symplectic singularities.   Our study of the isolated symplectic singularity $\widetilde{\SC}_{\0, e}$ contributes to this program.

\subsection{Motivation from representation theory}

The topology and geometry of the nilpotent cone $\NC$ have played an important role in representation theory
centered around Springer's construction of Weyl group representations and the resulting Springer correspondence 
(e.g., \cite{Springer:trig},  \cite{Lusztig:Green_sing}, \cite{Borho-MacPherson:homologie}, \cite{Joseph}, \cite{Lusztig:Green_character}). 
The second author of the present paper defined a modular
version of Springer's correspondence \cite{juteau} 
to the effect that the modular representation theory of the Weyl group of $\gg$
is encoded in the geometry of $\NC$.  In particular, its decomposition matrix is a part of the decomposition
matrix for equivariant perverse sheaves on $\NC$.
The connection with decomposition numbers makes it desirable to be able to compute the
stalks of intersection cohomology complexes with modular
coefficients.  In this setting the Lusztig-Shoji algorithm to compute Green functions 
is not available and one has to use other methods, such as Deligne's
construction which is general, but hard to use in practice. To actually
compute modular stalks it is necessary to have a good understanding of
the geometry.  The case of a minimal degeneration is the most tractable.

The decomposition matrices of the exceptional Weyl groups
are known, so here we are not trying to use the geometry to obtain new
information in modular representation theory. However, it is
interesting to see how the reappearance of certain singularities in different nilpotent cones leads to
equalities (or more complicated relationships) between parts of
decomposition matrices. 
In the $GL_n$ case, the row and column removal rule for nilpotent
singularities of \cite{Kraft-Procesi:GLn} gives a geometric
explanation for a similar rule for decomposition matrices of symmetric
groups (\cite{James}, \cite{juteau}).

It would also be interesting to investigate whether the equivalences of
singularities that we obtain in exceptional nilpotent cones have some
significance for studying primitive ideals in finite $W$-algebras (see the survey article \cite{Losev}).

\subsection{Simple surface singularities and their symmetries}

\subsubsection{Simple surface singularities} 

Let $\Gamma$ be a finite subgroup of $\SL_2(\CM) \cong \Sp_2(\CM)$.   Then $\Gamma$ acts on $\CM^2$ and the 
quotient variety $\CM^2/\Gamma$ is an affine symplectic variety with an isolated singularity at the image of $0$.  
This variety is known as a simple surface singularity and also as a rational double point, a du Val singularity, or a Kleinian singularity.
   
Up to conjugacy in $\SL_2(\CM)$, such $\Gamma$ are in bijection with the simply-laced, simple Lie algebras over $\CM$.
The bijection is obtained via the exceptional fiber of a minimal resolution of $\CM^2/\Gamma$.
The exceptional fiber (that is, the inverse image of $0$) is a union of projective lines which intersect transversally.  
The dual graph of the resolution is given by one vertex for each projective line in the exceptional fiber and an edge joining two vertices when the corresponding projective lines intersect.  The dual graph is always a connected, simply-laced Dynkin diagram, which defines the Lie algebra attached
to $\CM^2/\Gamma$.
Hence we denote simple surface singularities using the upper-case letters
$A_k, D_k (k \geq 4), E_6, E_7, E_8$, according to the associated simple Lie algebra. 

In dimension two, an isolated symplectic singularity is equivalent to a simple surface singularity, that is, 
it is locally analytically isomorphic to some $\CM^2/\Gamma$ (cf. \cite[Section 2.1]{Beauville}).
An algebraic version of this result is provided by Proposition \ref{P:2-dim}.
More generally, if $\Gamma \subset {\rm Sp}_{2n}(\CM)$ is a finite subgroup which acts freely on $\CM^{2n} \backslash \{ 0 \}$, 
then the quotient $\CM^{2n}/\Gamma$ is an isolated symplectic singularity.

\subsubsection{Symmetries of simple surface singularities}  \label{surface_symmetry}

Any automorphism of the simple surface singularity $X = \CM^2/\Gamma$ fixes $0 \in X$ 
and induces a permutation of  the projective lines in the exceptional fiber of a minimal resolution.  
Hence it gives rise to a graph automorphism of the dual graph $\Delta$ of $X$.    
Let $\Aut(\Delta)$ be the group of graph automorphisms of $\Delta$. Then $\Aut(\Delta) = 1$ when $\gg$ is $A_1$, $E_7$, or $E_8$;  $\Aut(\Delta) = \SG_3$ when $\gg$ is $D_4$; and otherwise, $\Aut(\Delta) = \SG_2$.

We now address the question of when the action of $\Aut(\Delta)$ on the dual graph comes from an algebraic action on $X$
(cf.\ \cite[III.6]{Slodowy:book}).
When $X$ is of type $A_{2k-1} (k \geq 2)$, $D_{k+1} (k \geq 3)$, or $E_6$,  
then $\Aut(\Delta)$ comes from an algebraic action on $X$.   
In fact, the action is induced from a  subgroup $\Gamma' \subset SL_2(\CM)$ containing $\Gamma$ as a normal subgroup.   More precisely, there exists such a $\Gamma'$ with $\Gamma'/\Gamma \cong \Aut(\Delta)$ and the induced action 
of $\Gamma'/\Gamma$ on the dual graph of $X$ coincides with the action of $\Aut(\Delta)$ on $\Delta$ via this isomorphism.    
Such a $\Gamma'$ is unique.     
The result also holds for any subgroup of $\Aut(\Delta)$, which is relevant only for the $D_4$ case.

Slodowy denotes the pair $(X, K$) consisting of $X$ together with the induced action of $K = \Gamma'/\Gamma$ on $X$ by
\begin{eqnarray*} \label{non_simply_laced_list}
& B_k, & \text{when } X   \text{ is of type } A_{2k-1} \text{ and } K = \SG_2,  \\
& C_k, & \text{when } X  \text{ is of type } D_{k+1} \text{ and } K = \SG_2,  \\
& F_4, & \text{when } X  \text{ is of type } E_{6} \text{ and } K = \SG_2,  \\
& G_2, & \text{when } X  \text{ is of type } D_{4} \text{ and } K = \SG_3.  
\end{eqnarray*}
The reasons for this notation will become clear shortly.  
We also refer to corresponding pairs 
$(\Delta, K)$, where $\Delta$ is the dual graph and $K$ is a subgroup of $\Aut(\Delta)$, in the same way.
The symmetry of the cyclic group of order $3$ when $X$ is of type $D_4$ is not considered.

When $X$ is of type $A_{2k}$, the symmetry of $X$ did not arise in Slodowy's work.  
It does, however, make an appearance in this paper.   
In this case $\Aut(\Delta)= \SG_2$, but the action on the dual graph does not lift to an action on $X$.  
Instead, there is a cyclic group $\langle \sigma \rangle$ of order $4$ acting on $X$, with $\sigma$ acting by non-trivial involution on $\Delta$, but $\sigma^2$ acts non-trivially on $X$.  
This cyclic action is induced from a $\Gamma' \subset \SL_2(\CM)$ corresponding to $D_{2k+3}$.
We define the symmetry of $X$ to be the induced action of $\Gamma'$ on $X$ and denote it by $A^+_{2k}$.
Only the singularities $A_2^+$ and $A_4^+$ will appear in the sequel, and then only when $\gg$ is of type $E_7$ or $E_8$.

\subsection{The regular nilpotent orbit}

\subsubsection{Generic singularities of the nilpotent cone.}

The problem of describing the generic singularities of the nilpotent cone $\NC$ of $\gg$ was carried out by 
Brieskorn \cite{Brieskorn} and Slodowy \cite{Slodowy:book} in confirming a conjecture of
Grothendieck.   In their setting $\0$ is the regular nilpotent orbit and so $\cl$ equals $\NC$, 
and there is only one minimal degeneration, at the subregular nilpotent orbit $\0'$.
Slodowy's result from \cite[IV.8.3]{Slodowy:book} is that when $e \in \0'$, 
the slice $\SC_{\OC,e}$ is algebraically isomorphic to a simple surface singularity. 
Moreover, as in \cite{Brieskorn}, when the Dynkin diagram of $\gg$ is simply-laced, 
the Lie algebra associated to this simple surface singularity is $\gg$.
On the other hand, when $\gg$ is not simply-laced, the singularity $\slice$ is determined from 
the list in \S  \ref{surface_symmetry}.
For example, if $\gg$ is of type $B_k$, then $\SC_{\OC,e}$ is a type $A_{2k-1}$ singularity.   
This explains the notation in the list in \S  \ref{surface_symmetry}.  
Next we explain an intrinsic realization of the symmetry of $\SC_{\OC,e}$ 
when $\gg$ is not simply-laced.

\subsubsection{Intrinsic symmetry action on the slice}\label{symmetry_nsl_case}

Let  $\Delta$ be the Dynkin diagram of $\gg$ and $K \subset \Aut(\Delta)$ be a subgroup.  The group $\Aut(\Delta)$ is trivial 
unless $\gg$ is simply-laced.
The action of $K$ on $\Delta$ can be lifted to an action on $\gg$ as in \cite[Chapter 4.3]{Onishchik-Vinberg:book}:
namely, fix a canonical system of generators of $\gg$.
Then there is a subgroup $\tilde{K} \subset \Aut(\gg)$, isomorphic to $K$, 
which permutes the canonical system of generators, and whose induced action on $\Delta$ coincides with $K$.  
Any two choices of systems of generators define conjugate subgroups of $\Aut(\gg)$.
The automorphisms in $\tilde{K}$ are called {\it diagram automorphisms} of $\gg$.

Now given $\gg$ we can associate a pair $(\gg_s, \tilde{K})$ 
where  $\gg_s$ is a simple, simply-laced Lie algebra with Dynkin diagram $\Delta_s$
and ${\tilde K} \subset \Aut(\gg_s)$ is a lifting of some $K \subset \Aut(\Delta_s)$ 
and $\gg \cong (\gg_s)^{\tilde K}$.
If $\gg$ is already simply-laced, then $\gg = \gg_s$ and $\tilde{K} =1$. 
If $\gg$ is not simply-laced, then the pair $(\Delta_s, K)$
appears in the list in \S  \ref{surface_symmetry}, but according to the type of $^L \gg$, 
where $^L \gg$ is the Langlands dual Lie algebra of $\gg$.

Recall $\sg \cong \sl_2(\CM)$ is the subalgebra of $\gg$ generated by $e$ and $f$.  
Let $C(\sg)$ be the centralizer of $\sg$ in $G$.
Then $C(\sg)$ acts on $\SC_{\0, e}$ for any nilpotent orbit $\0$, fixing the point $e$. 
Also the component group $A(e)$ of the centralizer in $G$ of $e$ is isomorphic to the component group of $C(\sg)$ (see \S \ref{transverse_slices} for more details).
When $e$ is in the subregular nilpotent orbit, 
Slodowy observed that $C(\sg)$ is a semidirect product of its connected component $C(\sg)^\circ$ and a subgroup $H \cong A(e)$.
Moreover $H$ is well-defined up to conjugacy in $C(\sg)$.  
This is immediate except when $\gg$ is of type $B_k$, since otherwise $C(\sg)^\circ$ is trivial. 
Also, $A(e) \cong K$.  In particular, $A(e)$ is trivial if $\gg$ is simply-laced since $G$ is adjoint.

Now let $((^L \Delta)_s, ^L\! K)$ be the pair attached above to $^L \gg$.
We have $A(e) \cong H \cong K \cong \! ^L\!K$.  Then Slodowy's classification and symmetry result can be summarized as follows:
the pair $(\slice, H)$, of $\slice$ together with the action of $H$, corresponds to the pair $((^L \! \Delta)_s, ^L \! K)$
\cite[IV.8.4]{Slodowy:book}.  
\subsection{The other nilpotent orbits in Lie algebras of classical type}
Kraft and Procesi described the generic singularities of nilpotent orbit closures for all the classical groups, up to smooth equivalence (see \S \ref{smooth_equivalence} for the definition of smooth equivalence) \cite{Kraft-Procesi:GLn}, \cite{Kraft-Procesi:classical}. 

\subsubsection{Minimal singularities} \label{def:minimal_singularity}

Let $\OC_{\text{min}}$ be the minimal (non-zero) nilpotent orbit in a simple Lie algebra $\gg$.  
Then $\cl_{\text{min}}$ has an isolated symplectic singularity at $0 \in \cl_{\text{min}}$.    
Following Kraft and Procesi \cite[14.3]{Kraft-Procesi:classical}, 
we refer to $\cl_{\text{min}}$ by the lower case letters for the ambient simple Lie algebra:  
$a_k, b_k, c_k, d_k (k \geq 4), g_2, f_4, e_6, e_7, e_8$.  
The equivalence classes of these singularities, under smooth equivalence, are called {\it minimal singularities}.

\subsubsection{Generic singularities in the classical types}

The results of Kraft and Procesi for Lie algebras of classical type can be summarized as follows:  
an irreducible component of a generic singularity is 
either a simple surface singularity or a minimal singularity, up to smooth equivalence.
Moreover, when a generic singularity is not irreducible, then it is smoothly equivalent to 
a union of two simple surface singularities of type $A_{2k-1}$ meeting transversally in the singular point.  This is denoted $2A_{2k-1}$.
In more detail:

\begin{theorem} \cite{Kraft-Procesi:GLn}, \cite{Kraft-Procesi:classical} \label{classical_results}
Assume $\OC'$ is a minimal degeneration of $\OC$
in a simple complex Lie algebra of classical type.  Let $e \in \OC'$.
Then
\begin{itemize}
\item[(a)] If the codimension of $\OC'$ in $\cl$ is two, then $\SC_{\OC, e}$ is smoothly equivalent to 
a simple surface singularity of type $A_k$, $D_k$, or $2 A_{2k-1}$.
The last two singularities do not occur for $\sl_n(\CM)$, and 
the singularity $A_k$ for $k$
even does not occur in the classical Lie algebras besides $\sl_n(\CM)$.

\smallskip
\item[(b) ] If the codimension is greater than two, then $\SC_{\OC, e}$ is smoothly equivalent to $a_k, b_k, c_k$, or $d_k$.
The last three singularities do not occur for $\sl_n(\CM)$.
\end{itemize}
\end{theorem}

\subsection{The case of type $G_2$}  \label{G2_intro}
This case was studied by Levasseur-Smith \cite{Levasseur-Smith} and Kraft \cite{Kraft}.
Let $\0_8$ denote the $\widetilde{A}_1$ orbit and let $\0_6$ denote the minimal orbit.
Kraft showed that the closure of the subregular orbit has $A_1$ singularity along $\0_8$.
Levasseur-Smith showed that $\cl_8$ has non-normal locus equal to 
$\cl_6$ and that the natural map from the closure of the minimal orbit in $\so_7(\CM)$ to 
$\cl_8$ is the normalization map and is bijective. 
From these results it follows that the singularity of $\cl_8$ along $\0_{6}$ is non-normal with smooth normalization.
We describe this singularity in \S \ref{mexample} and show that its normalization is $\CM^2$. 

\subsection{Main results}

We now summarize the main results of the paper describing the classification of generic singularities in the exceptional Lie algebras. 
Here and in the sequel, we may write the degeneration $\0'$ of $\0$ as $(\0, \0')$, that is, with the larger orbit appearing first.
In this subsection $\0'$ is a minimal degeneration of $\0$ and $e \in \0'$. 
  
\subsubsection{Overview}
Most generic singularities are like those in the classical types:  
the irreducible components are either simple surface singularities or minimal 
singularities.   
But some new features occur in the exceptional groups.  
There is more complicated branching and several singularities occur which did not occur in the classical types.
Among the latter are three singularities whose irreducible components are not normal 
(one of these already occurs in $G_2$ as the singularity of $\widetilde{A}_1$ in the minimal orbit), and three additional singularities of dimension four.

A key observation is that all irreducible components of $\SC_{\0, e}$ are mutually isomorphic 
since the action of $C(\sg)$ is transitive on irreducible components  (\S \ref{Component group action}).   
This result is not true in general when $\0'$ is not a minimal degeneration of $\0$.

For most minimal degenerations we determine the isomorphism type of $\SC_{\0, e}$,
a stronger result than classifying the singularity up to smooth equivalence.
In ten of these cases, all in $E_8$ (\S \ref{remaining_surfaces}), 
we can only determine the isomorphism type of $\SC_{\0, e}$ up to normalization.  In 
the remaining four cases, $\SC_{\0, e}$ is determined only up to smooth equivalence (\S \ref{dim4_exceptions}).
It is possible to use the methods here to establish that Kraft and Procesi's results in Theorem \ref{classical_results}
hold up to algebraic isomorphism (rather than smooth equivalence), but we defer the details to a later paper.   

We also calculate the symmetry action on $\SC_{\0,e}$ induced from $A(e)$, as Slodowy did when $\0$ is the regular nilpotent orbit.  
This involves extending Slodowy's result on the splitting of $C(\sg)$ and introducing the notion of symmetry on a minimal singularity.
Again, it is possible to carry out this program for the classical groups, but we also defer the details to a later paper.

\subsubsection{Symmetry of a minimal singularity}   \label{def:symmetry_min_sing}

Let $\gg$ be a simple, simply-laced Lie algebra with Dynkin diagram $\Delta$.  
As in \S \ref{symmetry_nsl_case}, let $\tilde{K} \subset \Aut(\gg)$ be a subgroup of diagram automorphisms lifting a subgroup $K \subset \Aut(\Delta)$.
We call a pair $(\cl_{\text{min}}, \tilde{K})$, consisting of $\cl_{\text{min}}$ with the action of $\tilde{K}$, a symmetry of a minimal singularity.   
We write these pairs as $a^+_k$, $d^+_k (k \geq 4)$, $d^{++}_4$ (for the action of the full automorphism group), and $e^+_6$.   As in the surface cases,
$|K|=3$ in $D_4$ does not arise.

\subsubsection{Intrinsic symmetry action on a slice: general case}  \label{symmetry_action}

In \S \ref{C-splitting} it is shown 
that the splitting of $C(\sg)$ that Slodowy observed for the subregular orbit holds in general, with four exceptions.    
More precisely, there exists a subgroup $H \subset C(\sg)$ such that $C(\sg) \cong C(\sg)^{\circ} \rtimes H$.  So in particular 
$H \cong C(\sg)/C(\sg)^{\circ} \cong A(e)$. 
The choice of splitting is in general no longer unique up to conjugacy in $C(\sg)$, but if we choose $H$ to represent diagram automorphisms
of the semisimple part of $\cg(\sg)$, then the image of $H$ in $\Aut(\cg(\sg))$ is unique up to conjugacy in $\Aut(\cg(\sg))$.
The four exceptions to the splitting of $C(\sg)$ have $|A(e)|=2$, but the best possible result is that there exists 
$H \subset C(\sg)$, cyclic of order $4$, with $C(\sg) = C(\sg)^{\circ} \cdot H$ \cite[\S 3.4]{Sommers:Bala-Carter}.
  
Next, imitating \S \ref{symmetry_nsl_case}, we describe the action of $H$ on $\SC_{\0, e}$.
The four cases where $C(\sg)$ does not split give rise to the symmetries which include $A^+_2$ and $A^+_4$
(\S \ref{surface_symmetry}).
Three of these four cases (when $\0'$ has type $A_4+A_1$ in $E_7$ and $E_8$ or type $E_6(a_1)+A_1$ in $E_8$) 
are well-known:  
under the Springer correspondence, their Weyl group representations 
lead to unexpected phenomena (see, for example, \cite[pg. 373]{Carter}).   The phenomena observed here 
for these three orbits is directly 
related to the fact that $A(e) = \SG_2$ acts without fixed
 points on the irreducible components of the Springer fiber over $e$.   
 It is not clear why the fourth orbit
(of type $D_7(a_2)$ in $E_8$) appears in the same company as these three orbits.

\subsubsection{Additional singularities}  \label{unexpected_singularities}

In the Lie algebras of exceptional type, 
there are six varieties, arising as components of slice singularities, which are neither 
simple surface singularities nor minimal singularities.
The ten cases in type $E_8$ where we know the singularity only up to normalization would give further examples if they turned out to be non-normal.

\vspace{0.2cm}
\noindent{\bf Non-normal cases}.  

\vspace{0.2cm}
\noindent
{\it The variety $m$.}
Let $V(i)$ denote the irreducible representation of highest weight $i \in \mathbb Z_{\geq 0}$ of $\SL_2(\CM)$.
Consider the linear representation of $\SL_2(\CM)$ on $V = V(2) \oplus V(3)$.
Let $v_2 \in V(2)$ and $v_3 \in V(3)$ be highest weight vectors for a Borel subgroup of $\SL_2(\CM)$.
The variety $m$ is defined to be the closure in $V$ of the 
$\SL_2(\CM)$-orbit through $v = v_2+v_3$, a two-dimensional variety with an isolated singularity at zero.  
It is not normal, but has smooth normalization, equal to the affine plane ${\mathbb A}^2$.
This is an example of an $S$-variety (i.e., the closure of the orbit of a sum of highest weight vectors) studied in \cite{PV:S_variety}, 
where these properties are proved (see \S \ref{S_vars}).
The first case where $m$ appears is for the minimal degeneration $(\tilde{A}_1, A_1)$ in $G_2$.  
This singularity appears at least once in each exceptional Lie algebra, always for two non-special orbits which lie in the same special piece (see \S \ref{special_pieces}).

\vspace{0.2cm}
\noindent
{\it  The variety $m'$.}  This is a four-dimensional analogue of $m$, with $\SL_2(\CM)$ replaced by $\Sp_4(\CM)$.  
It is an $S$-variety \cite{PV:S_variety} with respect to the $\Sp_4(\CM)$-representation on $V = V(2\omega_1) \oplus V(3\omega_1)$ 
where $V(\omega_1)$ is the defining $4$-dimensional representation of $\Sp_4(\CM)$, so 
that $V(2\omega_1)$ is the adjoint representation. 
Let $v_2 \in V(2\omega_1)$ and $v_3 \in V(3\omega_1)$ be highest weight vectors for a Borel subgroup of $\Sp_4(\CM)$.
The variety $m'$ is defined to be the closure in $V$ of the 
$\Sp_4(\CM)$-orbit through $v = v_2\!+\!v_3$,
a four-dimensional variety with an isolated singularity at zero.  
It is not normal, but has smooth normalization, equal to ${\mathbb A}^4$ (see \S \ref{S_vars}).
The singularity $m'$ occurs exactly once, for the minimal degeneration 
$(A_3\!+\!2A_1, 2A_2\!+\!2A_1)$ in $E_8$.

\vspace{0.2cm}
\noindent
{\it  The variety $\mu$.}  The coordinate ring of the simple surface singularity $A_3$ is $R = {\mathbb C}[st,s^4,t^4]$, as a hypersurface in ${\mathbb C}^3$.
We define the variety $\mu$ by $\mu = \Spec R'$ where $R' = {\mathbb C}[(st)^2,(st)^3,s^4,t^4,s^5t,st^5]$.
This variety is non-normal and its normalization is isomorphic to $A_3$ via the inclusion of $R'$ in $R$. 
Using the methods of \S \ref{section:codimension2}, the normalization of $\SC_{\0, e}$ for 
$(D_7(a_1),E_8(b_6))$ in $E_8$ is shown to be isomorphic to $A_3$
with an order two symmetry arising from $A(e)$.
In \cite{Fu-Juteau-Levy-Sommers:GeomSP} we will show that 
$\SC_{\0, e}$ is smoothly equivalent to $\mu$.
The closure of $\0$ was known to be non-normal, but our result establishes that it is non-normal in codimension two.

\vspace{0.2cm}

\noindent{\bf Normal cases}.  
These three singularities are each of dimension four and normal.

\vspace{0.2cm}
\noindent
{\it The degeneration $(2A_2+A_1,A_2+2A_1)$ in $E_6$}.
Let $\zeta=e^{\frac{2\pi i}{3}}$ and let $\Gamma$ be 
the cyclic subgroup of $\Sp_4(\CM)$ of order three generated 
by the diagonal matrix with eigenvalues $\zeta$, $\zeta$, $\zeta^{-1}$, $\zeta^{-1}$.
Then ${\mathbb C}^4/\Gamma$ has an isolated singularity at $0$, and we denote this variety by $\tau$.
We show in \S \ref{hard_E6} that $\SC_{\0, e}$ is smoothly equivalent to $\tau$ for the minimal degeneration 
$(2A_2+A_1,A_2+2A_1)$ in $E_6$.

\vspace{0.2cm}
\noindent
{\it The degeneration $(A_4+A_1,A_3+A_2+A_1)$ in $E_7$}.
Let ${\mathfrak S}_2$ be the cyclic group of order two acting on $\sl_3(\CM)$ via an outer involution.  
All such involutions are conjugate in $\Aut(\sl_3(\CM))$.
The quotient $a_2/{\mathfrak S}_2$ 
has an isolated singularity at $0$ 
since there are no minimal nilpotent elements in $\sl_3(\CM)$ which are fixed by an outer involution.
We will prove in \cite{Fu-Juteau-Levy-Sommers:GeomSP} that 
$\SC_{\0, e}$ is smoothly equivalent to $a_2/{\mathfrak S}_2$ for the minimal degeneration 
$(A_4+A_1,A_3+A_2+A_1)$ in $E_7$.

\vspace{0.2cm}
\noindent
{\it The degeneration $(A_4+A_3,A_4+A_2+A_1)$ in $E_8$}.
Let $\Delta$ be a dihedral group of order 10, acting on ${\mathbb C}^4$ via the sum of the reflection representation and its dual.
Then it turns out that the blow-up of ${\mathbb C}^4/\Delta$ at its singular locus 
has an isolated singularity at a point lying over $0$.  We denote this blow-up by $\chi$.
We show in \S \ref{hard_E8} that $\SC_{\0, e}$ is smoothly equivalent to $\chi$
for the minimal degeneration $(A_4+A_3,A_4+A_2+A_1)$ in $E_8$.
\footnote{In private communication with the authors, Bellamy has pointed out that it can be deduced from his work \cite{bellamy} that the symplectic quotient of $\CM^4$ by a dihedral group of order $4n+2$ has a unique $\QM$-factorial terminalization.  Since $\0$ is a rigid orbit, the singularity $\chi$ is also $\QM$-factorial terminal (see \S \ref{q_fact_term}). 
Hence $\chi$ can be identified with the unique $\QM$-factorial terminalization of $\CM^4/\Delta$.}

\subsubsection{Statement of the main theorem}
In our main theorem we classify the generic singularities of nilpotent orbit closures in a simple Lie algebra of exceptional type.
The graphs at the end of the paper list the precise results.

\begin{theorem}  \label{main_theorem}
Let $\0'$ be a minimal degeneration of $\0$ in a simple Lie algebra of exceptional type.  Let $e \in \0'$.  Taking into consideration the 
intrinsic symmetry of $A(e)$, we have 
\begin{itemize}

\item[(a)] If the codimension of $\OC'$ in $\ov{\OC}$ is two, then, with one exception, $\SC_{\0, e}$ 
is isomorphic either to a simple surface singularity of type $A-G$ or to one of the following
$$A_2^+, A_4^+, 2A_1,  3A_1, 3C_2, 3C_3, 3C_5, 4G_2,5G_2,10G_2, \;\mbox{or}\;\; m,$$
up to normalization for ten cases in $E_8$.   
Here, $kX_n$ denotes $k$ copies of $X_n$ meeting pairwise transversally at the common singular point.
In the one remaining case, the singularity is smoothly equivalent to $\mu$.

\smallskip
\item[(b)] If the codimension is greater than two, then, with three exceptions, $\SC_{\0, e}$ is isomorphic
either to  a minimal singularity of type $a-g$ or to one of the following types:
$$a_2^+,a_3^+,a_4^+,a_5^+, 2a_2,d_4^{++}, e_6^+, 2g_2, \;\mbox{or}\;\; m',$$
where the branched cases $2a_2$ and $2g_2$ denote two minimal singularities meeting transversally at the common singular point.
The singularities for the three remaining cases are smoothly equivalent to $\tau, a_2/{\mathfrak S}_2, \text{and} \ \chi,$ respectively.
\end{itemize}
\end{theorem}

\noindent In the statement of the theorem, we have recorded the induced symmetry of $A(e)$ relative to 
the stabilizer in $A(e)$ of an irreducible component of $\SC_{\0, e}$.  See \S \ref{full_intrinsic_action} for a complete statement
of the intrinsic symmetry action.

\subsubsection{Brief description of methods}

The methods in \S \ref{tools_explicit} are relevant for cases when $\slice$ has a 
dense $C(\sg)$-orbit and are motivated by arguments in \cite{Kraft-Procesi:GLn}.
For the higher codimension minimal degenerations (except the three which are normal of codimension four)
and the codimension two minimal degenerations where the singularity is $kA_1$ or $m$, 
we show that $\slice$ has a dense $C(\sg)$-orbit. 
This allows us to show that the irreducible components of $\slice$ are $S$-varieties for $C(\sg)^\circ$ (which are permuted transitively by $C(\sg)$), and we determine their isomorphism type.
Proposition \ref{dim4+_1} contains precise information about the connection
between $C(\sg)$ and $\slice$ in these cases.

The methods in \S \ref{section:codimension2} are applicable to the surface cases.
The idea is to use the fact that the normalization of a transverse slice is a simple surface singularity
and then obtain a minimal resolution of the singularity by restricting the
$\qit$-factorial terminalizations of the nilpotent orbit closure
to the transverse slice. Then we can apply a formula of
Borho-MacPherson to compute the number of projective lines in the
exceptional fiber and the action of $A(e)$ on the projective lines.
Proposition \ref{prop:dim2} summarizes the surface cases.

These two methods, as summarized in Propositions \ref{prop:dim2} and \ref{dim4+_1},
are sufficient to handle all the cases in the main theorem 
except when $\mu$, $\chi$, $\tau$, or $a_2/{\mathfrak S}_2$ occur. 
The two cases where $\chi$ or $\tau$ occur are dealt with separately in \S \ref{dim4_exceptions}.
The two cases where $\mu$ or $a_2/{\mathfrak S}_2$ occur are deferred to subsequent papers.

The determination of the symmetry action is given in \S \ref{section:split_and_intrinsic},
and the calculations supporting 
Propositions \ref{prop:dim2} and \ref{dim4+_1} are given in \S \ref{results_F4}, \S \ref{sec:E6}, \S \ref{sec:E7}, \S \ref{e8}.

\subsection{Some consequences}
\subsubsection{Isolated symplectic singularities coming from nilpotent orbits} 
Examples of isolated symplectic singularities include  $\overline{\0}_{\text{min}}$ and quotient singularities $\mathbb{C}^{2n}/\Gamma$, where
$\Gamma \subset {\Sp}_{2n}(\CM)$ is a finite subgroup acting freely on ${\mathbb C}^{2n}\setminus\{ 0\}$.
It was established in \cite{Beauville} that an isolated symplectic singularity with smooth projective tangent cone is locally analytically isomorphic to some $\overline{\0}_{\text{min}}$.
It turns out that all of the isolated symplectic singularities we obtain, with one exception, are finite quotients of $\overline\0_{\text{min}}$ or ${\mathbb C}^{2n}$.  It seems very likely that the singularity $\chi$ described above is not equivalent to a singularity of this form.%

Another byproduct is the discovery of examples of symplectic
contractions to an affine variety whose generic positive-dimensional fiber is of type
$A_2$ and with a non-trivial monodromy action.
These examples correspond to minimal degenerations with singularity $A_2^+$.   The orbits $\0$ in these cases have closures 
which admit a generalized Springer resolution.   Examples include the even orbits $A_4+A_2$ in $E_7$ and $E_8(b_6)$ in $E_8$.  
In \cite{Wierzba}, a symplectic contraction to a projective variety of the same type is constructed.
 As far as we know,  our examples are the first affine examples that have been constructed.
 This disproves Conjecture 4.2 in \cite{Andreatta-Wisniewski}.

\subsubsection{Special pieces}\label{special_pieces}

For a special nilpotent orbit $\0$, the {\em special
piece} $\P(\0)$ containing $\0$ is the union of all nilpotent orbits 
$\0' \subset \cl$ 
which are not contained in $\cl_1$ for any special
nilpotent orbit $\0_1$ with $\0_1 \subset \cl$ and $\0_1 \neq \0$.
This is a locally-closed subvariety of $\cl$ and 
it is rationally smooth (see \cite{Lusztig:unipotent} and the references therein).
To explain rational smoothness  geometrically, Lusztig conjectured in
\cite{Lusztig:unipotent} that every special piece is a finite
quotient of a smooth variety.  This conjecture is known for
classical types by \cite{Kraft-Procesi:special}, but for
exceptional types it is still open.

Each special piece contains a unique minimal orbit under the closure ordering.  
Motivated by the aforementioned  conjecture of Lusztig, we studied
the transverse slice of $\P(\0)$ to this minimal orbit. We
shall prove in \cite{Fu-Juteau-Levy-Sommers:GeomSP} the following:
\begin{theorem}\label{special_slice}
Consider a special piece $\P(\0)$ in any simple Lie algebra. Then a nilpotent Slodowy slice in $\P(\0)$
to the minimal orbit in $\P(\0)$ is isomorphic to
$$
(\mathfrak{h}_n \oplus \mathfrak{h}_n^*)^k/\mathfrak{S}_{n+1}
$$
where $\mathfrak{h}_n$ is the $n$-dimensional reflection representation
of the symmetric group $\mathfrak{S}_{n+1}$ and
$k$ and $n$ are uniquely determined integers.
\end{theorem}

This theorem also includes the Lie algebras of classical types where $n=1$, but $k$ can be arbitrarily large.
In the exceptional types Theorem \ref{main_theorem} covers the cases where 
$\P(\0)$ consists of two orbits, in which case $n=k=1$ (that is, the slice is isomorphic to the $A_1$ simple surface singularity).  
This leaves only those special pieces containing more than two orbits.   Some of these remaining cases can be
handled quickly with the same techniques, but others require more difficult calculations.

\subsubsection{Normality of nilpotent orbit closures}

By work of Kraft and Procesi \cite{Kraft-Procesi:classical}, together with the
remaining cases covered in \cite{Sommers:D}, 
in classical Lie algebras the failure of $\cl$ to be normal 
is explained by branching for a minimal degeneration, and then only with two branches and in codimension two.
In exceptional Lie algebras, the question of which nilpotent orbit closures 
are normal has not been completely solved in $E_7$ or $E_8$, but 
in \cite[Section 7.8]{Broer:Decomposition} a list of non-normal nilpotent orbit closures is given, which is expected to be the complete list.

Our analysis sheds some new light on the normality question.   
The occurrence of $m$, $m'$, and $\mu$ at a minimal degeneration of $\0$
gives a new geometric explanation for why $\cl$ is not normal.  Previously the only geometric explanation for the failure of normality
was branching (see \cite{Beynon-Spaltenstein}) 
and the appearance of the non-normal singularity in the closure of the $\tilde{A}_1$ orbit in $G_2$, which was known to be unibranch
(see \cite{Kraft}).

We also establish:  (1) for many $\cl$ known to be non-normal that $\cl$ is normal at points in some minimal degeneration; 
and (2) for many $\cl$ that are expected to be normal that $\cl$ is indeed normal at points in all of its minimal degenerations.   
So we are able to make a contribution to determining the non-normal locus of $\cl$.
Examples of the above phenomena are given starting in \S \ref{F4_surface}.
Along these same lines, we also note that a consequence of Theorem \ref{special_slice} 
is that the special pieces are normal, a question studied by Achar and Sage in \cite{Achar-Sage}.

\subsubsection{Duality}\label{dualitysubsection}

An intriguing result from \cite{Kraft-Procesi:GLn} for $\gg
= \sl_n(\CM)$ is the following: a simple surface singularity of type $A_k$  is always
interchanged with a minimal singularity of type $a_k$ under the
order-reversing involution on the set of nilpotent orbits in $\gg$ given by transposition of partitions.

This result leads to a generalization in the other Lie algebras, both classical and exceptional, 
by restricting to the set of special nilpotent orbits, which are
reversed under the Lusztig-Spaltenstein involution.
For a minimal degeneration of one special orbit to another,
in most cases a simple surface singularity is interchanged with a singularity corresponding
to the closure of the minimal {\it special} nilpotent orbit of dual type.
There are a number of complicating factors outside of $\sl_n(\CM)$, 
related to Lusztig's canonical quotient and the existence of multiple branches.
The duality can also be formulated as one from special orbits in $\gg$ to those in $^L \gg$, the more natural 
setting for Lusztig-Spaltenstein duality.

Numerical evidence for such a duality was discovered by Lusztig in the classical groups using the tables in \cite{Kraft-Procesi:classical}.
The duality is already hinted at by Slodowy's result for the regular nilpotent orbit in \S \ref{symmetry_nsl_case}, which requires passing from $\gg$ to $^L \gg$.
In a subsequent article \cite{FJLS:Duality} we will give a more complete account of the phenomenon of duality for degenerations between special orbits.

\subsection{Notation} 

$G$ will be a connected, simple algebraic group of adjoint type over the complex numbers $\CM$
with Lie algebra $\gg$,  and $\0$ and $\0'$ will be nilpotent $\Ad G$-orbits in $\gg$.
We use the notation in \cite[p. 401-407]{Carter} to refer to nilpotent orbits.
For $x\in\gg$, $\0_x$ refers to the orbit $\Ad G(x)$, also written $G \cdot x$.
For $x\in G$ or $\gg$ we denote by $G^x$ (resp. $\gg^x$) the centralizer of $x$ in $G$ (resp. $\gg$).
Similar notation applies to other algebraic groups which arise, including as subgroups of $G$.
For a subalgebra $\zg \subset \gg$, we denote by $C(\zg)$ its centralizer in $G$ and 
$\cg(\zg)$ its centralizer in $\gg$.

Generally, $e$ is a nilpotent element in an $\sl_2(\CM)$-subalgebra $\sg$ with standard basis $e, h, f$.  
If $e_0 \in \cg(\sg)$ is a nilpotent element, we use
$\sg_0$ for an $\sl_2(\CM)$-subalgebra in $\cg(\sg)$ with standard basis $e_0, h_0, f_0$.
Usually $\0'$ is a nilpotent orbit in $\cl$ with $\0' \neq \0$ and $e \in \0'$.
We write $(\0, \0')$ for such a pair of nilpotent orbits. 
Often, but not always, $\0'$ is a minimal degeneration of $\0$.
The nilpotent Slodowy slice $\cl \cap (e + \gg^f)$ is denoted $\slice$.

The field of fractions of an integral domain $A$ will be denoted $\Frac(A)$.
The symmetric group on $n$ letters is denoted $\mathfrak{S}_n$.
Where we refer to explicit elements of $\gg$, we use the structure constants in GAP \cite{GAP4}.

\subsection{Acknowledgements}
The authors thank 
Anthony Henderson, Thierry Levasseur, Vladimir Popov, Miles Reid, 
Dmitriy Rumynin, Gwyn Bellamy and Michel Brion for helpful comments and/or conversations.
B. Fu was supported
by NSFC 11225106 and 11321101 and the KIAS Scholar Program.
D. Juteau was supported by the ANR grant VARGEN (ANR-13-BS01-0001-01).
P. Levy was supported in part by Engineering and Physical Sciences Research Council grant EP/K022997/1.
E. Sommers was supported by NSA grant H98230-11-1-0173 and 
through a National Science Foundation Independent Research and Development plan.
Computer calculations were carried out in 
GAP (including the SLA package \cite{deGraaf}) and Python.
Research visits were also supported by the 
CNRS
and the AMSS of Chinese Academy of Sciences. 

\section{Transverse slices} \label{section:slices}

\subsection{Smooth equivalence} \label{smooth_equivalence}

To study singularities it is useful to introduce the notion of smooth equivalence.
Given two varieties $X$ and $Y$ and two points $x \in X$ and $y \in Y$, 
the singularity of $X$ at $x$ is {\it smoothly equivalent} to the singularity of $Y$ at $y$ if there exists a variety $Z$, a point $z\in Z$ and morphisms 
$$\varphi:Z\rightarrow X \text{ and } \psi:Z\rightarrow Y$$ 
which are smooth at $z$ and such that $\varphi(z)=x$ and $\psi(z)=y$ (see \cite[1.7]{Hesselink}).
This defines an equivalence relation on pointed varieties $(X,x)$ and the equivalence class of $(X,x)$ will be denoted $\Sing(X,x)$.  As in \cite[\S 2.1]{Kraft-Procesi:GLn}, two singularities $(X,x)$ and $(Y,y)$ with $\dim Y=\dim X+r$ are equivalent if and only if $(X\times{\mathbb A}^r,(x,0))$ is locally analytically isomorphic to $(Y,y)$.

Let $\0'$ and $\0$ be nilpotent orbits in ${\mathfrak g}$ with $\0' \subset \cl$.   Let $e \in \0'$.
The local geometry of $\cl$ at $e$ is captured by 
the equivalence class of $(\cl,e)$ under smooth equivalence.  The equivalence class of the singularity $(\cl,e)$ 
will be denoted $\Sing(\0, \0')$ since the equivalence class is independent of the choice of element in $\0' = \0_e$.


\subsection{Transverse slices}  \label{transverse_slices}

The main tool in studying $\Sing(\0, \0')$ 
is the transverse slice.   Both \cite[III.5.1]{Slodowy:book} and \cite[\S 12]{Kraft-Procesi:classical} are references for what follows.

Let $X$ be a variety on which $G$ acts, and let $x\in X$.
A {\it transverse slice} in $X$ to $G\cdot x$ at $x$ is a locally closed subvariety $S$ of $X$ with $x\in S$ such that the morphism 
$$G\times S\rightarrow X, (g,s)\mapsto g\cdot s$$ is smooth at $(1,x)$ and such that the dimension of $S$ is minimal subject to these requirements.
It is immediate  that $\Sing(X,x)=\Sing(S,x)$.  
If $X$ is a vector space then it is easy to construct such a transverse slice as $x+{\mathfrak u}$ where ${\mathfrak u}$ is a vector space complement to $T_x(G\cdot x)= [\gg, x]$ in $X$.
More generally, this also suffices to construct a transverse slice to a $G$-stable subvariety $Y\subset X$ containing $x$ by taking the intersection $(x+{\mathfrak u})\cap Y$ (\cite[III.5.1, Lemma 2]{Slodowy:book}).
In such a case $\codim_{Y}(G\cdot x ) = \dim_x ((x+{\mathfrak u})\cap Y)$.

These observations are especially helpful for nilpotent orbits in the adjoint representation, 
where there is a natural choice of transverse slice.
As before, pick $e \in \0'$.  Then there exists $h, f \in \gg$ 
so that $\{ e,h,f\}\subset\gg$ is an $\mathfrak{sl}_2$-triple.
Then by the representation theory of $\mathfrak{sl}_2(\CM)$, we have $[e, \gg]\oplus\gg^f=\gg$.
The affine space 
$$\SC_e = e + \gg^f$$ is a transverse slice of $\gg$ at $e$, called the {\it Slodowy slice}.
The variety
$$\SC_{\0,e}: = \SC_e \cap \cl$$
is then a transverse slice of $\cl$ to $\0'$ at the point $e$, 
which we call a nilpotent Slodowy slice.
In this setting
\begin{equation}  \label{dim=codim}
\codim_{\cl} (\OC') = \dim \SC_{\0,e}.
\end{equation}
Since any two $\sl_2$-triples for $e$ are conjugate by an element of $G^e$, 
the isomorphism type of $\SC_{\0,e}$ is independent of the choice of $\sl_2$-triple. 
Moreover, the isomorphism type of $\SC_{\0,e}$ is independent of the choice of $e \in \0'$.
By focusing on $\slice$ we reduce the study of $\Sing(\0, \0')$ to the study of the singularity of $\SC_{\0, e}$ at $e$.
In fact, most of our results will be concerned with determining the isomorphism type  of  $\SC_{\0, e}$.

\subsection{Group actions on $\slice$}  \label{group action}

An important feature of the transverse slice $\SC_{\0,e}$ is that it carries the action of two commuting algebraic groups, which both fix $e$. 
Let $\sg$ be the subalgebra spanned by $\{ e,h,f\}$ and $C(\sg)$ the centralizer of $\sg$ in $G$.  
Then $C(\sg)$ is a maximal reductive subgroup of $G^e$ and $C(\sg)$ acts on $\SC_{\0,e}$, fixing $e$.

The second group which acts is ${\mathbb C}^*$.
Since $[h,f] = -2f$, $\ad h$ preserves the subspace $\gg^f$ and by 
$\mathfrak{sl}_2$-theory all of its eigenvalues are nonpositive integers.
Set $$\gg^f(i)=\{ x\in\gg^f : [h,x]=ix\}$$ for $i \leq 0$.
The special case $\gg^f(0)$ is simply $\cg(\sg)$, the centralizer of $\sg$ in $\gg$, which coincides with  $\Lie(C(\sg))$.

Define $\phi: SL_2(\CM)  \to G$ such that the image of $d\phi$ is equal to $\sg$,  with 
$d\phi \left( \begin{smallmatrix}
 0 & 1 \\ 0 & 0 \end{smallmatrix} \right) = e$ and 
 $d\phi \left( \begin{smallmatrix}
 1 & 0 \\ 0 & -1 \end{smallmatrix} \right) = h$.  Set $\gamma(t) = \phi \left( \begin{smallmatrix}
 t^{-1} & 0 \\ 0 & t \end{smallmatrix} \right)$ for $t \in {\mathbb C}^*$.
On the one hand, $\Ad \gamma(t)$ preserves $\cl$ and so does the scalar action of ${\CM}^*$ on $\gg$ since $\cl$ is
conical in $\gg$. 
On the other hand, for $x_i\in\gg^f(-i)$ and $t \in {\mathbb C}^*$,
$$\Ad \gamma(t)( e+x_0+\ldots + x_m) =  t^{-2}e+x_0+t x_1+\ldots+t^{m}x_m.$$ 
Composing this action with the scalar action of $t^2$ on $\gg$, 
gives an action of $t \in \CM^*$  on $e + \gg^f$ by 
\begin{equation} \label{C-star}
t \cdot (e+x_0+x_1+\ldots)=e+t^2 x_0+t^3x_1+\ldots,
\end{equation}
which preserves
$\SC_{\0,e} = \cl \cap \SC_e.$   The $\CM^*$-action fixes $e$ and commutes with the action of $C(\sg)$,
since $C(\sg)$ commutes with $\ad h$ and so preserves each $\gg^f(i)$.
Thus $C(\sg) \times \CM^*$ acts on $\slice$.


 \subsection{Branching and component group action}   \label{Component group action}

The $C(\sg) \times \CM^*$-action on $\slice$ has consequences for the irreducible components of $\slice$.  

An irreducible variety $X$ is {\it unibranch} at $x$ if the normalization $\pi: (\tilde{X},\tilde{x})\rightarrow (X,x)$ of $(X,x)$ is locally a homeomorphism at $x$.   Since the ${\mathbb C}^*$-action on $\SC_{\0,e}$ in \eqref{C-star} is attracting to $e$, 
$\slice$ is connected and its irreducible components are unibranch at $e$.
Consequently the number of irreducible components of $\SC_{\0,e}$ is equal to the 
number of branches of $\cl$ at $e$.  The latter 
can be determined from the tables of Green functions in \cite{Beynon-Spaltenstein,Sho}, as discussed in
\cite[5(E)--(F)]{Beynon-Spaltenstein}.

The identity component $C(\sg)^\circ$ of $C(\sg)$, being connected, preserves each irreducible component of $\SC_{\0,e}$, hence
there is a natural action of $C(\sg)/C(\sg)^\circ$ on the irreducible components of $\SC_{\0,e}$.
The finite group $C(\sg)/C(\sg)^\circ$ is isomorphic to the component group $A(e):= G^e/ (G^e)^\circ$ of $G^e$ via
$C(\sg) \hookrightarrow  G^e \twoheadrightarrow G^e/ (G^e)^\circ$.
Since any two $\mathfrak{sl}_2$-triples containing $e$ are conjugate by an element of $(G^e)^\circ$, this gives a well-defined
action of $A(e)$ on the set of irreducible components of $\SC_{\0,e}$.
Moreover, as noted in {\it op. cit.}, the permutation representation of
$A(e)$ on the branches of $\cl$ at $e$, and hence on the irreducible components of $\SC_{\0,e}$, can be computed.
For a minimal degeneration, the situation is particularly nice.
We observe by looking at the tables in \cite{Beynon-Spaltenstein,Sho} that
\begin{proposition} \label{transitive_components}
When $\0'$ is a minimal degeneration of $\0$ in an exceptional Lie algebra, 
the action of $A(e)$ on the set of irreducible components of $\SC_{\0,e}$ is transitive.  
In particular, the irreducible components of $S_{\0,e}$ are mutually isomorphic.
\end{proposition}

The proposition also holds in the classical types, which can be deduced using the results in \cite{Kraft-Procesi:classical}.
In \S \ref{full_intrinsic_action} we will discuss the full symmetry action on $\slice$ induced from  $A(e)$.

\section{Statement of the key propositions} \label{section:main_results} 


In this section we state the key propositions which underlie Theorem \ref{main_theorem}.  The propositions give more precise information about $\slice$.   
Throughout this section $\0'$ is always a minimal degeneration of $\0$.

\subsection{Surface cases}  \label{summary:surface_cases}

The case of a minimal degeneration of codimension two is summarized by the following proposition.
\begin{proposition} \label{prop:dim2}
Let  $\0'$  be a minimal degeneration of $\0$ of codimension $2$.  
Then there exists a finite subgroup $\Gamma \subset SL_2(\CM)$ such that 
the normalization $\widetilde{\SC}_{\0,e}$ of $\slice$ 
is isomorphic to a disjoint union of $k$ copies of $X$ where 
$X = \CM^2/\Gamma.$
\end{proposition}
This is proved in \S \ref{section:codimension2} where techniques for determining $\Gamma$ and $k$ are given.
For most cases in Proposition \ref{prop:dim2} we can show that the irreducible components of $\slice$ are normal either by knowing that $\cl$ is normal, by using Lemma \ref{singequi} to move to a smaller subalgebra where the slice is known to be normal, or by doing an explicit computation using Lemma \ref{kp_prop}.
In the surface case, we found only two ways that an irreducible component of $\slice$ fails to be normal:
\begin{itemize}
\item When $\Gamma = 1$, we show below that $\slice \cong m$ (\S \ref{unexpected_singularities}).   This happens for several different minimal degenerations.
\item When $(\0, \0') = (D_7(a_1), E_8(b_6))$, we have $\Gamma \cong {\mathbb Z}/4$, but $\slice$ is not normal.  Instead, $\slice$
is smoothly equivalent to $\mu$ (\S \ref{unexpected_singularities}).
\end{itemize}
A handful of cases in $E_8$ are determined only up to normalization (see \S \ref{remaining_surfaces}).

\begin{remark} \label{C_star_compatibility}
The isomorphism in Proposition \ref{prop:dim2} is compatible with the natural $\CM^*$-actions on both sides. 
On $\widetilde{\SC}_{\0,e}$ the $\CM^*$-action 
is the one induced from \S \ref{group action}; on $\CM^2/\Gamma$ it is the one coming from the central torus in $\GL_2(\CM)$.   
This follows from Proposition \ref{P:2-dim}.
\end{remark}

\subsection{Cases with a dense $C(\sg)$-orbit}

Next we consider all the cases of codimension $4$ or greater other than the three normal cases in \S \ref{unexpected_singularities}.
Together with the surface cases where $|\Gamma| =1$ or $2$,  these $\slice$ have a dense orbit for the action of $C(\sg)$.
More is true, their irreducible components are examples of $S$-varieties \cite{PV:S_variety}.

\subsubsection{$S$-varieties} \label{S_vars}

Let  $\{\Lambda_1, \dots \Lambda_r\}$ be a set of dominant weights for a maximal torus in a fixed Borel subgroup of $C(\sg)^\circ$.
It will also be convenient to view the $\Lambda_i$'s as weights for the Lie algebra of this maximal torus.
Let $V(\Lambda_i)$ be the irreducible representation of $C(\sg)^\circ$ of highest weight $\Lambda_i$ 
and $v_i \in V(\Lambda_i)$ a non-zero highest weight vector.
Then the $S$-variety $X(\Lambda_1, \dots \Lambda_r)$ is 
defined to be the closure in 
$V:=V(\Lambda_1) \oplus \dots \oplus V(\Lambda_r)$ of the $C(\sg)^\circ$-orbit through $v:= (v_1, \dots, v_r)$.
In \cite{PV:S_variety} it is shown that $S$-varieties are exactly the varieties which carry a dense $C(\sg)^\circ$-orbit and every point in this dense orbit has stabilizer containing a maximal unipotent subgroup of $C(\sg)^\circ$.

In all the situations encountered in this paper, we find $\Lambda_i = b_i \lambda$ for each $i$, where $b_i \in \mathbb N$ and $\lambda$ is a fixed dominant weight.
In such cases Theorems 6 (and its Corollary), 8 and 10 from \cite{PV:S_variety} reduce to the following, respectively:
(1) the complement of $C(\sg)^\circ \cdot v$ in $X$ is the origin in $V$;
(2) the determining invariant of the $C(\sg)^\circ$-isomorphism type of $X:= X(b_1 \lambda, \dots, b_r \lambda)$ is the monoid in 
${\mathbb N}$ generated by $b_1, \ldots, b_r$; and (3) the normalization of $X$ is $X(d \lambda)$, where $d$ is the greatest common divisor of $b_1, \dots, b_r$.
More is true in (2):  if $b_1, \ldots, b_j$ generate the same monoid as  $b_1, \ldots, b_r$ for $j <r$, then the
natural projection from $X$ to $X(b_1 \lambda, \dots, b_j \lambda)$ is an isomorphism.
We will assume that $b_1 \leq b_2 \leq \dots \leq b_r$.
If $V$ factors through $Z \subset C(\sg)^\circ$ with Lie algebra $\sl_2(\CM)$, then we sometimes write
$X(b_1, b_2, \dots)$ instead of $X(b_1 \lambda, b_2 \lambda, \dots)$ where $\lambda$ is the fundamental weight for $\sl_2(\CM)$.

\subsubsection{}  \label{natural_proj} 

Let $\pi_{0}: \slice \to \cg(\sg)$ be the restriction of the 
$C(\sg) \times {\mathbb C}^*$-equivariant linear projection of $\SC_{e}$ onto $\cg(\sg) = \gg^f(0)$.
Let $\pi_{0,1}: \slice \to \gg^f(0) \oplus \gg^f(-1)$ be 
the restriction of the $C(\sg) \times {\mathbb C}^*$-equivariant linear projection of $\SC_{e}$ onto $ \gg^f(0) \oplus \gg^f(-1)$.   
Recall $v \in \cg(\sg)$ belongs to a minimal nilpotent $C(\sg)^\circ$-orbit 
if and only if $v$ is a highest weight vector (relative to a Borel subgroup of $C(\sg)$)
in a simple summand of $\cg(\sg)$.
The proof of the next proposition is given in \S \ref{section:dim4}.  

\begin{proposition}  \label{dim4+_1} 
Let $\0'$  be a minimal degeneration of $\0$ of codimension at least four (other than the three normal cases in \S \ref{unexpected_singularities}) or of codimension two with  $|\Gamma| =1$ or $2$.

Then there exists $J = \{ i_1, \dots, i_r \} \subset \mathbb N$ so that for each $i \in J$ there exists
a highest weight vector $x_i \in \gg^f(-i)$ for the action of $C(\sg)^\circ$, 
and there exists $x_0 \in \cg(\sg)$ minimal nilpotent, such that
\begin{itemize}
\item $e + x_0 + \sum_{i \in J}  x_i \in \slice$,
\item if the weight of $x_0$ is given by $\lambda$, then the weight of $x_i$ equals $(\tfrac{i}{2}+1) \lambda$,
\item one of the irreducible components of $\slice$ is $e + X$, where $X$ is
the corresponding $S$-variety 
$$X:= X\left( \lambda,  (\tfrac{i_1}{2}\!+\!1) \lambda, (\tfrac{i_2}{2}\!+\!1) \lambda, \dots,  (\tfrac{i_r}{2}\!+\!1) \lambda \right) \subset \gg^f$$
for $C(\sg)^\circ$,
through the vector $v := x_0 + \sum_{i \in J}  x_i$,
\item the irreducible components of $\slice$ are permuted transitively by $C(\sg)$, whence $\slice = \ov{C(\sg)\cdot(e+v)}$.
\end{itemize}
Moreover, there are two cases:
\begin{enumerate}
\item All $i \in J$ are even.  Then $\pi_0$ induces an isomorphism of each irreducible component of $\slice$ with $X(\lambda)$. Furthermore, $X(\lambda)$ is a minimal singularity, being isomorphic 
to the closure of the $C(\sg)^{\circ}$-orbit 
through the minimal nilpotent element 
$x_0 \in V(\lambda) \subset \cg(\sg)$.  Hence each component is a normal variety.
\item We have $1 \in J$.  
This case occurs only if $\cg(\sg)$ contains a simple factor $\zg$ of type $\sl_2(\CM)$ or $\sp_4(\CM)$ and 
$\zg = V(\lambda) $. 
Note that $\lambda = 2 \omega$ where $V(\omega)$ is the defining representation of $\zg$.  
Then  $\pi_{0,1}$ gives an isomorphism
$\slice \cong X(\lambda, \tfrac{3}{2} \lambda) = X(2\omega, 3\omega)$.  
In the $\zg = \sl_2(\CM)$ case, $\slice \cong m$ and in the $\zg = \sp_4(\CM)$ case, $\slice \cong m'$.   In both cases $\slice$ is irreducible and non-normal. 
\end{enumerate}

\end{proposition}

\begin{remark}  \label{which_J}
In case (1) of Proposition \ref{dim4+_1}, 
when the codimension is at least four, we find that $J = \emptyset$ except for 
the two minimal degenerations ending in the orbit $D_4(a_1)+A_2$ in $E_8$, 
where $J = \{ 2 \}$ (see \S \ref{dim_4_case_2}). 
On the other hand, when the codimension is two in case (1), then $\slice \cong kA_1$.
If $k >1$, then always $J = \emptyset$.  
If $k=1$, then $J$ can be either $\emptyset, \{2\},$ or $\{2, 4\}$.

In case (2) of Proposition \ref{dim4+_1},
the possibilities for $J$ that arise are $\{ 1 \}, \{1, 2 \},$ or $\{1, 2, 3\}$; however, 
the singularity $m'$ only occurs for $(A_3+2A_1, 2A_2+2A_1)$ in $E_8$,
where $J = \{ 1 \}$. 
\end{remark}

\begin{remark}
From the previous remark, we see that $\slice$ is not irreducible only when $J = \emptyset$. 
In that case, each irreducible component of $\slice$ corresponds to the minimal orbit closure in a unique simple summand of $\cg(\sg)$.
The direct sum of these summands is an irreducible representation for $C(\sg)$
and the summands are permuted transitively by $C(\sg)/C(\sg)^{\circ}$ as expected from Proposition \ref{transitive_components}.
The first non-irreducible example covered by the Proposition is $(C_3(a_1), B_2)$, see Table \ref{table1_f4}.
\end{remark}

\begin{remark} \label{not_everyone_contributes}
In each case of Proposition \ref{dim4+_1}, the map $\pi_0$ is surjective onto the closure 
of a minimal nilpotent $C(\sg)$-orbit in $\cg(\sg)$ (namely, the one through $x_0$).
It is not true, however, that every such minimal nilpotent orbit will arise in this way.
For example, when  $e \in \0' = 3A_1$ in $E_6$, the centralizer $\cg(\sg)$ has type $A_2+A_1$.   
If $e_0$ belongs to the minimal orbit in the summand of type $A_2$
and $e'_0$ belongs to the minimal orbit in the summand of type $A_1$, 
then $\0' \subset \cl_{e+e'_{o}} \subset \cl_{e+e_{o}}$
and $\0'$ is a minimal degeneration of $\0_{e+e'_{o}}$ and of no other orbit.  So the minimal orbit in the $A_2$ summand 
is not in the image of $\pi_0$ for any minimal degeneration ending in $\0'$.  
\end{remark}

\section{Tools for establishing Proposition \ref{dim4+_1}}  \label{tools_explicit}

In this section we give some tools for identifying $\slice$, which can often be applied even when the degeneration is not minimal. Therefore we do not in general assume degenerations are minimal in this section.  At the end of the section in \S \ref{section:dim4}, we prove Proposition \ref{dim4+_1}.
We keep the notation that $\0$ and $\0'$ are nilpotent $G$-orbits in $\gg$ with $\0' \subset \cl$ and $e \in \0'$.   
Unless specified otherwise, $\0'$ is not assumed to be a minimal degeneration of $\0$.

\subsection{Some reduction lemmas}  

The first two lemmas give frameworks to relate $\slice$ to a variety attached to a proper subalgebra of $\gg$.
Both lemmas are variants of \cite[Cor 13.3]{Kraft-Procesi:classical}. 

\subsubsection{Passing to a reductive subalgebra} \label{passing_to_subalgebra}

Let $M \subset G$ be a closed reductive subgroup and $\mg =  \Lie(M)$.
Assume that $e \in \mg \cap \0'$.  Let $x \in \mg \cap \0$ and suppose that $M \cdot e \subset \ov{M\cdot x}$.
Since $\mg$ is reductive, we may assume the  $\sl_2$-subalgebra $\sg$ containing $e$ lies in $\mg$.
Let $\SC_{M\cdot x, e}$ be the nilpotent Slodowy slice 
$\ov{M\cdot x} \cap (e + \mg^f)$ in $\mg$.
Of course, $\SC_{M\cdot x, e} \subset \slice$.

\begin{lemma}\label{singequi}
Suppose that  $\codim_{\ov{M\cdot x}} (M \cdot e) = \codim_{\cl} (\0')$ and $\SC_{M\cdot x, e}$ is equidimensional.
Then $\SC_{M\cdot x, e}$ is a union of irreducible components of $\slice$.
Moreover if $\cl$ is unibranch at $e$  or  if the number of branches of $\cl$ at $e$ equals the number of branches of $\ov{M\cdot x}$ at $e$, then $\SC_{M\cdot x, e} = \slice$.
\end{lemma}

\begin{proof}
The first conclusion follows from \eqref{dim=codim} and the fact that $\SC_{M\cdot x, e} \subset \slice$, together with the hypotheses of the lemma. 
The second conclusion follows from the fact that the irreducible components 
of $\slice$ and $\SC_{M\cdot x, e}$ are unibranch (\S \ref{Component group action}).
\end{proof}

\begin{example}
Let $\gg$ be of type $E_8$, $\mg$ a Levi subalgebra of $\gg$ of type $E_6$, and $(\0, \0')$ of type $(D_5, E_6(a_3))$.
Since $\cl \cap \mg$ is known to be normal in $E_6$, we are able to conclude (see \S \ref{surfaces_E6}) that $\SC_{M\cdot x, e}$
is geometrically a simple surface singularity of type $D_4$.   Since $\cl$ is unibranch at $e$ and the codimension hypothesis of the lemma holds (both sides equal two), it follows that $\slice = \SC_{M\cdot x, e}$ and so $\slice$ has the same singularity.  
Here, $\cl$ is conjectured to be normal, but this is still an open question.
\end{example}

Lemma \ref{singequi} is needed for the cases in Table \ref{restriction_subalgebra}, but we also use it to check results that can be obtained by other methods. 

\subsubsection{The case of a $C(\sg)$-orbit of maximal possible dimension}  \label{C_closures}

Every $x \in \slice$ can be written as 
\begin{equation} \label{x_form1}
x = e + x_0 + \sum_{i > 0 }  x_i 
\end{equation} 
with $x_i \in \gg^f(-i)$.  
Set $$x_{+} = \sum_{i \geq 0} x_i  \hspace{.1in} \text{    and  }  \hspace{.1in} X = \ov{C(\sg) \cdot  x_+ }.$$ 
Since $C(\sg)$ fixes $e$, we have  
$C(\sg) \cdot x = e + C(\sg) \cdot x_+$ and thus
$\ov{C(\sg)\cdot x}  = e + X .$
Also, $e + X \cong X$ as $C(\sg)$-varieties.
By construction $X$ is equidimensional, with irreducible components permuted transitively by $C(\sg)/C(\sg)^\circ$.
Of course $\ov{C(\sg)\cdot x} \subset \slice$.
Hence the same argument as in the previous lemma gives



\begin{lemma}  \label{kp_prop}
Let $x \in \slice$ be written as in \eqref{x_form1}.
Suppose that $\dim (C(\sg) \cdot x) = \codim_{\cl} (\0')$.
Then $e+X$ is a union of irreducible components of $\slice$. 
Moreover, if the number of branches of $\cl$ at $e$ equals the number of irreducible components of $X$,
then $e+X =\slice$. 
\end{lemma}

The previous lemma allows us to study $\slice$ by studying $\ov{C(\sg) \cdot  x_+ }$, which is often easier to understand.  
Of course any $x \in \slice$ satisfying the dimension hypothesis in the lemma must lie in $\slice \cap \0$.  Furthermore 
we have

\begin{lemma}  \label{oh_yeah}
Let $x \in \slice$ be written as in \eqref{x_form1}.
Then $\dim( C(\sg) \cdot x ) = \codim_{\cl} (\0')$ if and only if
\begin{equation}  \label{both_conditions}
x_0 \text{ is nilpotent and } \dim (C(\sg) \cdot x_0)  = \codim_{\cl} (\0').
\end{equation}
\end{lemma}

\begin{proof}
We always have $\dim (C(\sg) \cdot x_0)  \leq \dim( C(\sg) \cdot x ) \leq \dim \slice = \codim_{\cl} (\0')$,
so the reverse direction is straightforward and does not use the hypothesis that $x_0$ is nilpotent.

For the forward direction, we are given that $\dim( C(\sg) \cdot x ) = \dim \slice$.
Consider the $\CM^*$-action (\S \ref{group action}) on $\slice$.
We have $C(\sg) \cdot x \subset (C(\sg) \times \CM^*) \cdot x \subset \slice$ and
all have the same dimension, so $C(\sg) \cdot x$ is dense in $(C(\sg) \times \CM^*) \cdot x$.
Therefore for $\lambda \in \CM^*$, we have $\lambda \cdot (C(\sg) \cdot x)$ meets $C(\sg) \cdot x$,
from which it follows that $\lambda \cdot x \in C(\sg) \cdot x$.  So in fact $\CM^*$ preserves 
$C(\sg) \cdot x$.  But if $\CM^* \cdot x \subset  C(\sg) \cdot x$, then 
$\CM^* \cdot x_0 \subset  C(\sg) \cdot x_0$.  
The $\CM^*$-action on $\cg(\sg)$ is contracting, hence $0 \in \ov{C(\sg) \cdot x_0}$, and $x_0$ 
must be nilpotent in $\cg(\sg)$.

Next by \cite[Corollary 7.2]{Gan-Ginzburg}, $\slice \cap \0$ is a symplectic subvariety of $\0$, so the symplectic form on  $T_x(\0)$
remains non-degenerate on restriction to $T_x(\slice \cap \0)$.
As usual, we identify $T_x(\0)$ with $[x, \gg]$ and the symplectic form on $T_x(\0)$ is 
then expressed as $\langle [x,u], [x, v] \rangle := \kappa (x, [u,v])$
for $u, v  \in \gg$, where $\kappa$ is the Killing form on $\gg$.
But since $C(\sg) \cdot x$ has dimension equal to $\slice \cap \0$, we can identify $T_x(\slice \cap \0)$ with $[x, \cg(\sg)]$.

Now suppose  $u \in \cg(\sg)$ satisfies $[x_0, u] = 0$.  
Then for  any $v \in \cg(\sg)$, we have 
$$\langle [x,u], [x, v] \rangle  = \kappa (x, [u,v]) = \kappa (x_0, [u,v])$$ 
since $[u,v] \in  \cg(\sg) = \gg^f(0)$ pairs nontrivially only with elements in the $0$-eigenspace of $\ad h$.
But then  $\langle [x,u], [x, v] \rangle  =\kappa (x_0, [u,v]) = \kappa ([x_0, u], v) =  0$.  Hence
$[x,u]$ is in the kernel of the symplectic form and thus $[x,u]=0$ by non-degeneracy of the form.  This shows that
$\cg(\sg)^{x_0} \subset \cg(\sg)^{x}$, which forces
$\dim (C(\sg) \cdot x)  = \dim (C(\sg) \cdot x_0)$, as desired.
\end{proof}

The case where $x_i = 0$ for all $i \geq 1$ in \eqref{x_form1}, and \eqref{both_conditions} holds, is particularly common 
and is also easier to handle.  In that case, $x = e + x_0$ is a sum of two commuting nilpotent elements and 
$X =  \ov{C(\sg) \cdot x_0}$ is the closure of a nilpotent orbit in $\cg(\sg)$, which is 
a union of irreducible components of $\slice$.  We discuss this case in detail in \S \ref{dim_condition}.
We next prove a lemma useful for when some $x_i$ is nonzero with $i \geq 1$.

\subsubsection{}  \label{add_in_dimension_condition}

Assume $x \in \slice \cap \0$, written as in \eqref{x_form1},
satisfies \eqref{both_conditions}.   The next lemma uses the $\CM^*$-action on $\slice$
to say more about the $x_i$'s which appear in  \eqref{x_form1}.
Since $x_0$ is nilpotent, we can find an $\sl_2$-subalgebra $\sg_{x_0}$ in $\cg(\sg)$ containing $x_0$, with
standard semisimple basis element $h_{x_0}$.

\begin{lemma}  \label{2+}
Let $x \in \0$ be written as in \eqref{x_form1} so that 
 \eqref{both_conditions} holds.
Then the following are true.
\begin{enumerate}
\item $[x_0 , x_i] = 0$ for $i \geq 0$.
\item $[h_{x_0}, x_i] = (i+2) x_i$ for $i \geq 0$.
\item If $x_0$ lies in a minimal nilpotent $C(\sg)^\circ$-orbit of $\cg(\sg)$, then 
each nonzero $x_i$ is a highest weight vector for a Borel subgroup $B$ of $C(\sg)^\circ$.
In particular, $X$ is a union of $S$-varieties.  If $x_0$  has weight $\lambda$ relative to a maximal torus  of $B$, 
then $x_i$ has weight $(\tfrac{i}{2} + 1) \lambda$.
\end{enumerate}
\end{lemma}

\begin{proof}
(1) By Lemma \ref{oh_yeah} we know that $\dim (C(\sg) \cdot x_0)  = \dim (C(\sg) \cdot x)$,
which is equivalent to $\cg(\sg)^{x_0} = \cg(\sg)^{x}$, which is equivalent to
$\cg(\sg)^{x_0} \subset \cg(\sg)^{x_i}$ for $i \geq 0$. 
Since $x_0$ commutes with itself, we get $[x_0 , x_i] = 0$ for $i \geq 0$.

(2)  
From the proof of Lemma \ref{oh_yeah},
the dimension condition in \eqref{both_conditions} implies that 
the $\CM^*$-action on $\slice$ preserves $C(\sg) \cdot x$.
Let $\CC:=C(\sg) \times \CM^*$. 
The equality $\dim (C(\sg) \cdot x_0)  = \dim (C(\sg) \cdot x)$ therefore
implies $\dim (\CC \cdot x_0) = \dim (\CC \cdot x)$, which means
we have the inclusion of identity components of centralizers
\begin{equation} \label{centralizer_containment}
(\CC^{x_0})^{\circ} \subset(\CC^{x_i})^{\circ} \text{ for all } i. 
\end{equation}


Now write $\chi(t)$ for the element $(1, t) \in C(\sg) \times \CM^*$.
Let $\phi: \CM^* \to C(\sg) \times \CM^*$ be the cocharacter coming from $h_{x_0}$.
Of course $\phi(\CM^*) \subset C(\sg)$ commutes with $\chi(\CM^*)$.
Consider now the action of the element
$\chi(t^{-1}) \phi(t)$ on $x_0$ for $t \in \CM^*$.
We have $\phi(t).x_0 = t^2 x_0$ since $[h_{x_0},x_0]=2$ and $\chi(t).x_0 = t^{0+2}x_0$ since $x_0 \in \gg^f(0)$,
so the one-dimensional torus $\{ \chi(t^{-1}) \phi(t) \ | \ t \in \CM^*\}$ fixes $x_0$ and hence by \eqref{centralizer_containment},
it also fixes each $x_i$.  The result follows from $\phi(t).x_i = \chi(t).x_i = t^{i+2}x_i$ since $x_i \in \gg^f(-i)$.
Combining with part (1), we see that each nonzero $x_i$ is a highest weight vector for the Borel subalgebra of $\sg_{x_0}$ spanned by $x_0$ and $h_{x_0}$.

(3)  Since $x_0$ lies in a minimal nilpotent $C(\sg)$-orbit of $\cg(\sg)$, 
the stabilizer of the line through $x_0$ is a parabolic subgroup $P$ of $C(\sg)$,
containing $\phi(\CM^*)$.  Let $B$ be a Borel subgroup of $P$ containing $\phi(\CM^*)$ and 
let $T$ be a maximal torus in $B$ with $\phi(\CM^*) \subset T$.  Let $U$ be the unipotent radical of $B$,
which acts trivially on $x_0$.
Now $x_0$ is a root vector relative to $T$,
so $T^{x_0}$ is codimension one in $T$.
Thus $T$ is generated by $T^{x_0}$ and $\phi(\CM^*)$, so
each element of $B$ can be written as $u t_0 \phi(t)$ for $u \in U$, $t_0 \in T^{x_0}$, and $t \in \CM^*$.
It follows that the connected subgroup 
$\{ u t_0 \phi(t) \chi(t^{-1}) \}$ of $C(\sg) \times \CM^*$ centralizes $x_0$ and hence centralizes 
each $x_i$ by \eqref{centralizer_containment}.  Hence $B$ (and indeed also $P$) preserves the line through $x_i$ 
since $\chi(\CM^*)$ does;  in other words, $x_i$ is 
a highest weight vector relative to $(B,T)$.   
Moreover, the weight of $x_i$ must equal $r \lambda$, with $r$ rational, since $x_i$ and $x_0$ are both
acted upon trivially by $T^{x_0}$.  
On the one hand, $\phi(t).x_i = t^{i+2}x_i$ from part (2), and 
on the other hand, $\phi(t).x_i = (r\lambda)(\phi(t))x_i = (\lambda(\phi(t)))^r x_i = (t^{2})^r x_i$.
We conclude that $r = (\tfrac{i}{2} + 1)$.
\end{proof}

Another way to phrase part (2) of Lemma \ref{2+} is that 
$x$ must be in the $2$-eigenspace for $\ad(h+h_{x_0})$.   
This fact can be used to help locate an $x$ written as in \eqref{x_form1} so that 
 \eqref{both_conditions} holds, see Lemma \ref{slice_in_2_space}.

Part (3) of Lemma \ref{2+} will be used in the proof of Proposition \ref{dim4+_1} in \S \ref{section:dim4},
since for each minimal degeneration covered by the proposition,
it turns out that there exists $x \in \slice$ 
satisfying  \eqref{both_conditions} with $x_0$ in a minimal nilpotent orbit of $\cg(\sg)$.


\subsubsection{}


When the degeneration is codimension two and 
 $\cg(\sg)$ has a simple summand isomorphic to $\sl_2(\CM)$, 
the next lemma gives a criterion which guarantees that 
$\slice$ has a dense $C(\sg)$-orbit, allowing us to apply the previous lemmas.

\begin{lemma}  \label{S_var_sl2}
Let $\0'$ be a degeneration of $\0$ of codimension two (necessarily minimal). 
Suppose $C(\sg)^\circ$ contains a simple factor $Z$ 
with Lie algebra $\zg \cong \sl_2(\CM)$.  
Let $C(\zg)$ be the centralizer of $\zg$  in $G$, with Lie algebra $\cg(\zg)$.
\begin{enumerate}
\item If $Z$ acts non-trivially on $\slice$, then there exists $x \in \slice \cap \0$, written as in \eqref{x_form1},
satisfying \eqref{both_conditions} with $x_0 \in \zg$.
\item If $x \not \in \cg(\zg)$, or if $x \in \cg(\zg)$ but $e \not \in \overline{C(\zg) \cdot x}$, then $Z$ acts non-trivially on $\slice$.
\end{enumerate}
\end{lemma}

\begin{proof}

(1) Consider the decomposition of $\gg^f$ under $Z$.  Since the action of $Z$ is non-trivial on $\slice$, 
there exists $x= e+ x_+ \in \slice$ and a non-trivial irreducible $Z$-submodule $V$ of some $\gg^f(-j)$ so that 
$x_+$ has nonzero image $x_V$ under the $C(\sg) \times \CM^*$-equivariant projection $\pi_V$ of $\gg^f$ onto $V$.
The equivariance of $\pi_V$ ensures that the image $\pi_V(\slice)$ is conical for the scalar action on $V$.
Let $Y$ be the projectivization of $\pi_V(\slice)$, which is in $\PM(V)$.  Since $\dim \slice =2$ by hypothesis,
we know $\dim Y \leq 1$.  Now $Y$ has a closed $Z$-orbit, necessarily irreducible.  
If $Y$ contains a closed orbit which is a point, then $Z$ preserves a line in $V$, contradicting that $V$ is non-trivial irreducible.
Thus, since $\dim Y \leq 1$, each irreducible component of $Y$ must be a one-dimensional (closed) $Z$-orbit.  
The stabilizer of any point in $Y$ is therefore a proper parabolic subgroup in $Z$, namely a Borel subgroup $B$. 
Since then $\dim (Z \cdot x_V) = 2 = \dim \slice$, we have also 
$\dim (Z \cdot x) = \dim \slice$.
Then Lemma \ref{oh_yeah} ensures that $x$ 
satisfies \eqref{both_conditions} with $x_0 \in \cg(\sg)$,
when $x$ is expressed as in \eqref{x_form1}.    
But then for dimension reasons the other connected simple factors of $C(\sg)$ 
must preserve $Z \cdot x_0$, hence act trivially, which implies that $x_0 \in \zg$.

(2) If $Z$ acts trivially on $\slice$, then
$\slice \subset \cg(\zg)$.  Hence $x \in \cg(\zg)$.  Since $\sg \subset \cg(\zg)$, it follows that the $\CM^*$-action 
on $\slice$ preserves nilpotent $C(\zg)$-orbits.
Hence $e \in  \overline{C(\zg) \cdot x}$.
(This part did not use the assumption on the codimension of $\0'$ in $\0$).
 \end{proof}



 \begin{corollary} \label{A1_or_m}
 Let $Z$ act non-trivially on $\slice$.
 Then there exist ${b_{i}} \in \mathbb N$ satisfying  $2 < b_2 < b_3 < \dots$
 so that each irreducible component of $\slice$ is an $S$-variety for $Z$
 of the form $X(2, b_2, b_3, \dots)$.  
 
 Consequently, the components of $\slice$ are
 normal if and only if all the $b_i$ are even, in which case each component is isomorphic to the nilcone in $\sl_2(\CM)$.
In the non-normal case, the normalization of an irreducible component is $\mathbb A^2$.
 \end{corollary}
 
 \begin{proof}
 By part (1) of the lemma, there exists $x \in \slice \cap \0$, when written as in \eqref{x_form1},
satisfying \eqref{both_conditions} with $x_0 \in \zg$.
Hence $x_0$ must belong to the minimal nilpotent orbit in $\zg$.
 By Lemma \ref{2+} the $x_i$'s are highest weight vectors, of weight $i+2$, 
 for a Borel subalgebra of $\zg$.
Thus the irreducible component $\ov{Z \cdot x}$ of $\slice$ 
is an $S$-variety for $Z$ of the form $X(2, b_2, b_3, \dots)$ with $2 < b_2 < b_3 < \dots$.  
 But since the degeneration is minimal, $C(\sg)$ acts transitively on the irreducible components of $\slice$ by Proposition \ref{transitive_components}, and thus 
 each irreducible component of $\slice$ takes this same form.

By \S \ref{S_vars}, $X(2, b_2, b_3, \dots)$ is normal if and only if the $b_i$'s are all even, 
in which case it is isomorphic to $X(2)$, the nilcone in $\sl_2(\CM)$, which is the $A_1$-singularity.
Otherwise, its normalization is $X(1) \cong \mathbb A^2$. 
\end{proof}

\subsubsection{An example in $C_3$}  \label{C3_example_1} \label{C3_example_2}
Let $\gg = \sp_6(\mathbf C)$.  
Nilpotent orbits in $\gg$ can be parametrized by the Jordan partition for any element in the orbit, 
viewed as a $6 \times 6$-matrix.
Pick $e \in \gg$ with partition $[2^3]$. 
Then 
$\cg(\sg) \cong \sl_2(\CM)$.  
Set $\zg = \cg(\sg)$.  
A nonzero nilpotent $e_0 \in \zg$ has type $[3^2]$ in $\gg$, 
so $\cg(\zg) = \sg$ and thus the only non-zero nilpotent $G$-orbit that meets $\cg(\zg)$
is the one through $e$.
Consequently, part (2) of Lemma \ref{S_var_sl2} applies 
to any $\0$ of which $\0' = \0_e$ is a degeneration of codimension two.
Let $\0_{[4,1^2]}$ and $\0_{[3^2]}$ be the nilpotent orbits with given partition type.  Then
$\0'$ is a minimal degeneration of both orbits, in each case of codimension two.
So in both cases Corollary \ref{A1_or_m} applies.  
Now $C(\sg)$ acts on $\gg^f$ with $\gg^f(0) = \cg(\sg) = V(2)$ and $\gg^f(-2) \cong V(4) \oplus V(0)$.  
Also $e+e_0$ belongs to $\0_{[4,2]}$ (see \S \ref{subsection:elements_in_red_centralizer}), 
so we can say that for both $\0 = \0_{[4,1^2]}$ and $\0 = \0_{[3^2]}$,  each irreducible component of 
$\slice$ is of the form $X(2,4) \cong X(2)$, which is the nilcone in $\sl_2(\CM)$. 
But since $\cl$ is normal for both orbits $\0$ by \cite{Kraft-Procesi:classical}, it follows that $\slice$ is irreducible.
In particular the singularity of $\cl$ at $e$ is an $A_1$-singularity in both cases,
as was already known from \cite{Kraft-Procesi:classical}. 


\subsection{Locating nilpotent elements in $\cg(\sg)$.}  \label{subsection:elements_in_red_centralizer}

In order to make use of Lemma \ref{oh_yeah} or part (2) of Lemma \ref{S_var_sl2}, we will need to 
describe nilpotent elements in $c(\sg)$ relative to the embedding of $c(\sg)$ in $\gg$. 
We will also need to be able to start with nilpotent $e_0 \in \cg(\sg)$ and then compute the $G$-orbit to which $e+e_0$ belongs.

First, if $e_0 \in \cg(\sg)$, then $e_0$ centralizes the semisimple element $h \in \sg$.  
Hence $e_0 \in \gg^h$, which is a Levi subalgebra of $\gg$.   Assume $h$ lies in a chosen Cartan subalgebra $\hg \subset \gg$ 
and is dominant for a chosen Borel subalgebra $\bg \subset \gg$ containing $\hg$.
The type of the Levi subalgebra $\gg^h$ can then be read off from the weighted Dynkin diagram for $h$:  
the Dynkin diagram for the semisimple part of $\gg^h$ corresponds to the zeros of the diagram.   
Therefore in order to locate a nilpotent element in $\cg(\sg)$, we first choose a nilpotent element $e_0 \in \gg^h$;  
the $G^h$-orbits of such elements are known by Dynkin's and Bala and Carter's results \cite{Carter}.   
In particular we can compute the semisimple element $h_0 \in \gg^h \cap \hg$ of an 
$\sl_2$-subalgebra $\sg_0$ through $e_0$ in $\gg^h$.   

Next, we compute $h+h_0$ and see whether it corresponds to a nilpotent orbit in $\gg$:  
for if $e$ and $e_0$ (or some conjugate of $e_0$ under $G^h$) commute, then $h+h_0$ will be the semisimple element in an 
$\sl_2$-subalgebra through the nilpotent element $e+e_0$.
Together with knowledge of the Cartan-Killing type of the reductive Lie algebra $\cg(\sg) \subset \gg^h$ (see \cite{Carter}), 
this search usually suffices to locate the nilpotent orbit through $e_0$ in $\gg$ for nilpotent elements $e_0 \in \cg(\sg)$ and 
the resulting nilpotent orbit through $e + e_0$.
In particular we carried out this approach for all the minimal nilpotent $C(\sg)$-orbits in $\cg(\sg)$.  
Two special situations are worth mentioning.  

\subsubsection{}  \label{Minimal in g}

One special situation is when $e_0$ is minimal in $\gg$, that is, of type $A_1$.
Then the semisimple part of $\cg(\sg_0)$ is the semisimple part of a Levi subalgebra of $\gg$, the one corresponding to the nodes 
in the Dynkin diagram which are not adjacent to the affine node in the extended Dynkin diagram.
Of course $e \in \cg(\sg_0)$.   Consequently it is easy to locate all $e$ which have $e_0 \in \cg(\sg)$ when $e_0$ is of type $A_1$ in $\gg$.

We will see in Corollary \ref{slice_corollary} that Lemma \ref{kp_prop} always applies in this setting with $x = e+e_0$.
Moreover the type of $x$ in $\gg$ is easy to determine: 
if we know the type of $e$ in $\cg(\sg_0)$, call it $X$, then $x$ has generalized 
Bala-Carter type $X+A_1$.  Then the usual type can be looked up in \cite{Sommers:Bala-Carter} or in Dynkin's seminal paper \cite{Dynkin}.

For example, in $E_8$ when $e_0$ is of type $A_1$, then $\cg(\sg)$ is of Cartan-Killing type $E_7$.   Any nilpotent element $e$ in a Levi subalgebra of type $E_7$ will have a conjugate of $e_0$ in $c(\sg)$.  If, for instance, $e$ is a regular nilpotent element, then $e+e_0$ has generalized Bala-Carter type $E_7+A_1$, which is the same as $E_8(a_3)$.

There is another way to determine $e+e_0$ when $e_0$ is minimal in $\gg$.  
It has the advantage of locating the simple summand of $\cg(\sg)$ in which $e_0$ lies.    
As above, assume $h$ is dominant relative to $\bg$.
Since $e_0 \in \gg^h$ has type $A_1$,
the semisimple element $h_0 \in \hg$ is equal to the coroot of a long root $\theta$ for $\gg^h$.   Therefore,
$\alpha(h_0) \geq -2$ for any root of $\gg$ and equality holds if and only $\alpha = - \theta$.
Now choose $h_0$ dominant in $\gg^h$ (relative to $\bg \cap \gg^h$).
Then $\alpha(h_0) \geq -1$ for all simple roots $\alpha$ of $\gg$ since $-\theta$ is a negative root.
Moreover $\alpha(h_0) = -1$ only if $\alpha$ is not a simple root for $\gg^h$.  In that case $\alpha(h) \geq 1$ since 
the simple roots of $\gg^h$ correspond to the zeros of the weighted Dynkin diagram for $h$.   
This shows that $\alpha(h+h_0) \geq 0$ for all simple roots $\alpha$ of $\gg$ and thus $h+h_0$ 
yields the weighted Dynkin diagram for $e+e_0$ without 
having to conjugate by an element of the Weyl group.

For example,  let $e$ belong to the orbit $E_7(a_3)$ in $E_8$, which has weighted Dynkin diagram
$$\begin{smallmatrix}  2 & 0 & 1 & 0 & 1 & 0  & 2 \\  & & 0 \end{smallmatrix}.$$   
Then $\gg^h$ has type $4A_1$ 
and $\cg(\sg)$ is isomorphic to $\sl_2(\CM)$ since $\cg(\sg)$ has rank one (because $e$ is distinguished in a Levi subalgebra of rank $7$) 
and $\cg(\sg)$ contains $e_0$, a nonzero nilpotent element. 
We want to know in which summand of $\gg^h$ 
the element $e_0$ lies and what is $e+e_0$.   The diagram for $h_0$ relative to $\gg$, and dominant for $\gg^h$, is either:
$$\begin{smallmatrix} 0 & 0 & 0  & 0 & -1 & 2 & -1   \\  & & 0 \end{smallmatrix}   \text{ or } \begin{smallmatrix}  0 & 0 & -1 & 2 & -1 & 0  & 0 \\   & & 0 \end{smallmatrix} \text{ or } \begin{smallmatrix}   -1 & 2 & -1 & 0 & 0 & 0 & 0 \\   & & 0 \end{smallmatrix}  
\text{ or } \begin{smallmatrix}  0 & 0 & -1 & 0 & 0 & 0 & 0 \\  & & 2 \end{smallmatrix}.$$
Only the second choice leads to a weighted Dynkin diagram for $h+h_0$, namely for $D_7(a_1)$.  
Hence we know the type of $e+e_0$ and the embedding of $\cg(\sg)$ in $\gg^h$.

\subsubsection{} 
The other special situation occurs when $\cg(\sg)$ has rank $1$.  
Let $\lg$ be a minimal Levi subalgebra containing $e$.  Then $\lg$ has semisimple rank equal to the rank of $\gg$ minus one.  Assume that $\lg$ is a standard Levi subalgebra.   Let $\alpha_i$ be the simple root of $\gg$ which is not a simple root of $\lg$.  
For nonzero $e_0 \in \cg(\sg)$, the corresponding $h_0$ centralizes $\lg$ and 
hence lies in the one-dimensional subalgebra of $\hg$ spanned by the coweight 
$\omega^{\vee}_i$ for 
$\alpha_i$.  Since the values in any weighted Dynkin diagram are $0, 1,$ or $2$, if $h_0$ is dominant, then $h_0$ must be either $\omega^{\vee}_i$ or $2 \omega^{\vee}_i$.

For example, let $e$ be of type $A_7$ in $E_8$, which has weighted Dynkin diagram 
$\begin{smallmatrix}  1 & 0 & 1 & 0 & 1 & 1  & 0 \\  & & 0 \end{smallmatrix}$.  Then $\cg(\sg)$ has type $A_1$ and 
the weighted Dynkin diagram of a nonzero $h_0 \in \cg(\sg)$ must either be 
$$\begin{smallmatrix}  0 & 0 & 0 & 0 & 0 & 0  & 0 \\  & & 1 \end{smallmatrix}  
  \text{ or }  \begin{smallmatrix}  0 & 0 & 0 & 0 & 0 & 0  & 0 \\  & & 2 \end{smallmatrix}.$$  
 Both of these are actual weighted Dynkin diagrams in $E_8$, the first is $4A_1$ and the second is $D_4(a_1)+A_2$. 
Only the orbit $4A_1$ meets $\gg^h$ (which has semisimple type exactly $4A_1$).  
Therefore a nonzero nilpotent element $e_0 \in \cg(\sg) \subset \gg^h$ has type $4A_1$ in $\gg$.

\subsection{The case where $x_i=0$ for $i \geq 1$ in \eqref{x_form1}, and \eqref{both_conditions} holds.}\label{dim_condition}
Once a nilpotent $e_0 \in \cg(\sg)$ is located, as in the previous section, 
with corresponding semisimple element $h_0 \in \cg(\sg)$,
we can compute 
$h+h_0$  and check by hand whether the dimension condition
\begin{equation} \label{dim_equality_full_e0}
\dim C(\sg)\cdot e_0 =\codim_{\cl}( \OC_e) 
\end{equation}
holds for the orbit $\0$ through $e+e_0$.   
If it does,  then certainly $x:= e+e_0$ satisfies \eqref{x_form1} with $x_0 = e_0$ and $x_i = 0$ for $i \geq 1$, 
and the dimension condition in \eqref{both_conditions} just becomes \eqref{dim_equality_full_e0}.
By Lemmas \ref{kp_prop} and \ref{oh_yeah},
the union of some 
of the irreducible components of $\slice$ is thus isomorphic to $\ov{C(\sg) \cdot e_0}$.
Next we give a condition for \eqref{dim_equality_full_e0} to hold for the orbit $\0 = \0_{e+e_0}$
and show that this condition is always true when $e_0$ belongs to the minimal orbit in $\gg$.


As before, let $\sg_0$  be an $\sl_2$-subalgebra in $\cg(\sg)$ with standard basis $e_0, h_0, f_0$.  
Clearly, $\sg$ and $\sg_0$ commute.   
We will now establish an equivalent condition
to the dimension condition \eqref{dim_equality_full_e0} 
in terms of the decomposition of $\gg$ into irreducible subrepresentations for 
$\sg \oplus \sg_0 \cong \sl_2(\CM) \oplus \sl_2(\CM)$. 

Let $V_{m,n}$ denote an irreducible
representation of $\sg \oplus \sg_0$ 
with $h \in \sg$ acting by $m$ and $h_0 \in \sg_0$ acting by $n$ 
on a highest weight vector $u \in V_{m,n}$ annihilated by both $e$ and $e_0$.
The eigenvalues of $h+h_0$ on $V_{m,n}$ are either all even if 
$m$ and $n$ have the same parity or all odd if 
$m$ and $n$ have opposite parities. 
In the former case the quantity 
$$\min(m,n)+1$$ 
is equal to the dimension of the 0-eigenspace of $h+h_0$; in the latter case,
it is equal to the dimension of the 1-eigenspace of $h+h_0$.   This is analogous to what occurs 
in the proof of the Clebsch-Gordan formula.

Let
\begin{equation}  \label{sl2_decomp}
\gg = \bigoplus_{i=1}^N V^{(i)}_{m_i, n_i}
\end{equation}
be a decomposition into irreducible subrepresentations $V^{(i)}_{m_i, n_i} \cong V_{m_i, n_i}$ 
for the action of $\sg \oplus \sg_0$.
The relationship between 
\eqref{dim_equality_full_e0}
and this decomposition in \eqref{sl2_decomp} is the following:

\begin{proposition}   \label{prop:codimensions_match}
Let $\0$ be the orbit through $e+e_0$.  
The dimension condition \eqref{dim_equality_full_e0}
holds if and only if
\begin{equation}  \label{constraints}
m_i \geq n_i \text{ whenever }  m_i > 0.
\end{equation} 
\end{proposition}

\begin{proof}
By $\mathfrak{sl}_2(\CM)$-theory, the sum of the dimensions of the
0-eigenspace and the 1-eigenspace for $\ad(h+h_0)$ on $\gg$ 
equals the dimension of the centralizer of $x=e+e_0$ in $\gg$. 
It therefore follows
that 
$$\dim {\gg}^x  = \sum_{i=1}^N  \left(\min(m_i, n_i) + 1\right).$$
At the same time 
$$\dim {\gg}^e = \sum_{i=1}^N \left(n_i + 1\right)$$
since the kernel of $\ad(e)$ on $V^{(i)}_{m_i, n_i}$ is isomorphic to $V(n_i)$.  Here, 
$V(n_i)$ is an irreducible representation of $\sg_0 \cong \sl_2(\CM)$ of highest weight $n_i$, hence of dimension $n_i+1$.
Putting the two formulas together, the codimension of $\OC_{e}$ in $\overline{\OC}_{x}$ is equal to 
$$\sum_{i=1}^N \left(n_i - \min(m_i, n_i)\right).$$

It is also necessary to compute $\dim \cg(\sg)^{e_0}$. 
Since $\sg_0 \subset \cg(\sg)$ and $\cg(\sg)$ is exactly $\ker \ad e \cap \ker \ad h$,
it follows that $\cg(\sg)$ coincides with the sum of all $V^{(i)}_{m_i, n_i}$ where $m_i=0$.   
The centralizer $\cg(\sg)^{e_0}$ is then the span of the highest weight vectors of
these $V^{(i)}_{0, n_i}$ and hence its dimension is given by the number
of these subrepresentations.  That is,
$$\dim \cg(\sg)^{e_0} = \#  \{ 1 \leq i \leq N \ | \ m_i = 0 \}.$$
Thus 
$$\dim C(\sg)\cdot e_0=  \dim \cg(\sg) - \dim \cg(\sg)^{e_0} = \sum_{m_i=0} (n_i+1) - \sum_{m_i=0} 1 = \sum_{m_i=0} n_i.$$
The equality of $\dim C(\sg)\cdot e_0$ and the codimension of $\0_{e}$ in $\cl_x$ 
is therefore equivalent to
$\min(m_i, n_i) = n_i$ for all $i$ with $m_i \neq 0$.  
\end{proof}

It follows from the above proof that if $\JC= \{  i  \ | \  n_i > m_i >0 \}$, then
\begin{equation} \label{dimension_diff}
\dim \slice - \dim C(\sg) \cdot e_0 = \sum_{i \in \JC} (n_i - m_i).
\end{equation}

The element $e_0 \in \gg$ is called height $2$ if all the eigenvalues of $\ad h_0$ on $\gg$ are at most $2$, and
$e$ is called even if all the eigenvalues of $\ad h$ on $\gg$ are even.

\begin{corollary} \label{slice_corollary}
Suppose that either (1) $e_0$ belongs to the minimal nilpotent orbit in $\gg$, or (2)
$e_0$ is of height $2$ in $\gg$ and $e$ is even.
Then
the dimension condition \eqref{dim_equality_full_e0} holds.
\end{corollary}

\begin{proof}
If $e_0$ belongs to the minimal nilpotent orbit of $\gg$, then 
$e_0$ is of height two and the $2$-eigenspace of $\ad h_0$ is spanned by $e_0$.
This is the case since $h_0$ is conjugate to the coroot of the highest root.   
But since $\sg_0 \subset \cg(\sg)$, it follows that $\sg_0 \cong V_{0,2}$ is the unique subrepresentation 
of $\gg$ isomorphic to $V_{m,n}$ with $n \geq 2$.
Therefore all other $V^{(i)}_{m_i, n_i}$ must have $n_i = 0$ or $n_i = 1$ and so condition (\ref{constraints}) holds.

Next assume the second hypothesis.  Since $e$ is even, all $V^{(i)}_{m_i, n_i}$ with $m_i > 0$ satisfy $m_i \geq 2$.  Since 
$e_0$ is of height two, $n_i \leq 2$ and 
thus condition (\ref{constraints}) is true. 
\end{proof}

\subsection{The case where $x_i \neq 0$ for some $i \geq 1$ in \eqref{x_form1}, and \eqref{both_conditions} holds.}  \label{extra_eigenvalues}

Let $e_0 \in \cg(\sg)$ be nilpotent 
and suppose that the dimension condition 
\eqref{dim_equality_full_e0}
does not hold for $\0 = \0_{e + e_0}$.   It may happen instead that 
Lemma \ref{kp_prop} applies for a different nilpotent orbit $\0$ 
with $\0_e \subset \cl \subset \cl_{e+e_{\circ}}$.  More precisely,
it may be possible to find $x \in \slice$, written as in \eqref{x_form1},
so that $x_0 = e_0$ and \eqref{dim_equality_full_e0} does hold for this $\0$.
Then Lemma \ref{oh_yeah} ensures that Lemma \ref{kp_prop} applies to $\slice$.
Now in such a situation, Lemma \ref{2+} implies that $x$ must lie in the $2$-eigenspace for $\ad(h+h_0)$.
We now use this information to give one way to help locate such an $x$ when it exists.

\subsubsection{A smaller slice result}  \label{small_slice}

Let $y = e + e_0$, 
which is nilpotent with corresponding semisimple element 
$h_y = h+h_0$.
Write 
$\gg_j$ for the $j$-eigenspace of $\ad h_y$ on $\gg$.    
The centralizer $G_0 := G^{h_y}$ has Lie algebra $\gg_0$ and $G_0$ acts on each $\gg_j$.
Then $y \in \gg_2$ and the $G_0$-orbit through $y$ is the unique dense orbit in $\gg_2$.
Now $e \in \gg_2$ since 
\begin{equation} \label{lies_in_2_space}
[h_y, e] = [h+h_0, e] = 2e + 0 = 2e.
\end{equation}
We want to find a transverse slice in $\gg_2$ to the $G_0$-orbit through $e$.
In fact, since $\gg_2$ is a direct sum of $\ad h$-eigenspaces,
the decomposition $\gg = \im \ad e \oplus \ker \ad f$
restricts to a decomposition
$$\gg_2 = [e,\gg_0] \oplus (\gg_2 \cap  \ker \ad f).$$
Therefore, setting $\SC^{(2)}_e = e + (\gg_2 \cap  \ker \ad f)$, it follows that
the affine space $\SC^{(2)}_{e}$ is a tranverse slice of $\gg_2$ at $e$ with respect to the $G_0$-action.
Consequently, every $G_0$-orbit in $\gg_2$ containing $e$ in its closure meets $\SC^{(2)}_{e}$.

Let $\gg(r,s)$ denote the subspace of $\gg$ where $\ad h$ has eigenvalue $r$ and $\ad h_0$ has eigenvalue $s$.
Define $\gg^f(r,s) = \gg(r,s) \cap \ker \ad f$.
Then $$\gg_2 \cap  \ker \ad f = \displaystyle \bigoplus_{r \geq 0} \gg^f(-r,r+2).$$ 
Next, we relate this decomposition to the decomposition \eqref{sl2_decomp} of $\gg$ under $\sg \oplus \sg_0$.
Let $$\EC = \{  i  \ | \  n_i > m_i >0 \text{ and } n_i - m_i \text{ even} \}$$
where $(m_i, n_i)$ are defined in \eqref{sl2_decomp}.   Then $\EC \subset \JC$.
For each $i \in \EC$, let $w_i$ be a nonzero vector in the one-dimensional space 
$V^{(i)}_{m_i, n_i} \cap \gg(-m_i, m_i+2)$.  
Then $w_i$ is a lowest weight vector for $\sg$, but not in general
a highest weight vector for $\sg_0$.   The set
$\{ w_i \ | \ i \in \EC\}$ is then a basis for
$$\displaystyle \bigoplus_{r \geq 1} \gg^f(-r,r+2)$$ 
since each vector in $\gg^f(-r, r+2)$ lies in a sum of subrepresentations of type $V_{r,s}$ with $r+2 \leq s$ and $s-r$ even.
The subspace $\gg^f(0,2)$ is just the $2$-eigenspace of $\ad h_0$ in $\cg(\sg)$, which coincides with 
$\cg(\sg) \cap \gg(0,2)$.  It contains $e_0$.
A consequence of the above observations is the following
\begin{lemma}  \label{slice_in_2_space}
Let $x \in \gg_2$.   If $e \in \overline{G_0 \cdot x}$, 
then some $G_0$-conjugate of $x$ can be expressed as
\begin{equation} \label{coeffs}
e + w + \sum_{i \in \EC} d_i w_i
\end{equation}
where $w \in \cg(\sg) \cap \gg(0,2)$ and $d_i \in \CM$.  
\end{lemma}

Given a nilpotent orbit $\0$, Lemma \ref{slice_in_2_space} gives a way to show the existence of 
some $x \in \0$ that can be written as in \eqref{x_form1} with $x_0 = w$ nilpotent.
But it does not guarantee that
$w$ is equal to the prescribed $e_0$ or that \eqref{dim_equality_full_e0} holds. 
When $w=e_0$ and \eqref{dim_equality_full_e0} holds, which is the case we are interested in,
then we also know by Lemma \ref{2+} that the $w_i$ which appear in \eqref{coeffs} with $d_i \neq 0$
must satisfy $n_i - m_i = 2$, so only the terms in $\EC$ with $n_i-m_i =2$ will contribute in this case.

\subsubsection{Applying Lemma \ref{slice_in_2_space}}  \label{ideal_containment}

In order to apply Lemma \ref{slice_in_2_space}
for some $x \in \gg$ with $\cl_e \subset \cl_x \subset \cl_y$,  
we need to check two things, after possibly replacing $x$ by a conjugate:
\begin{enumerate} 	
\item $x \in \gg_2$
\item $e \in \overline{G_0 \cdot x}$
\end{enumerate}

The first condition can often be shown as follows.  
Let $\sg_x$ be an $\sl_2$-subalgebra through some conjugate of $x$ with standard semisimple element $h^x \in \hg$.  
In all cases we are interested in, there exists nilpotent $e^x_0 \in \cg(\sg_x)$ with semisimple element $h^x_0 \in \hg$,
such that $h^x + h^x_0 = h_y$, after possibly replacing $x$ again by a conjugate.
Then just as in \eqref{lies_in_2_space}, $x \in \gg_2$ and the first condition holds.

We may further assume that $h_y$ is dominant with respect to the Borel subalgebra $\bg \subset \gg$ and 
$h^x$ is dominant for the corresponding Borel subalgebra $\bg_y$ of $\gg^{h_y}$.
Then since $[h^x, x] = 2x$ and $[h_y, x] = 2x$, it follows that $x$ belongs to 
$$I_x := \gg_2 \cap \bigoplus_{i \geq 2} \gg(h^x; i),$$ 
where $\gg(h^x;i)$ are the eigenspaces for $\ad h^x$.   This subspace 
is preserved by the action of $\bg_y$.  
Thus $G_0 \cdot I_x = \overline{G_0 \cdot x}$. 
We can carry out a similar process for $e$ and obtain a subspace $I_e \subset \gg_2$,
with $G_0 \cdot I_e = \overline{G_0 \cdot e}$. 
Then if $I_e \subset I_x$, it necessarily follows that 
$$G_0 \cdot I_e \subset G_0 \cdot I_x$$ and the second condition holds.
For the cases we are interested in, this approach will suffice to check the hypothesis in Lemma \ref{slice_in_2_space}.

\subsubsection{Example: $(\tilde{A}_1,A_1)$ in type $G_2$}  \label{mexample}

Let $\gg$ be of type $G_2$ and let $e \in \gg$ be minimal nilpotent.
Then $\cg(\sg) \cong \sl_2(\CM)$.  Let $e_0 \in \cg(\sg)$ be minimal nilpotent, which has 
type $\tilde{A}_1$ in $\gg$. 
The decomposition in \eqref{sl2_decomp} is  $\gg = V(0,2) \oplus V(2,0) \oplus V(1,3).$
Therefore $(m_i,n_i) = (1,3)$ for the unique $i \in \EC$ and \eqref{dim_equality_full_e0} fails for $\0 = \0_{e+e_0}$
by Proposition \ref{prop:codimensions_match}.  Indeed,  $e+e_0$ has type $G_2(a_1)$ and thus 
if $\0$ is the orbit of type $\tilde{A}_1$, then $\0_e \subset \cl \subset \cl_{e+e_0}$.  
Since $\sg$ and $\cg(\sg)$ are mutual centralizers, and $\cl$ is unibranch at $e$,
the argument in Example \ref{C3_example_1} gives that $\slice$ takes the form
of the $S$-variety $X(2,3)$ for $SL_2(\CM)$, which is isomorphic to $m$.

We can show also that Lemma \ref{slice_in_2_space} holds by checking the
two conditions in \S \ref{ideal_containment}.   
Fix nonzero $w_i \in V(1,3)$ satisfying $[e_0, w_i] = 0$ and $[f, w_i] = 0$.
Choose $h_y$ so that its weighted diagram is the usual weighted Dynkin diagram
$\begin{smallmatrix}  2 & 0 \end{smallmatrix}$ of $y$ and choose $h^x$ and $h^x_0$ to have weighted diagrams
$\begin{smallmatrix}  0 & 1  \end{smallmatrix} \text{ and } \begin{smallmatrix}  2 & -1 \end{smallmatrix},$ respectively.
Then $h_y = h^x + h^x_0$ and 
thus by the above discussion we may replace $x$ by a conjugate and assume $x \in \gg_2$.   
Similarly, let $h$ and $h_0$ have weighted diagrams 
$\begin{smallmatrix}  -1 & 1  \end{smallmatrix} \text{ and } \begin{smallmatrix}  3 & -1  \end{smallmatrix},$
respectively.  Then $I_e$  is one-dimensional and $I_e \subset I_x$.   
The two conditions in \S \ref{ideal_containment} are met, so by Lemma \ref{slice_in_2_space}
there exists $x \in \0$ with $x = e + ae_0 + dw_i$ for $a, d \in \CM$ .
Now $d \neq 0$ since $x$ and $e+ae_0$ are not in the same $G$-orbit for any value of $a$,
and $a \neq 0$ by Lemma \ref{oh_yeah}, 
and thus we get another proof that $\slice$ takes the form $X(2,3)$.

\subsubsection{Finding $w_i$ for $i \in \EC$}  \label{powers_of_X}

We sometimes need to do explicit computations to verify 
\eqref{coeffs} or to show that $w=e_0$, especially for degenerations
which are not minimal (e.g.,  \S \ref{F4_non_min}) or where Lemma \ref{S_var_sl2} does not apply.
In these cases there arises the need for an analogue of 
the result describing isomorphisms between $S$-varieties (\S \ref{S_vars}).
Here we describe a way that is often helpful in finding $w_i$ for $i \in \EC$, 
which frequently leads to an isomorphism of
$\ov{C(\sg) \cdot  x }$ with 
$\ov{C(\sg) \cdot  e_0}$ in Lemma \ref{kp_prop}, when such an isomorphism exists.

Write $\gg(h;j)$ for the $j$-eigenspace of $\ad h$.
Since $\cg(\sg) \subset \gg^h = \gg(h; 0)$, the $\gg(h;j)$ are $\cg(\sg)$-modules.
Also $\gg^h \oplus \gg(h;1)$ is isomorphic to $\gg^f$, as $\cg(\sg)$-modules.  
Indeed, for $j \geq 0$, 
\begin{eqnarray*}
\gg^f(-2j) \cong (\ad f)^j (\gg^h) \cap \gg^f \text{ and } \\
\gg^f(-2j-1) \cong (\ad f)^{j+1} (\gg(h;1)) \cap \gg^f,
\end{eqnarray*}
as $\cg(\sg)$-modules.

Suppose that $\gg^h$ is a direct sum of classical Lie algebras.
Then for $M \in \gg^h$, the matrix power $M^r$ is in $\gg^h$
for $r$ odd, or if all the factors of $\gg^h$ are type $A$, then for any $r$.
Of course $[M, M^r] = 0$ in $\gg^h$, and hence in $\gg$.
Set $M:=e_0$, where as before $e_0 \in \cg(\sg)$ is nilpotent.
The identity $[h_0, M^r] = 2r M^r$ holds in $\gg^h$ because $[h_0, M] = h_0M - Mh_0 = 2M$, where
matrix multiplication takes place in $\gg^h$; hence this identity also holds in $\gg$.   
Thus $M^r$, if nonzero, is a highest weight vector for $\sg_0$ relative to $e_0$ and $h_0$.
Now assume $M^r$ is not zero.  Then for some largest $j$, 
$$(\ad f)^j M^r \in \gg^f(-2j)$$ is nonzero.
Since $\sg$ and $\sg_0$ commute, $(\ad f)^j M^r$ is both a highest weight vector for $\sg_0$ 
and a lowest weight vector for $\sg$ (relative to $f$ and $h$).

Now suppose $\EC \neq \emptyset$
and consider $(m_i, n_i)$ for $i \in \EC$.  Suppose $m_i$ is even.
In the cases of interest (see Lemma \ref{2+} and the paragraph after Lemma \ref{slice_in_2_space}), we have $n_i - m_i = 2$.
In such cases we often find that $w_i$ can be taken to be 
$$(\ad f)^{\tfrac{m_i}{2}}( M^{\tfrac{m_i}{2}+1}).$$
Moreover, if $x$ in \eqref{coeffs}  is a linear combination of such $w_i$'s and $w=e_0$, then it follows that 
$\ov{C(\sg) \cdot x} \cong \ov{C(\sg) \cdot e_0}$ via the projection $\pi_0$
(\S \ref{natural_proj}) since the $G^h$-action, and thus the $C(\sg)$-action, commutes with taking matrix powers.

\subsubsection{Example:  the non-minimal degeneration $(C_3, \tilde{A}_2)$ in $F_4$}  \label{F4_non_min} 

We illustrate the previous discussion in $F_4$ in proving that $\slice$ contains an irreducible component isomorphic to the nilpotent cone $\NC_{G_2}$ in $G_2$, 
when $\0$ is of type $C_3$ and $e$ lies in the $\tilde{A}_2$ orbit.  For this choice of $e$, the centralizer $C(\sg)$ is connected, simple of type $G_2$.   
Let $e_0 \in \cg(\sg)$ be regular nilpotent.
Then $e + e_0$ lies in the orbit $F_4(a_2)$ and the decomposition of $\gg$ in \eqref{sl2_decomp} is 
$$V(0,2) \oplus V(0,10) \oplus V(2,0) \oplus V(4,6),$$ so $\EC$ has a single element, with $(m_i,n_i) = (4,6)$.
Then $\0_e \subset \cl \subset \cl_{e+e_0}$ and we could use \S \ref{ideal_containment} to show that there exists
$x \in \0$ satisfying \eqref{coeffs} with some additional work.  Instead, we report on a direct computation using GAP.
Let $$e=e_{0010}+e_{0001},\;\;\; f=2f_{0010}+2f_{0001},\;\;\; h=[e,f],$$ 
and
$$e_0 = e_{0111}-e_{0120} + e_{1000}.$$
The space $\gg(-4,6)$ is one-dimensional, spanned by $w_i := e_{1220}$.
This is also a highest weight vector for the full action of $C(\sg)$ on $\gg^f(-4) \cong V(\omega_2)$, the $7$-dimensional 
irreducible representation of $G_2$.
We computed in GAP that there is an $x \in \0$ with 
$x= e+e_0 -\tfrac{1}{4}w_i,$  which establishes \eqref{x_form1}
with $x_0 = e_0$ and $x_4 = w_i$.
Since $\dim \SC_{\0,e} = \dim C(\sg) \cdot e_0$, Lemma \ref{kp_prop} applies 
and thus $\slice$ contains $e + X$ as an irreducible component, 
where $X := \overline{C(\sg) \cdot (e_0 -\tfrac{1}{4}w_i})$.  

We now show that $X$ is isomorphic to $\overline{C(\sg) \cdot e_0}$, which is the nilcone of $\cg(\sg)$, 
by relating the choice of $w_i$ to the discussion in \S \ref{powers_of_X}.   
We have $\gg^h \cong \so_7(\CM) \oplus \CM$ and the $\so_7(\CM)$ component contains $\cg(\sg)$
and decomposes under $\cg(\sg)$ into $\cg(\sg) \oplus V(\omega_2)$. 
Now $\ad f$ annihilates $\cg(\sg)$, while $(\ad f)^2$ carries the $V(\omega_2)$ summand isomorphically onto $\gg^f(-4)$. 
Let $M = e_0 \in \cg(\sg) \subset \so_7(\CM)$.  
Then $M^3 \in \so_7(\CM)$ and $M^3 \neq 0$ since $e_0$ has type $B_3$ in $\gg$ 
(i.e., the embedding of $\cg(\sg)$ of type $G_2$ in $\so_7(\CM)$ is the expected one).
Since $M^3$ is centralized by $e_0$ and is an eigenvector for $\ad h_0$ with eigenvalue $6$, 
we have $M^3 \in V(\omega_2)$ since only the
eigenvalues $2$ and $10$ are possible for the $\cg(\sg)$ summand.  Moreover, 
$(\ad f)^2 (M^3)$ is a nonzero vector in $\gg(-4,6)$ and so must be a multiple of $w_i$.   
Although $X$ is not an $S$-variety (since $e_0$ is not minimal in $\cg(\sg)$), 
it is the closure of the $C(\sg)$-orbit through $(e_0, w_i) \in \cg(\sg) \oplus \gg^f(-4)$, which can now be described as
the set of elements $(M,M^3) \in \cg(\sg) \oplus V(\omega_2) = \so_7(\CM)$ with $M \in \cg(\sg)$ nilpotent.  
Hence, there is a $C(\sg)$-equivariant isomorphism 
of $X$ with $\overline{ C(\sg) \cdot e_0} \cong \NC_{G_2}$ coming from $\pi_0$.

\begin{remark}
There are two branches of $\cl$ in a neighborhood of $e$.  These two branches are not conjugate under the action of $G^e$, 
which shows that Proposition \ref{transitive_components} does not generally hold for degenerations which are not minimal.
The other branch of $\cl$ at $e$ splits into three separate branches in a neighborhood of a point in the orbit $F_4(a_3)$
(see \S \ref{4G2}).
\end{remark}

\subsection{Proof of Proposition \ref{dim4+_1}.} \label{section:dim4}

The proof is case-by-case until we exhaust all minimal degenerations covered by the Proposition.
First, we consider those $e$ for which there exists $e_0 \in \cg(\sg)$ that is minimal nilpotent in $\gg$,  
and then compute the $G$-orbit $\0$ to which $e+e_0$ belongs (\S \ref{Minimal in g}).  
Corollary \ref{slice_corollary} ensures that 
\eqref{dim_equality_full_e0} holds for this $\0$, and 
then applying Lemma \ref{kp_prop} to $x:= e + e_0$, we 
conclude that $e+ \overline{C(\sg) \cdot e_0}$ is a union of
irreducible components of $\slice$.   Such degenerations turn out always to 
be minimal degenerations, and so $C(\sg)$ acts transitively on the 
irreducible components of $\slice$ by Proposition \ref{transitive_components}.  
Hence $\slice = e+ \overline{C(\sg) \cdot e_0}$.     
The results are recorded in Tables \ref{table1_f4},  \ref{table1_e6}, \ref{table1_e7}, and \ref{table1_e8} 
for each of the exceptional groups $F_4$, $E_6$, $E_7$, and $E_8$, respectively.
Next, we consider all other cases where $e_0$ belongs to a minimal nilpotent $C(\sg)$-orbit in $\cg(\sg)$ 
and check whether or not \eqref{dim_equality_full_e0} holds for $\0 = \0_{e+e_0}$.
In the cases where it does hold, the degeneration $(\0, \0_e)$ turns out to be
a minimal degeneration, and thus $\slice= e + \overline{C(\sg) \cdot e_0}$ as in the first step.
The results are recorded in the first lines of Tables \ref{table4_f4},  \ref{table4_e6}, \ref{table4_e7}, and \ref{table4_e8}.
These two sets of calculations cover all the minimal degenerations in Proposition \ref{dim4+_1} where $J = \emptyset$.

For the remaining cases, we study those $e_0$ which are minimal in $\cg(\sg)$, but where
\eqref{dim_equality_full_e0} does not hold for the orbit through $e+e_0$.
For such $e$ and $e+e_0$, we look for nilpotent orbits $\0$ with $\0_e \subset \cl \subset \cl_{e+e_{\circ}}$
such that $\0_e$ is a minimal degeneration of $\0$ and $\dim C(\sg) \cdot e_0 =\codim_{\cl}( \0_e)$.
Then $\0' = \0_e$ and $\0$ are candidate orbits to apply Lemma \ref{kp_prop}.
For the cases where the degeneration is dimension two, which 
is the vast majority, we can show that Lemma \ref{S_var_sl2} (and hence Corollary \ref{A1_or_m}) applies.  
Sometimes, though, we have to restrict to a subalgebra as in Lemma \ref{singequi} or carry out 
a computer calculation to determine for which $i \in \mathbb N$ the corresponding $x_i$ is nonzero in \eqref{x_form1}.
There are just three others cases, all of dimension four, and for these
we can show that there exists $x \in \0$ satisfying \eqref{both_conditions} 
by restricting to a subalgebra as in Lemma \ref{singequi} (\S \ref{dim_4_case_1}, \S \ref{dim_4_case_2}).
Thus for all the remaining cases, which are the ones in the Proposition where  $J \neq \emptyset$, 
we find that  Lemma \ref{kp_prop} applies and 
Lemma \ref{2+} ensures that the $x_i$'s 
are highest weight vectors for $C(\sg)$ with weights as prescribed in the Proposition.
The possibilities for $J$ turn out to be  $\{2\}, \{2, 4\}, \{ 1 \}, \{1, 2 \},$ and $\{1, 2, 3\}$, as noted in 
Remark \ref{which_J}.  By \S \ref{S_vars}, the first two possibilities give the isomorphism under $\pi_0$ in (1) of the Proposition,
and the last three possibilities give the isomorphism under $\pi_{0,1}$ in (2) of the Proposition.

Comparing with the surface cases treated in \S \ref{section:codimension2}, in order to know which surface cases have $|\Gamma| = 1$ or $2$,
we find that all the cases in Proposition \ref{dim4+_1} have been addressed.  
The results are recorded in Tables \ref{table4_f4},  \ref{table4_e6}, \ref{table4_e7}, and \ref{table4_e8},
where $\EC \neq \emptyset$.  The set $J$ consists of those $m_i$ with $i \in \EC$ and $d_i \neq 0$ in \eqref{coeffs}, or
equivalently, $x_{m_i} \neq 0$ in \eqref{x_form1}. Such $m_i$
are the ones in the boldface pairs $(m_i, n_i)$ in these tables.  They all must satisfy $n_i-m_i =2$ by Lemma \ref{2+}.



\section{Geometric method for surface singularities}  \label{section:codimension2}

In this section we consider a minimal degeneration $\0'$ of $\0$
such that  $\0'$ is of codimension 2 in $\cl$.  Let $e \in \0'$.  
We show that the normalization of each irreducible component of 
$\slice$  is isomorphic to 
$\cit^2/\Gamma$ for some finite subgroup $\Gamma \subset {\rm SL}_2(\CM)$.
Our method allows us to determine the group $\Gamma$, hence 
we determine $\slice$ up 
to normalization.   As mentioned in \S \ref{summary:surface_cases}, we can often use results on normality of
nilpotent orbit closures or other methods (e.g. Lemma \ref{singequi}) to decide whether the irreducible components of $\slice$ are normal.
Sometimes we have to state our results up to normalization.

\subsection{Two-dimensional Slodowy slices}

Recall that a contracting $\cit^*$-action on a variety $X$ is a $\cit^*$-action on $X$ with a unique fixed point $o \in X$ such that for any 
$x\in X$, we have $\lim_{\lambda \to 0} \lambda \cdot x = o$.  Recall from \cite{Beauville} that a symplectic variety is a normal
variety $W$ with a holomorphic symplectic form $\omega$ on its smooth
locus such that for any resolution $\pi: Z \to W$, the pull-back
$\pi^* \omega$ extends to a regular 2-form on $Z$. 
For a nilpotent orbit, we write $\widetilde{\0}$ for the normalization of $\cl$.

\begin{lemma}
The normalization $\widetilde{\mathcal{S}}_{\0, e}$ of $\mathcal{S}_{\0, e}$ is an affine normal variety with each irreducible component having at most an isolated symplectic singularity and endowed with a contracting $\cit^*$-action.
\end{lemma}
\begin{proof}
As $\widetilde{\0}$ has rational Gorenstein singularities by \cite{Hinich} and \cite{Panyushev},   $\widetilde{\mathcal{S}}_{\0, e}$ 
has only rational Gorenstein singularities.  On the other hand, there exists
a symplectic form on its smooth locus, hence  $\widetilde{\mathcal{S}}_{\0, e}$ has only symplectic singularities by \cite{Namikawa} (Theorem 6).
By construction, the contracting $\cit^*$-action on $\slice$ in \S \ref{group action} has positive weights, 
hence it lifts to a contracting $\cit^*$-action on  $\widetilde{\mathcal{S}}_{\0, e}$.
\end{proof}

The two-dimensional symplectic singularities are exactly rational double points (cf. \!\cite[Section 2.1]{Beauville}).   
The following is immediate from  \cite[Lemma 2.6] {Flenner-Zaidenberg}.
\begin{proposition} \label{P:2-dim}
Let $X$ be an affine irreducible  surface with an isolated rational double point at $o$. If there exists a contracting $\cit^*$-action on $X$,
then  $X$ is isomorphic to $\cit^2/\Gamma$ for some finite subgroup $\Gamma \subset {\rm SL}_2(\CM)$.
\end{proposition}

Note that by Proposition \ref{transitive_components}, the irreducible components of $\slice$ are mutually isomorphic. 
As an immediate corollary, we get 
\begin{corollary} \label{Cor:2-dim-isom}
Let $\slice$ be a two-dimensional nilpotent Slodowy slice.
Then there exists a finite subgroup  $\Gamma \subset {\rm SL}_2(\CM)$ 
such that each irreducible component of the normalization $\widetilde{\mathcal{S}}_{\0, e}$ is isomorphic to 
$\cit^2/\Gamma$.
\end{corollary}

Hence to determine $\widetilde{\mathcal{S}}_{\0, e}$, we only need to determine the subgroup $\Gamma$. In the following, we shall
describe a way to construct the minimal resolution of $\widetilde{\mathcal{S}}_{\0, e}$. Then the configuration of exceptional $\pit^1$'s
in the minimal resolution will determine $\Gamma$.

\subsection{$\mathbb{Q}$-factorial terminalization for nilpotent orbit closures}\label{q_fact_term}
A general reference for the minimal model program in algebraic geometry is 
\cite{Matsuki}.
Here we recall some basic definitions.

Let $X$ be a normal variety.  A Weil divisor $D$ on $X$ is called {\em $\mathbb{Q}$-Cartier} if $N D$ is a Cartier divisor for some non-zero integer $N$. We say that $X$ is {\em $\mathbb{Q}$-Gorenstein} if its canonical divisor
$K_X$ is $\mathbb{Q}$-Cartier. 
The variety $X$ is called {\em $\mathbb{Q}$-factorial} if every Weil divisor on $X$ is $\mathbb{Q}$-Cartier. 
A $\mathbb{Q}$-Gorenstein variety $X$ is said to have {\em terminal singularities}
if  there exists a resolution $\pi: Z \to X$ such that
$K_Z = \pi^* K_X + \sum_{i=1}^k a_i E_i$ with $a_i >0$ for all $i$, where $E_i, i=1, \cdots, k$ are the irreducible components of the exceptional divisor of $\pi$.
A {\em $\mathbb{Q}$-factorial terminalization} of a $\mathbb{Q}$-Gorenstein variety $X$ is a projective birational morphism $\pi: Z \to X$ such that  $K_Z = \pi^* K_X$ and $Z$ is $\mathbb{Q}$-factorial with  only terminal singularities. 

It is well-known that two-dimensional terminal singularities are necessarily smooth (cf. Theorem 4-6-5 \cite{Matsuki}), hence a normal variety $X$ with only terminal singularities is smooth in codimension 2, that is, ${\rm codim}_X {\rm Sing}(X) \geq 3$.

For the normalization of the closure of a nilpotent orbit, 
one way to obtain its $\mathbb{Q}$-factorial terminalization
is by the following method. 
Consider a parabolic subgroup $Q$ in $G$. Let $L$ be a Levi
subgroup of $Q$. For a nilpotent element $t \in \Lie(L)$, we
denote by $\OC^L_t$ its orbit under $L$ in $\Lie(L)$. Let $\ng(\qg)$
be the nilradical of $\Lie(Q)$.   Then the natural map 
$p: G \times^Q (\ng(\qg)+\overline{\OC}^L_t) \to \gg$
has image equal to $\cl$ for some nilpotent orbit $\OC$ and $p$ is
called a {\em generalized Springer map} for $\OC$.
Then $\OC$ is said to be induced from $(L, \OC^L_t)$ \cite{Lusztig-Spaltenstein}.
When $t=0$,  then $\0$ is called the Richardson orbit for $Q$ and
$G \times^Q \ng(\qg)$ identifies with the cotangent bundle $T^*(G/Q)$;
if in addition $p$ is birational, then we call $p$ a generalized Springer resolution.
By \cite{Fu:sympl}, those are the only symplectic resolutions of nilpotent orbit closures.
More generally, if $p$ is birational and the normalization of $\overline{\OC}^L_t$ is $\QM$-factorial terminal, then
the normalization of $p$ gives a $\QM$-factorial
terminalization of  $\widetilde{\0}$, the normalization of $\cl$.   In \cite{Fu:Qfact}, it was proved in confirming
a conjecture of Namikawa that for a nilpotent orbit $\0$ in an exceptional Lie algebra, 
either $\widetilde{\0}$ is $\qit$-factorial terminal 
or every $\qit$-factorial terminalization of $\cl$ is given by a generalized Springer map.

\subsection{Minimal resolutions of two-dimensional nilpotent Slodowy slices}  \label{minimal_resolution}

We now use the generalized Springer maps to construct a minimal resolution of $\widetilde{\mathcal{S}}_{\0, e}$
when $\slice$ is two-dimensional.

Recall from \cite{Fu:Qfact} that in a simple Lie algebra of exceptional
type, 
$\widetilde{\0}$
has only terminal singularities if and only if $\0$ is either a
rigid orbit or it belongs to the following list: $2A_1, A_2+A_1,
A_2+2A_1$ in $E_6$; $A_2+A_1$, $A_4+A_1$ in $E_7$; $A_4+A_1$,
$A_4+2A_1$ in $E_8$.

First consider the case where $\widetilde{\0}$ has only terminal singularities.
Then  $\widetilde{\0}$
is smooth in codimension two by the previous subsection. 
This implies that the
singularities of $\cl$ along $\OC_e$ are smoothable by its
normalization.  In other words, $\widetilde{\mathcal{S}}_{\0, e}$ is smooth, 
which is then isomorphic to $\cit^2$ by Proposition \ref{P:2-dim} and we are done.  

\begin{example}
Consider again the minimal degeneration $(\widetilde{A}_1, A_1)$ in $G_2$ from \S \ref{mexample}. 
As $\0=\0_{\widetilde{A}_1}$ is a rigid orbit, 
its normalization has $\QM$-factorial terminal
singularities by \cite{Fu:Qfact}.
In particular, the singular locus of
$\widetilde{\0}$ has codimension at least $4$.  Since the orbit
$A_1$ is of codimension two in $\cl$, this implies
that $\cl$ is non-normal and  $\widetilde{\mathcal{S}}_{\0, e} \cong \CM^2$
for $e \in \0_{A_1}$, which is consistent with the description $\slice \cong m$ in \S \ref{mexample}.
\end{example}

Next, assume that the normalization $\widetilde{\0}$ is not terminal.  Then by \cite{Fu:Qfact},
$\0$ is an  induced orbit and $\widetilde{\0}$ admits a $\QM$-factorial terminalization $\pi: Z \to
\widetilde{\0}$ given by the normalization of a generalized Springer
map.
We denote by $U$ the open subset $\0 \cup \OC_e$
of $\cl$ and $\nu: \widetilde{U} \to U$ the normalization map.
As $Z$ has only terminal singularities, it is smooth in codimension two.
As $\pi$ is $G$-equivariant and $\OC_e \subset \cl$ is of codimension two, we get that $\pi(\Sing(Z)) \cap
\nu^{-1}(\OC_e) =\emptyset$.
We deduce that $V:=\pi^{-1}(\widetilde{U})$ is smooth.
In particular, we obtain a symplectic resolution $\pi|_V: V \to \widetilde{U}$.
By restriction, 
we get a resolution $\pi: \pi^{-1}(\widetilde{\SC}_{\0, e}) \to \widetilde{\SC}_{\0, e}$, which is a symplectic, hence minimal, resolution.

Let $y \in  \nu^{-1}(e)$.  If we know:
(1) the number of $\mathbb{P}^1$'s in $\pi^{-1}(y)$ and in $\pi^{-1}(\nu^{-1}(e))$; and
(2) the action of $A(e)$ on the $\mathbb{P}^1$'s in $\pi^{-1}(\nu^{-1}(e))$, then in most cases we
can determine the configuration of $\mathbb{P}^1$'s in $\pi^{-1}(\nu^{-1}(e))$, and hence in $\pi^{-1}(y)$, and therefore determine
$\widetilde{\SC}_{\0, e}$.  We next introduce some methods to compute this information.

\subsection{The method of Borho-MacPherson}   \label{Bo-M method} 

 Let $W$ be the Weyl group of $G$. The Springer
correspondence assigns to any irreducible $W$-module a unique pair
$(\OC, \phi)$ consisting of a nilpotent orbit $\OC$ in $\gg$ and
an irreducible representation $\phi$ of the component group $A(x)$
where $x \in \OC$.
 The corresponding irreducible $W$-module will be denoted by  $\rho_{(x, \phi)}.$

Let $W_L$ denote the Weyl group of $L$, viewed as a subgroup of $W$.
Let $\BC_x$ denote the Springer fiber over $x$ for the resolution
of the nilpotent cone $\NC$ in $\gg$ and let $\BC^L_t$ be the Springer fiber of
$t$ for the group $L$.  
If $\0^L_t$ is the orbit of $L$ through the nilpotent element $t \in \mbox{Lie}(L)$, we denote by $\rho^{\tiny L}_{(t, 1)}$
the $W_L$-module corresponding to the pair $(\0^{\small L}_t, 1)$ via the Springer correspondence for $L$.

\begin{lemma}\label{formula}
Let $Z = G \times^Q (\ng(\qg)+\overline{\OC}^L_t)$. 
Let $p: Z \to \cl$ be the generalized Springer map.  Let $\OC' \subset \cl$ be a
nilpotent orbit of codimension $2d$.  Assume that $Z$ is rationally
smooth at all points of $p^{-1}(e)$ for $e \in \OC'$. Then the
number of irreducible components of $p^{-1}(e)$ of dimension $d$
is given by the formula
\[
\frac{\deg \rho^L_{(t,1)}}{\dim H^{\mathrm top}(\BC^L_t)}
\bigoplus_{\phi\in \Irr A(e)} \deg \phi \cdot [\Res^W_{W_L} \rho_{(e,\phi)} : \rho^{\tiny{L}}_{(t,1)}],
\]
where the sum is over
the irreducible representations $\phi$ of $A(e)$ appearing in the Springer correspondence for $G$.
\end{lemma}

\begin{proof}
By \cite[Thm. 3.3]{Borho-MacPherson}, we have $H^{\mathrm top} (p^{-1}(e)) \otimes
H^{\mathrm top}(\BC^L_t) \cong H^{\mathrm top}(\BC_e)^{\rho^{\tiny{L}}_{(t,1)}}$, where the
right hand side denotes
the $\rho^L_{(t,1)}$-isotypical component of the restriction of $H^{\mathrm top}(\BC_e)$ to $W_L$.
Recall that $H^{\mathrm top}(\BC_e) = \oplus_\phi \rho_{(e,\phi)} \otimes \phi$, which gives
\[
h^{\mathrm top}(p^{-1}(e)) \cdot h^{\mathrm top} (\BC^L_t) =
\deg \rho^L_{(t,1)} \sum_\phi \deg \phi \cdot [\Res^W_{W_L} \rho_{(e,\phi)} : \rho^L_{(t,1)}].
\]
where $h^{\mathrm top}(X)$ denotes the dimension of $H^{\mathrm top}(X)$.
\end{proof}

Now the component group $A(e)$ acts on the left-hand side of
$$
H^{\mathrm top} (p^{-1}(e)) \otimes H^{\mathrm top}(\BC^L_t) \cong H^{\mathrm top}(\BC_e)^{\rho^L_{(t,1)}}
$$
where it acts trivially on $H^{\mathrm top}(\BC^L_t)$.
It also acts on the right-hand side since the $A(e)$-action commutes with the $W$-action, and hence the
$W_L$-action.
Note that the action of $A(e)$ is compatible with the isomorphism (see Corollary  3.5 \cite{Borho-MacPherson}). 
This gives the following
\begin{corollary}\label{cor_action}
The permutation action of $A(e)$ on the irreducible components of dimension $d$ of $p^{-1}(e)$ gives rise 
to the linear representation 
\begin{equation}\label{perm_rep}
\bigoplus_{\phi\in\Irr A(e)} \deg \rho^L_{(t,1)} [\Res^W_{W_L} \rho_{(e,\phi)} : \rho^L_{(t,1)}]  \phi
\end{equation}
\end{corollary}
In particular the number of orbits of $A(e)$ on the irreducible components of $p^{-1}(e)$ of 
dimension $d$ equals the multiplicity of the trivial representation of $A(e)$ in \eqref{perm_rep}.
The number of $A(e)$-orbits is therefore equal to
$\deg \rho^L_{(t,1)}  [\Res^W_{W_L} \rho_{(e,1)} : \rho^L_{(t,1)}].$  

\begin{example} \label{example:B3inF4}
Let $\gg$ be of type $F_4$.   Let $\0$ be the nilpotent orbit of type $B_3$ and $\0'$  of type $F_4(a_3)$.
Then $\0' \subset \cl$ is codimension two.
Since $\0$ is even, its weighted Dynkin diagram 
shows that $\0$ is Richardson for the parabolic subgroup $Q$
with Levi subgroup $L$ of semisimple type $\widetilde{A}_2$.
This gives rise to the generalized Springer map $p: G \times^Q \ng(\qg) \to \cl$
as in Lemma \ref{formula}, with $t=0$.
The map  $p$ is birational because $e$ is even.
Since $\cl$ is normal and $p$ is birational, the restriction of $p$
gives a minimal resolution of $\widetilde{\SC}_{\0, e} = \slice$
where $e \in \0'$ as in \S \ref{minimal_resolution}.

Now $A(e)=\mathfrak{S}_4$.  
Since $t=0$, the representation $\rho^L_{(t, 1)}$ is the sign representation of $W_L$.
By the Springer correspondence for $F_4$,  
$\rho_{(e, [21^2])} = \phi'_{1, 12}$, $\rho_{(e, [2^2])} = \phi''_{6,6}$,
$\rho_{(e, [31])} = \phi'_{9,6}$ and $\rho_{(e,[4])} =\phi_{12, 4}$ (see \cite[pg. 428]{Carter}).
The multiplicity of the sign representation in the restriction of $\rho_{(e, [2^2])}$ to $W_L$
is 1 and in the restriction of $\rho_{(e, [4])}$ is 2 and it is zero otherwise.
By  Lemma \ref{formula}, the number of $\mathbb{P}^1$'s in $p^{-1}(e)$ 
is $1 \cdot 2+ 2 \cdot 1 = 4$ and by Corollary \ref{cor_action}, the group 
$A(e)$ fixes one component and permutes the remaining three components transitively.
Consequently the dual graph of $\slice = \widetilde{\SC}_{\0, e}$ is the Dynkin diagram of type $D_4$
and $A(e)$ acts on the dual graph via the unique quotient of $A(e)$ isomorphic to $\mathfrak{S}_3$.
Hence the singularity is $G_2$.  

The fact that the dual graph is $D_4$ 
could also be obtained by restricting to a maximal subalgebra of type $B_4$ (\S \ref{passing_to_subalgebra}).  
In this way we would only need to know 
that the degeneration in $F_4$ is unibranch, instead of the stronger statement that $\cl$ is normal.
\end{example}

\subsection{Orbital varieties and the exceptional divisor of $\pi$}  \label{orbital_exceptional}

The next lemma (see  \cite[Lemma 4.3]{Fu:Qfact})  can sometimes be used to simplify computations.

\begin{lemma} \label{qfact}
Let $\0$ be a nilpotent orbit with $\Pic(\OC)$ finite and such that there is a generalized Springer resolution
$\pi: G \times^Q \ng(\qg) \to \cl$ (see \S \ref{q_fact_term}).
Then the number of irreducible exceptional divisors of $\pi$ equals $b_2(G/Q)$, the second Betti number of $G/Q$, which is equal to
 the rank of $G$ minus the semisimple rank of a Levi subgroup of $Q$.
\end{lemma}

From \cite[Prop 4.4]{Fu:Qfact} it follows that $\Pic(\OC_x)$ is finite 
whenever the character group of $G^x$ is finite.  Picking an $\sl_2$-subalgebra $\sg_x$ containing $x$,
the latter is equivalent to the finiteness of the character group of $C(\sg_x)$, or equivalently, 
to the finiteness of the center
of $C(\sg_x)$.  The latter can be read off from the tables in
 \cite{alexeevsky} or deduced from the tables in \cite{Sommers:Bala-Carter}.  
 Such calculations are closely related to those in \S \ref{section:split_and_intrinsic}.
In the exceptional groups, $\Pic(\OC)$ is finite unless $\0$ is one of the following orbits in $E_6$:  $2A_1, A_2+A_1, A_2+2A_1, A_3, A_3+A_1, A_4, A_4+A_1, D_5(a_1), D_5$. 
For these orbits in $E_6$, the number of irreducible exceptional divisors of a generalized Springer resolution or a $\mathbb{Q}$-factorial terminalization has been explicitly computed in the proof of \cite[Prop 4.4]{Fu:Qfact}.

Let $\0_1, \dots \0_s$ be the maximal orbits in the complement of $\0$ in $\cl$.  We restrict to the case where all $\0_i$'s are 
codimension two in $\cl$.  
Then the irreducible exceptional divisors of $\pi$ have a description in terms 
of the orbital varieties for the $\0_i$'s.
Recall that an orbital variety for $\0_i$ is an irreducible component of 
$\cl_i \cap \ng$ where $\ng:=\ng(\bg)$ is the nilradical of the Borel subalgebra $\bg$.
It is known that each orbital variety has dimension $\tfrac{1}{2} \dim \0_i$.
Let $X$ be an orbital variety for $\0_i$ which is contained in $\ng(\qg)$.
Then $X$ is of codimension one in $\ng(\qg)$ 
since $\0_i$ is of codimension two in $\cl$ and $\dim \ng(\qg) = \tfrac{1}{2} \dim \0$.
Moreover $X$ is stable under the action of the connected group $Q$ 
since $X \subset Q \cdot X \subset \cl_i \cap \ng$ and $X$ is maximal irreducible in $\cl_i \cap \ng$.

Let $\pi_X$ be the restriction of $\pi$ to $G \times^Q  X$.
The image of $\pi_X$ is $\cl_i$ since $X$ is irreducible and $Q$ is a parabolic.
By dimension considerations, $\pi_X^{-1}(\cl_i) = G \times^Q  X$ is an irreducible exceptional divisor  of $\pi$. 
Conversely, any irreducible  exceptional divisor  of $\pi$ equals
$G \times^Q Y$  for some irreducible component $Y$ of $\cl_i \cap \ng(\qg)$.    Now $\dim Y$ can only
equal $\dim \ng(\qg) -1$ or $\dim \ng(\qg) -2$ since $\im \pi_Y = \cl_i$.  In the former case,
$Y$ is an orbital variety of $X$ contained in $\ng(\qg)$.  In the latter case, $\pi^{-1}_Y(e_i)$ is finite where $e_i \in \0_i$, 
contradicting the fact, from above, that the irreducible components of $\pi^{-1}(e_i)$ are $\PM^1$'s.   
This shows that the irreducible exceptional divisors  of $\pi$ are exactly the $G \times^Q  X$ 
where $X$ is an orbital variety of some $\0_i$ lying in $\ng(\qg)$.

Next, the map $G \times^B X \to G \times^Q  X$ has connected fibers isomorphic to $Q/B$.  It follows from \cite{Spaltenstein}
that the $\PM^1$'s in $\pi_X^{-1}(e_i)$ are permuted transitively under the induced action of $A(e_i)$ since the analogous statement
holds for the irreducible components of $p_X^{-1}(e_i)$ where $p_X: G \times^B X \to \NC$.
Consequently, if  $\Pic(\OC)$ is finite and $r_i$ equals the number of $A(e_i)$-orbits on $\pi^{-1}(e_i)$,
then $\sum r_i = b_2(G/Q)$ by Lemma \ref{qfact}.   See, for example, \cite[Thm 1.3]{Wierzba}) for a more general setting where this phenomenon occurs.

\begin{example} 
Consider the minimal degeneration where $\0$ has type $\widetilde{A}_2$ and $\0'$ has type $A_1+\widetilde{A_1}$ in $F_4$. 
The codimension of $\0'$ in $\cl$ is two.  
The orbit $\0$ is Richardson for the parabolic subgroup $Q$ whose Levi subgroup has type $B_3$.
Moreover the map $\pi: Z:= G \times^Q \ng(\qg) \to \cl$ is birational, hence a generalized Springer resolution.
The hypotheses of Lemma \ref{qfact} hold.
Since $b_2(Z)=1$ and there is no other minimal degeneration of $\0$, there must be exactly one
irreducible component in $\pi^{-1}(\overline{\0'})$. 
Since $A(e)=1$ for $e \in \0'$, there is only one irreducible component in $\pi^{-1}(e)$.
Since $\cl$ is normal, the singularity of $\cl$ at $e$ is of type $A_1$.
\end{example}

\subsection{Three remaining cases}  \label{orbital_varieties_I}

There are three cases where the 
information in Lemma \ref{formula} and Corollary \ref{cor_action} is not sufficient
to determine a minimal surface degeneration, even up to normalization.
They are $(E_6(a_1), D_5)$ in $E_6$, $(E_7(a_1), E_7(a_2))$ in $E_7$, and $(E_8(a_1), E_8(a_2))$ in $E_8$.
In this section we give an ad hoc way to determine the singularity.   

In each of the three cases, the larger orbit $\0$ is the subregular nilpotent orbit and so $\cl$ is normal.  
Since $\gg$ is simply-laced, $A(x)$ is trivial 
for $x \in \0$.  Hence for any parabolic subgroup
$Q$ with Levi factor $A_1$ the map $\pi: G \times^Q \ng(\qg) \to \cl$ is birational.
Moreover in each case the smaller orbit $\0'$ is the unique maximal orbit in $\cl \backslash \0$.  Since $A(e)=1$
for $e \in \0'$, there are $\mbox{rank}(\gg) -1$ $\PM^1$'s in $\pi^{-1}(e)$ by \S \ref{orbital_exceptional}.
At the same time, this uniqueness means that 
$\0'$ is the Richardson orbit for any parabolic $Q'$ with Levi factor of semisimple type $A_1 \times A_1$, so if $Q'$
is such a parabolic, then
$\ng(\qg')$ is an orbital variety for $\0'$.   Hence if we fix $Q$ corresponding to a simple root $\alpha$, then we find
an orbital variety $\ng(\qg') \subset \ng(\qg)$ for $\0'$ for each simple root $\beta$ not connected to $\alpha$ in the Dynkin diagram.
Since $A(e)$ is trivial, each of these $\ng(\qg')$ gives rise to a unique $\PM^1$ in $\pi^{-1}(e)$.  
By looking in the Levi subalgebra corresponding to the simple roots not connected to $\alpha$, 
it is possible to determine the intersection pattern of these $\PM^1$'s.

\subsubsection{ The case of $(E_6(a_1), D_5)$ in $E_6$} \label{example: E6_subregular}

There are $5$ $\PM^1$'s in $\pi^{-1}(e)$.  The singularity could only be $A_5$ or $D_5$ since $\cl$ is normal.
If we choose $\alpha$ so that the remaining simple roots form a root system of type $A_5$, then there are 4 orbital varieties
of the form $\ng(\qg')$ in $\ng(\qg)$.  The $4$ $\PM^1$'s have intersection diagram of type $A_2 + A_2$.  
This could only happen for a dual graph of type $A_5$, so $\slice \cong A_5$.

\subsubsection{ The case of $(E_7(a_1), E_7(a_2))$ in $E_7$} \label{example: E7_subregular}

There are $6$ $\PM^1$'s in $\pi^{-1}(e)$.  The singularity could only be $A_6$, $D_6$, or $E_6$ since $\cl$ is normal.
Choosing $\alpha $ so that the remaining simple roots form a system of type $E_6$, there are 5 orbital varieties
of the form $\ng(\qg')$ in $\ng(\qg)$.  Then the $5$ $\PM^1$'s have intersection diagram of type $D_5$.  This eliminates $A_6$ as a possibility.   
If we choose $\alpha $ so that the remaining simple roots form a system of type $A_6$, then there are 5 orbital varieties
of the form $\ng(\qg')$ in $\ng(\qg)$ and the corresponding $5$ $\PM^1$'s have intersection diagram of type $A_2+A_3$.  
This eliminates $E_6$, hence $\slice \cong  D_6$.

\subsubsection{ The case of $(E_8(a_1), E_8(a_2))$ in $E_8$} \label{example: E8_subregular}

There are $7$ $\PM^1$'s in $\pi^{-1}(e)$.  The singularity could only be $A_7$, $D_7$, or $E_7$ since $\cl$ is normal.
If we choose $\alpha $ so that the remaining simple roots form a system of type $E_7$, then there are 6 orbital varieties
of the form $\ng(\qg')$ in $\ng(\qg)$.  The corresponding $6$ $\PM^1$'s have intersection diagram of type $E_6$.  
Hence $\slice \cong E_7$.

\begin{remark}
Ben Johnson and the fourth author have also confirmed these three results using Broer's description of the ideal defining the closure of the subregular nilpotent orbit and the Magma algebra system.
\end{remark}

\section{On the splitting of $C(\sg)$ and intrinsic symmetry action}   \label{section:split_and_intrinsic}

\subsection{The splitting of $C(\sg)$}   \label{C-splitting}
In this section we establish the splitting on $C(\sg)$ discussed in \S \ref{symmetry_action}.  
That is, we determine when
$$C(\sg) \cong C(\sg)^{\circ} \rtimes H$$
for some $H \subset C(\sg)$. Necessarily $H \cong A(e)$.  We continue to assume that $G$ is of adjoint type.

In the classical groups, $C(\sg)$ is a product of orthogonal groups and a connected group, possibly up to a quotient by a central subgroup of order two.   Since the result holds for any orthogonal group, it holds for $C(\sg)$.

Let $\CC  \subset A(e)$ be a conjugacy class.  
There exists $s \in C(\sg)$ whose image $\bar{s}$ in $A(e)$ 
lies in $\CC$ such that the order of $s$ equals the order of $\bar{s}$, except when 
$e$ belongs to one the following four orbits:
\begin{equation} \label{the_four}
A_4+A_1 \text{ in } E_7;  \qquad A_4+A_1,  D_7(a_2) \text{ and } E_6(a_1)+A_1 \text{ in } E_8.
\end{equation}
For these four orbits, which all have $A(e) = \SG_2$, the best result is an $s$ of order $4$ to represent the non-trivial $\CC$  in $A(e)$ 
\cite[\S 3.4]{Sommers:Bala-Carter}.   Hence the splitting holds for all other orbits where $A(e) = \SG_2$, with $H = \{1, s\}$.

This leaves the cases where $A(e) = \SG_3, \SG_4,$ or $\SG_5$.   If $e$ is distinguished, meaning $C(\sg)^{\circ} =1$, there is nothing to check.
This leaves a handful of cases where $A(e) = \SG_3$ and $e$ is not distinguished.   The first such case is $e = D_4(a_1)$ in $E_6$, which we now explain.

\subsubsection{$\SG_3$ cases}
 Let $G$ be of type $E_6$ and $s \in G$ be an involution  with $G^s$ of semisimple type $A_5+A_1$.   
Then there exist $\tilde{e} \in \gg^s$ nilpotent of type $2A_2$.  
Let $\tilde{\sg} \subset \gg^s$ be an $\sl_2$-triple through $\tilde{e}$.
Then $\cg(\tilde{\sg})$ has type $G_2$.  
It is easy to compute $\gg^s \cap \cg(\tilde{\sg})$ inside of $A_5+A_1$; it is a semisimple subalgebra of type $A_1+A_1$.
Let $\tilde{e}_0$ be regular nilpotent in $\gg^s \cap \cg(\tilde{\sg})$.   Then $\tilde{e}_0$ is in the subregular nilpotent orbit in $\cg(\tilde{\sg})$.
Clearly $s$ belongs to the centralizer of $\tilde{e}_0$ in $C(\tilde{\sg})$, 
which is a finite group $H \cong \SG_3$, from the case of the subregular orbit in $G_2$.
Next, a calculation in $A_5+A_1$ shows that $\tilde{e}+\tilde{e}_0$ has 
generalized Bala-Carter type $A_3+2A_1$.  From this we conclude that 
$e = \tilde{e}+\tilde{e}_0$ belongs to the nilpotent orbit $D_4(a_1)$ in $E_6$ and $s$ represents an involution in $A(e)$ 
\cite[\S 4]{Sommers:Bala-Carter}.

A similar argument works if $s \in G$ is an element of order $3$ with $G^s$ of semisimple type $3A_2$.   Therefore the centralizer $H \cong \SG_3$ of 
$\tilde{e}_0$ in $C(\tilde{\sg})$ also centralizes $\tilde{e}+\tilde{e}_0$ and the image of $H$ in $A(e)$ is all of $A(e)$.  
This proves the splitting for $e = D_4(a_1)$ in $E_6$.
The same procedure works for the other $\SG_3$ cases.

\subsubsection{}
We have shown
\begin{proposition}  \label{prop:split}
There exists $H \subset C(\sg)$ such that
$$C(\sg) \cong C(\sg)^{\circ} \rtimes H,$$
except when $e$ belongs to one of the four orbits in \eqref{the_four}.
For those four cases, $A(e) = \SG_2$ and 
$$C(\sg) = C(\sg)^{\circ} \cdot H$$
where $H \subset C(\sg)$ is cyclic of order $4$.
\end{proposition}

While the above splitting is unique up to conjugacy in $C(\sg)$ in the subregular case (\S \ref{symmetry_nsl_case}), 
this is not the case in general, as the next example shows.

\begin{example}
Let $e$ be in the $A_2$ orbit in $\gg$ of type $E_8$.  Then $\cg(\sg)$ has type $E_6$ and $A(e) = \SG_2$.  The generalized Bala-Carter notation for the non-trivial class $\CC$ in $A(e)$ is $(4A_1)''$.  From this it follows that both conjugacy classes of involution in $G$ can represent $\CC$.
For one choice of involution $s_1 \in C(\sg)$ lifting $\CC$, $\gg^{s_1}$ has type $D_8$.  The partition of $e$ in $\gg^{s_1}$ is $[2^8]$, so the reductive centralizer of $e$ in $\gg^{s_1}$ is $\sp_8$.   For the other choice $s_2 \in C(\sg)$ lifting $\CC$, 
$\gg^{s_2}$ has type $E_7+A_1$ and $e$ corresponds to $(3A_1)'' + A_1$.  Hence the reductive centralizer of $e$ in $\gg^{s_2}$ is of type $F_4$.
Consequently, there are two choices of splitting in Proposition \ref{prop:split} that are not only non-conjugate under $C(\sg)$, but also in $\Aut(\cg(\sg))$.
\end{example}

Although the choice of splitting in Proposition \ref{prop:split} is not unique up to conjugacy in $C(\sg)$ or even $\Aut(\cg(\sg))$, 
we can restrict the choice of $H$ further so that the image of $H$ in $\Aut(\cg(\sg))$
will be well-defined up to conjugacy in $\Aut(\cg(\sg))$.
Let $\cg(\sg)^{\mathrm{ss}}$ be the semisimple summand of $\cg(\sg)$.
Let $$a: C(\sg) \to \Aut(\cg(\sg)^{\mathrm{ss}})$$  be the natural map.  
Then $\im a = \Int(\cg(\sg)^{\mathrm{ss}}) \rtimes K$ for some subgroup of diagram automorphisms (\S \ref{symmetry_nsl_case}).
By a case-by-case check, $H$ in the Proposition \ref{prop:split} can be chosen 
so that $H$ maps onto $K$ via $a$.   Then the image of $H$ in $\Aut(\cg(\sg))$ is well-defined
up to conjugacy in $\Aut(\cg(\sg))$.   In the above example, $H = \langle s_2 \rangle$ has the desired property, since $F_4$ is the fixed subalgebra under the non-trivial diagram automorphism of $E_6$.  
We note that \cite{alexeevsky} is the original source for computing the image of the map $a$.

\subsection{Computing the intrinsic symmetry}  \label{full_intrinsic_action}

Having chosen $H$ with $a(H) = K$ as above, we can determine the action of $H$ on $\slice$.   
Here, we restrict to the exceptional groups and to a minimal degeneration $\0'$ of $\0$, with $e \in \0'$.
We summarize the possibilities and record the action of $H$ on $\slice$ in the graphs at the end of the paper.

\subsubsection{Minimal singularities: $A(e) = \SG_2$ cases} 

Let $\slice$ be an irreducible minimal singularity admitting an involution as in \S \ref{def:symmetry_min_sing}.
If $|H|=2$, then it turns out that $H$ realizes this involution.  There is one case of this kind when $|H|=4$,
when $e=A_4+A_1$ in $E_8$ and $\slice \cong a_2$.   Let $H = \langle s \rangle$.  
Then $s \in H$ realizes the involution on $\slice$ and $s^2$ acts trivially on $\slice$.  
We will still refer to this singularity with induced symmetry by $a_2^+$.

If $\slice$ is a reducible minimal singularity, then it is turns out that $\slice$ has exactly two irreducible components and $H$ interchanges the two components.  
The only three cases which occur are the singularities with symmetry action $[2A_1]^{+}, [2a_2]^+$, and $[2g_2]^+$.

\subsubsection{Minimal singularities: $A(e) = \SG_3$ cases} 

If $\slice$ is the unique irreducible minimal singularity admitting an action of $\SG_3$ as in \S \ref{def:symmetry_min_sing}, then $H$ 
realizes the full symmetry $d_4^{++}$.  This only occurs once, in $E_8$.

If $\slice$ is a reducible minimal singularity, then $\slice$ turns out to have $3$ irreducible components and $H$ acts by permuting transitively the three components.  In other words, the stabilizer of a component acts trivially on the component.  
All of these cases are of the form $3A_1$ and the singularity with symmetry action is denoted $[3A_1]^{++}$.

\subsubsection{Simple surface singularities: $A(e) = \SG_2$ cases} 

If $\slice$ is an irreducible simple surface singularity admitting an involution as in \S \ref{surface_symmetry}
(or in the case of $A_2$ and $A_4$, admitting the appropriate cyclic action of order $4$), then $H$ realizes this symmetry.  
To show this, we first checked that $A(e)$ has the appropriate action on the dual graph of a minimal resolution in Corollary \ref{cor_action}.  
Then since $C(\sg)$ acts symplectically on $\slice$, Corollary 1.1 and Theorem 1.2 in \cite{Catanese} imply 
that $H$ corresponds to the $\Gamma' \subset \SL_2(\CM)$ which defines the symmetry involution. 

The only reducible surface singularities with $A(e) = \SG_2$ are those with $\slice \cong 2A_1$, hence covered previously.

\subsubsection{Simple surface singularities: $A(e) = \SG_3$ cases} 
If $\slice$ is an irreducible simple surface singularity admitting an $\SG_3$ action as in \S \ref{surface_symmetry},
then $H$ realizes the symmetry action and so $\slice \cong G_2$. 

An unusual situation occurs for the minimal degeneration $(D_7(a_1),E_8(b_6))$.   Here, $A(e) = \SG_3$, but $\slice$ only admits a two-fold  symmetry, compatible with its normalization $\widetilde{\SC}_{\0, e}$ which is $A_3$.   
Here, $\Gamma \subset  \SL_2(\CM)$ corresponding to $\widetilde{\SC}_{\0, e}$ is cyclic of order $4$.   The normal cyclic subgroup of $H \cong \SG_3$ is generated by an element $s$ with $\gg^s$ of type $E_6+A_2$ and hence $s$ acts without fixed point on the orbit $D_7(a_1)$ since the latter orbit does not meet the subalgebra $E_6+A_2$. 
On the other hand, using Corollary \ref{cor_action}, 
we see that $A(e)$ induces the involution on the dual graph of a minimal resolution of $\widetilde{\SC}_{\0, e}$.
Since $C(\sg)$ acts symplectically on $\slice$ and $\widetilde{\SC}_{\0, e}$, Corollary 1.1 and Theorem 1.2 in \cite{Catanese} imply  
that $H$ acts on $\widetilde{\SC}_{\0, e} = \CM^2/\Gamma$ 
via the action of $\Gamma' \subset \SL_2(\CM)$, the binary dihedral group of order $24$ containing 
$\Gamma$ as normal subgroup.


If $\slice$ is a reducible surface singularity, then 
$\slice$ is isomorphic to $3C_2, 3C_3, 3(C_5),$ or the previously covered $[3A_1]^{++}$.  We have omitted
the superscript in $3C_2$, etc.   The notation means that $H$ permutes the three components transitively and the stabilizer of any component
is order $2$, which acts by the indicated symmetry.
The notation $(C_5)$ refers to the fact that we do not know whether an irreducible component is normal.

\subsubsection{Simple surface singularities: $A(e) = \SG_4$ case} 

This only occurs in $F_4$.   One degeneration has $\slice \cong G_2$ (see \S \ref{example:B3inF4}).  Here, the Klein $4$-group in $H$ acts trivially on $\slice$ and the quotient action realizes
the full symmetry of $\SG_3$ on $\slice$.  This follows either from the list of possible symplectic automorphisms of $\slice$ or from a direct calculation that the Klein $4$-group in $H$ fixes $\slice$ pointwise.

The other degeneration has $\slice \cong 4G_2$ (see \S \ref{example:C3inF4}).  Here, $H$ permutes the four components transitively and the stabilizer of any component
is an $\SG_3$, which acts by the indicated symmetry.

\subsubsection{Simple surface singularities: $A(e) = \SG_5$ case} 

This only occurs in $E_8$.   One degeneration has $\slice \cong 10G_2$.  Here, 
$H$ permutes the ten components transitively and the stabilizer of any component
is a Young subgroup $\SG_3 \times \SG_2$.  The $\SG_2$ factor acts trivially on the given component and the 
$\SG_3$ factor acts by the indicated symmetry.

The other degeneration has $\slice \cong 5G_2$.  Here, 
$H$ permutes the five components transitively and the stabilizer of any component
is a $\SG_4$.  The $\SG_4$ factor acts on the given component as in the $F_4$ case above. 

\begin{remark}
Even when $A(e)$ is non-trivial, it might not induce a non-trivial symmetry on any $\slice$.  
For example, when $e = C_3(a_1)$ in $F_4$, the only degeneration above $\0_e$ has $\slice \cong A_1$.
Here, $H$ acts trivially on $\slice$, reflecting the fact that $\SL_2(\CM)$ has no outer automorphisms.  
Indeed, $C(\sg)$ is just the direct product $C(\sg)^{\circ} \times H$.
\end{remark}

\section{Results for $F_4$}\label{results_F4}

\subsection{Details in the proof of Proposition \ref{dim4+_1}}

Here we record the details for establishing Proposition \ref{dim4+_1} for $\gg$ of type $F_4$, as outlined in \S \ref{section:dim4}. 
First, we enumerate the $G$-orbits of those $e$ such that $\cg(\sg)$ has non-trivial intersection with the minimal nilpotent orbit in $\gg$.
To that end, let $e_0 \in \gg$ be minimal nilpotent and recall that $\sg_0$ is an $\sl_2(\CM)$-subalgebra through $e_0$.  The centralizer $\cg(\sg_0)$ 
is a simple subalgebra of type $C_3$, equal to the semisimple part of a Levi subalgebra of $\gg$.
The relevant nonzero nilpotent elements $e \in \cg(\sg_0)$ are therefore those in the $G$-orbits
$$A_1, \tilde{A}_1, A_1+\tilde{A}_1, \tilde{A}_2, B_2, C_3(a_1) \text{ and } C_3$$
and hence Corollary \ref{slice_corollary} applies to these elements.  
The computation of $e+e_0 \in \0$ proceeds as in \S \ref{subsection:elements_in_red_centralizer}.  The results are in Table \ref{table1_f4}. 
We use boldface font in Table \ref{table1_f4} to locate the simple factors whose minimal nilpotent orbit is of type $A_1$ in $\gg$. 
Where more than one such simple factor is in boldface, this indicates that the factors are conjugate under the action of $C(\sg)$.
The first two lines of Table \ref{table4_f4} have the remaining cases where Lemma \ref{kp_prop} applies 
with $x = e + e_0$ for an element $e_0$ in a minimal nilpotent orbit of $\cg(\sg)$.  
This now exhausts all minimal degenerations covered by Proposition \ref{dim4+_1} with $J = \emptyset$.

\begin{table}[htp]
\caption{${\bf F_4}$:  cases with $e_0 \in \cg(\sg)$ of type $A_1$ in $\gg$}%
\begin{center}
\begin{tabular}{|c|c|c|c|}  \hline
$e$  &  $e+e_0 \in \0$  & $\cg(\sg)$   &   Isomorphism type of $\SC_{\0, e}$  \\ \hline 
$A_1$  &  $2A_1 = \tilde{A}_1$ & $C_3$  &  $c_3$\\ \hline 
$\tilde{A}_1$   & $\tilde{A}_1+A_1$ &  $A_3$ &  $a^+_3$ \\ \hline 
$A_1 + \tilde{A}_1$  &  $2A_1 + \tilde{A}_1 = A_2$  &  $A_1+{\bf A_1}$ & $A_1$ \\ \hline 
$\tilde{A_2}$  &   $\tilde{A}_2+A_1$ & $G_2$ & $g_2$   \\ \hline
$B_2$   &   $B_2 + A_1 = C_3(a_1)$  &  ${\bf 2A_1}$  &  $[2A_1]^+$  \\ \hline 
$C_3(a_1)$ &  $C_3(a_1) + A_1 = F_4(a_3)$  &  $A_1$ & $A_1$\\  \hline
$C_3$   &   $C_3 + A_1 = F_4(a_2)$  &   $A_1$ &  $A_1$ \\ \hline
\end{tabular}
\end{center}
\label{table1_f4}
\end{table}

The remaining minimal degenerations in the proposition, of which there are four, are all codimension two and unibranch.
We use \S \ref{F4_surface} to determine that these exhaust the remaining codimension two cases with $|\Gamma| =1$ or $2$.
We now show that all four cases are $S$-varieties for $SL_2(\CM)$ of the form $X(2, i_1+2, i_2+2,\dots)$, with 
$J = \{i_1, i_2, \dots \}$
among those listed in Remark  \ref{which_J}.  
All cases can be handled with Lemma \ref{S_var_sl2} and Corollary \ref{A1_or_m},
but in one case we need to pass to a subalgebra (as in Lemma \ref{singequi}) 
and in another, do an explicit computation to find the exact form of $\slice$.
The values of $(m_j, n_j)$ for $j \in \EC$ are listed in Table \ref{table4_f4}.  
Boldface is used for those $(m_j, n_j)$ 
where $x_{m_j} \neq 0$ in \eqref{x_form1}. 
Equivalently, the set $J$ consists of the $m_j$'s in boldface.

\subsubsection{The degeneration $(\tilde{A}_2, A_1+\tilde{A}_1)$}  \label{first_restriction_example}
For $e$ of type $A_1+\tilde{A}_1$, 
$\cg(\sg) \cong  \sl_2(\CM) \oplus \sl_2(\CM)$.
The nonzero nilpotent elements in one simple factor of $\cg(\sg)$ are minimal in $\gg$ and this case was handled earlier.   
The nonzero  nilpotent elements in the other simple factor $\zg$ are of type $\tilde{A}_2$ in $\gg$.  
Let $e_0 \in \zg$ be such an element.  The centralizer
$\cg(\zg)$ is contained in a Levi subalgebra of $\gg$ whose semisimple type is $B_3$, 
and thus $\cg(\zg)$ does not meet the $G$-orbit $\0$ of type $\tilde{A}_2$.  
Hence Lemma \ref{S_var_sl2} applies
to $\slice$.  Now $e+e_0$ is of type $C_3(a_1)$ and $(m_i,n_i)=(2,4)$ for the unique element $i \in \EC$.
The argument in Example \ref{C3_example_1} then gives that
$\slice = e + X(2,4) \cong X(2)$ is an $A_1$-singularity.

In fact Example \ref{C3_example_1} can be used more directly.  
This also illustrates the process of passing to a subalgebra
to establish that $\slice$ has the desired form as an $S$-variety.  
Let $\lg$ be a Levi subalgebra of $\gg$ whose semisimple type is $C_3$.  
The $G$-orbit through $e$ meets $\lg$ in the orbit $[2^3]$, so we may assume $e \in \sg \subset \lg$.  
Then $\cg(\sg) \cap \lg$ coincides with $\zg$ and $\0 \cap \lg$ coincides with the orbit in $\lg$ of type $[3^2]$.
For dimension reasons it follows that $\SC_{\0 \cap \lg, e} \subset \lg$ equals $\slice$.  
Thus Example \ref{C3_example_1} directly gives $\slice = e + X(2,4)$.

\subsubsection{The degeneration $(C_3(a_1), \tilde{A}_2 + A_1)$}  \label{import_G2_1}
For $e$ of type $\tilde{A}_2 + A_1$, $\cg(\sg) \cong  \sl_2(\CM)$.  
 Let $e_0 \in \cg(\sg)$ be a nonzero element in $\zg = \cg(\sg)$, which has 
  type  $A_1+\tilde{A}_1$ in $\gg$.
The orbit $\0$ of type $C_3(a_1)$ does not meet $\cg(\zg)$, so Lemma \ref{S_var_sl2} applies.
The sum $e+e_0$ is of type $F_4(a_3)$ and $(m_i,n_i)=(1,3)$ for the unique element $i \in \EC$.
Hence as in Example \ref{mexample}, we have $\SC_{\0, e}  = e + X(2,3) \cong m$.
The result can also be obtained by reducing to the subalgebra $\sg' \oplus \cg(\sg')$, 
where $\sg'$ is the $\sl_2$-subalgebra through an element $e'$ of type $\tilde{A}_2$.
A key factor making this work is that $\cg(\sg')$ has type $G_2$ and we can directly
use Example \ref{mexample}. We omit the details.

\subsubsection{The degenerations $( B_2,  A_2+\tilde{A}_1)$ and $(\tilde{A}_2 + A_1,  A_2+\tilde{A}_1)$}  \label{F4hard_m}
For $e$ of type $A_2+\tilde{A}_1$, $\cg(\sg) \cong \sl_2(\CM)$.  
The nonzero nilpotent orbit in $\cg(\sg)$ also has type  $A_2+\tilde{A}_1$ in $\gg$.
Hence for $\zg:=\cg(\sg)$, we have $\cg(\zg) = \sg$ and so Lemma \ref{S_var_sl2}
applies for $\0$ both of type $B_2$ and of type $\tilde{A}_2 + A_1$.
Let $e_0 \in \zg$ be nonzero nilpotent.
The sum $e+e_0$ is of type $F_4(a_3)$ 
and $\{ (1,3), (2,4)\}$ are the values for the two elements in $\EC$.   
Indeed the decomposition of $\gg$ in \eqref{sl2_decomp} is
$$V(0,2) \oplus V(1,3) \oplus V(2,0) \oplus   V(2,4)  \oplus V(3,1) \oplus V(4,2) .$$
So the only remaining question to determine the isomorphism type of $\slice$
is whether $x_1$ is nonzero when expressing $x \in \0$ 
as in \eqref{x_form1}.

For $\0$ of type $B_2$ in $\gg$, we see that 
$\0$ meets the maximal simple subalgebra $\lg \cong \so_9(\CM)$ in $F_4$
in the orbit with partition $[4^2,1]$, while the orbit $\0_e$ meets $\lg$ in the orbit
with partition $[3^3]$.
So we may assume $\sg \subset \lg$ and consider $\SC_{\0 \cap \lg, e}$.
The centralizer of $\sg$ in $\lg$ remains $\sl_2(\CM)$, so we may also assume that 
$\sg_0 = \cg(\sg) \subset \lg$.   
Calculating $\EC$ for $\sg$ and $\sg_0$ relative to $\lg$, we find that 
only $(2,4)$ occurs.  Hence we can identify $\SC_{\0 \cap \lg, e}$ with $\slice$ since both are dimension two.  
We conclude that $x_1 = 0$ in \eqref{x_form1}.
It follows that $\SC_{\0, e} = e + X(2,4) \cong X(2)$. 

On the other hand, for the orbit $\0$ of type $\tilde{A}_2 + A_1$, we have to carry out an explicit computation in GAP.
We find that both $x_1$ and $x_2$ are nonzero in \eqref{x_form1} and
thus $\SC_{\0, e} = e +  X(2,3,4)$, which is isomorphic to $m$ by \S \ref{S_vars}.


\begin{table}[htp]
\caption{${\bf F_4}$:  Remaining relevant cases with $e_0$ minimal in $\cg(\sg)$}
\begin{center}
\begin{tabular}{|c|c|c|c|c|c|c|c|}  \hline
$e$  &  $e_0$   &  $e+e_0$  & $\0$ & $\cg(\sg)$   &  $(m_i, n_i)$ for  & Isomorphism \\
&&&&& $i \in \EC$ &  type of $\SC_{\0, e}$   \\ \hline 
$A_2$ &  $\tilde{A}_1$  & $A_2 + \tilde{A}_1$ & $A_2 + \tilde{A}_1$ &   $A_2$   & $\emptyset$   & $a_2^+$   \\ \hline 
$B_3$  & $\tilde{A}_2$  &  $F_4(a_2)$  &  $F_4(a_2)$ &  $A_1$   & $\emptyset$   & $A_1$ \\ \hline
\hline
$A_1 + \tilde{A}_1$  &   $\tilde{A}_2$ &  $C_3(a_1)$  &   $\tilde{A}_2$   & ${\bf A_1}+A_1$ &   ${\bf (2,4)}$ & $A_1$ \\ \hline 
$A_2+\tilde{A}_1$  &  $A_2+\tilde{A}_1$   &  $F_4(a_3)$  &  $\tilde{A}_2+A_1$  &   $A_1$   & ${\bf (1,3), (2,4)}$ &  $m$   \\ \hline
			  &     & &    $B_2$ &   & $(1,3), {\bf (2,4)}$ &  $A_1$   \\ \hline
$\tilde{A}_2+A_1$ &  $A_1+ \tilde{A}_1$   &   $F_4(a_3)$ &  $C_3(a_1)$   & $A_1$ &  ${\bf (1,3)}$ & $m$   \\ \hline
\end{tabular}
\end{center}
\label{table4_f4}
\end{table}%

\subsection{Remaining surface singularities}  \label{F4_surface}  \label{example:C3inF4}
This section summarizes the calculations of the singularities of the minimal degenerations of codimension two, using the methods in \S \ref{section:codimension2}.

For the cases in Proposition \ref{dim4+_1}, we did not need to know whether a nilpotent orbit has closure which is normal to determine 
the singularity type of a minimal degeneration.  Knowing the branching was sufficient.
Indeed, the closure of the orbit $B_2$ is non-normal, 
but it was shown above that it is normal at points in the orbit $A_2 + \tilde A_1$ since the singularity of that degeneration is of type $A_1$.  
Similarly for the orbit $\tilde{A}_2$.  The remaining non-normal orbit closures, of which there are three \cite{Broer}, are detected through a minimal degeneration:  either the closure is branched at a minimal degeneration (as for $C_3$) or is isomorphic to $m$ at a minimal degeneration (as for $C_3(a_1)$ and for $\tilde{A}_2+A_1$).
In what follows we use the fact that the orbit $F_4(a_1)$ has closure which is normal \cite{Broer}
to classify the type of its minimal degeneration.  
This is the only case where we need to know whether the closure is normal in order to resolve 
the type of a minimal degeneration in $F_4$.

\begin{enumerate}
\item $(\0, \0') = (F_4(a_1), F_4(a_2))$.
The even orbit $F_4(a_1)$ is Richardson for the parabolic subgroup $Q$ with Levi factor of type $\tilde{A}_1$
and the resulting map $p: G \times^Q \ng(\qg)  \to \cl$ is birational, hence a generalized Springer resolution.
The hypotheses of Lemma \ref{qfact} hold and $b_2(G/Q) = 3$. 
Since $\0'$ is the unique orbit of codimension two in $\cl$,
it follows from \S \ref{orbital_exceptional} that there 
are $3$ orbits of $A(e)=\MATHFRAK{S}_2$ on the
irreducible components of $p^{-1}(e)$.
On the other hand, there are a total of four irreducible components of $p^{-1}(e)$ by \S \ref{Bo-M method}.
Thus the singularity must be $C_3$, given that $\cl$ is normal.

\item 
$(\0, \0') = (C_3, F_4(a_3))$. The orbit $\0$ is 
Richardson for the parabolic subgroup $Q$ with Levi factor of type $A_2$.
The map $p: G \times^Q \ng(\qg)  \to \cl$ is birational, hence a generalized Springer resolution,
since $A(x)=1$ for $x \in \0$.
If $e \in \0'$, then $A(e)=\SG_4$.  By Lemma \ref{formula} and Corollary \ref{cor_action}, there 
are 16 irreducible components in $p^{-1}(e)$ with two orbits under $A(e)$.  
The number of orbits can also be deduced from \S \ref{orbital_exceptional}.
Looking at the possibilities for the dual graph, it is clear that $\cl$ is non-normal and 
the normalization map $\nu: \widetilde{\OC} \to \cl$ restricts to a degree 4 map over $\OC'$.
This also follows from \cite{Sho} (\S \ref{Component group action}).  
By Corollary \ref{cor_action}, there is a fixed component of $\pi^{-1}(y)$ under the $A(e)$-action 
for $y = \nu^{-1}(e)$.  This implies that the singularity of $\widetilde{\OC}$ at $y$ is $G_2$.
We show in \S \ref{4G2} that $\slice$ is isomorphic to $4G_2$.  In other words,
the irreducible components of $\slice$ are normal and hence each is isomorphic to $G_2$.

\item $(\0, \0') = (B_3, F_4(a_3))$.   The singularity is $G_2$ by \S \ref{example:B3inF4}.
\end{enumerate}

\subsection{The degeneration $(C_3,F_4(a_3))$ is $4G_2$} \label{4G2}

We now show each irreducible component of this slice is normal. 
By \S \ref{F4_non_min} 
the nilpotent Slodowy slice $\SC$ of $C_3$ at $\tilde{A}_2$ contains an irreducible component isomorphic to the nilpotent
cone $\NC_{G_2}$.   Recall $e$ belongs to the $\tilde{A}_2$ orbit, with corresponding $\sl_2$-subalgebra
$\sg$.  Let $e_0 \in \cg(\sg)$ be subregular nilpotent.   Then a calculation shows that $e':=e+e_0$ lies in the 
$F_4(a_3)$ orbit in $\gg$ and also that $e'$ lies in the component of $\SC$ isomorphic to $\NC_{G_2}$.
Hence the nilpotent Slodowy slice of $C_3$ at $F_4(a_3)$ contains a component that
is smoothly equivalent to the nilpotent Slodowy slice in $\cg(\sg)$ of $G_2$ at $G_2(a_1)$.
But then this component must be isomorphic to the simple surface singularity $D_4$ by Lemma \ref{Cor:2-dim-isom}.
Incorporating the symmetry of $A(e') = \mathfrak{S}_4$, 
the nilpotent Slodowy slice of $C_3$ at $F_4(a_3)$ is isomorphic to $4G_2$.

\section{Results for $E_6$}\label{sec:E6}

\subsection{Details in the proof of Proposition \ref{dim4+_1}}

In Table \ref{table1_e6} we list the cases where Corollary \ref{slice_corollary} holds for $e_0$ in the minimal orbit of $E_6$.  
Here $\cg(\sg_0)$ is the semisimple part of a Levi subalgebra and has type $A_5$.
The relevant nonzero nilpotent $G$-orbits are those that have non-trivial intersection with $\cg(\sg_0)$. 
In the first line of Table \ref{table4_e6}  is the remaining case 
where Lemma \ref{kp_prop} applies 
with $x = e + e_0$ for an element $e_0$ in a minimal nilpotent orbit of $\cg(\sg)$.
This exhausts all minimal degenerations covered by the proposition with $J = \emptyset$.
There are only two cases where $J \neq \emptyset$, both of codimension two.
The degeneration $(A_3, A_2+2A_1)$ follows
 from working in the Levi subalgebra of semisimple type $D_5$, 
similar to \S \ref{first_restriction_example}.
The degeneration $(A_3+A_1, 2A_2+A_1)$ is similar to \S \ref{import_G2_1}.
Details are given in Table \ref{table4_e6}. 

\begin{table}[htp]
\caption{${\bf E_6}$:  cases with $e_0 \in \cg(\sg)$ of type $A_1$ in $\gg$}
\begin{center}
\begin{tabular}{|c|c|c|c|}  \hline
$e$  &  $e+e_0 \in \0$  & $\cg(\sg)$   &  Isomorphism type of $\SC_{\0, e}$  \ \\ \hline 
$A_1$  &   $2A_1$ &  $A_5$  &    $a_5$ \\ \hline
$2A_1$  &  $3A_1$  & $B_3+T_1$ &   $b_3$   \\ \hline
$3A_1$  &  $4A_1 = A_2$  &   $A_2+ {\bf A_1}$ &   $A_1$   \\ \hline
$A_2$   &  $A_2 + A_1$  &   ${\bf 2A_2}$   &  $[2a_2]^+$   \\ \hline
$A_2+A_1$  &  $A_2 + 2A_1$    &  $A_2 +T_1$  & $a_2$  \\ \hline 
$2A_2$ &   $2A_2+A_1$ &  $G_2$  & $g_2$   \\ \hline
$A_3$   &    $A_3 + A_1$  &   $B_2+T_1$   &  $b_2$   \\ \hline
$A_3+A_1$ & $A_3+2A_1 = D_4(a_1)$ &  $A_1+T_1$ &  $A_1$   \\ \hline 
$A_4$   &   $A_4 + A_1$  &   $A_1+T_1$   &   $A_1$  \\ \hline
$A_5$   &   $A_5 + A_1 = E_6(a_3)$  &   $A_1$    &   $A_1$  \\ \hline
\end{tabular}
\end{center}
\label{table1_e6}
\end{table}

%
%

\begin{table}[htp]
\caption{${\bf E_6}$:  Remaining relevant cases with $e_0$ minimal in $\cg(\sg)$}
\begin{center}
\begin{tabular}{|c|c|l|c|c|c|c|c|}  \hline
$e$  &  $e_0$   &  $e+e_0$  & $\0$ & $\cg(\sg)$   & $(m_i, n_i)$ for  & Isomorphism  \\
&&&& & $i \in \EC$ &  type of $\SC_{\0, e}$ \\ \hline 
$D_4$   &  $2A_1$  &   $D_5(a_1)$  &  $D_5(a_1)$  &  $A_2$  &  $\emptyset$ &$a_2$   \\ \hline \hline
$A_2+2A_1$  & $A_2+2A_1$ & $D_4(a_1)$  &  $A_3$ &  $A_1 +T_1$ &   $(1,3), (1,3), {\bf (2,4)}$ &  $A_1$  \\ \hline
$2A_2+A_1$  & $3A_1$ & $D_4(a_1)$ & $A_3+A_1$ &  $A_1$ &  ${\bf (1,3)}$ & $m$   \\ \hline
\end{tabular}
\end{center}
\label{table4_e6}
\end{table}

\subsection{Remaining surface singularities, and an exceptional degeneration}  \label{surfaces_E6}

The results are listed in Table \ref{extra3_e6}.   
In the first four entries of the table, we use the fact that the larger orbit has closure which is normal \cite{Sommers:E6}.
The entry for $(E_6(a_1), D_5)$ is from \S \ref{example: E6_subregular}.  
The entry for $(A_4, D_4(a_1))$ is $3C_2$ since the irreducible components are isomorphic and one of them is isomorphic to $C_2$
from Table \ref{full_nilcones}.  Alternatively, it follows from working in the Levi subalgebra of semisimple type $D_5$
and using Lemma \ref{singequi} and \cite{Kraft-Procesi:classical}. 
The entry for $(D_4, D_4(a_1))$ is also clear from working in the Levi subalgebra of semisimple type $D_4$.  
The degenerations $(E_6(a_3), D_5(a_1))$ and $(2A_2, A_2+A_1)$ are both $A_2$ using larger slices (see Table \ref{full_nilcones}).   
Note that the $2A_2$ orbit is unibranch at $A_2+A_1$, but its closure is not normal.

The exceptional degeneration $(2A_2+A_1, A_2+2A_1)$ of codimension four is treated in \S \ref{dim4_exceptions}.

\begin{table}[htp]
\caption{Surface singularities using \S \ref{section:codimension2}: $E_6$}
\begin{center}
\begin{tabular}{|c|c|c|c|c|c|c|}
\hline Degeneration &  Induced from  & $\sharp \mathbb{P}^1$'s &$A(e)$ & $\sharp$ orbits of $A(e)$ 
&  $\slice$ \\
\hline
$(E_6(a_1), D_5)$ & $(A_1, 0)$ & 5 & $1$ &  & $A_5$ \\
\hline
$(D_5, E_6(a_3))$ & $(2A_1, 0)$ & 4 & $S_2$ & 3  & $C_3$ \\
\hline
$(D_5(a_1),A_4+A_1)$ & $(A_2+A_1, 0)$ & 2 & $1$ &   & $A_2$ \\
\hline
$(A_5, A_4+A_1)$ & $(D_4, 32^21)$ & 2 & $1$ &   & $A_2$ \\
\hline
$(D_4, D_4(a_1))$ & $(2A_2, 0)$ & 4 & $S_3$ & 2 
& $G_2$ \\
\hline
$(A_4, D_4(a_1))$ & $(A_3, 0)$ & 9 & $S_3$ & 2 & $3C_2$ \\
\hline
\end{tabular}
\end{center}
\label{extra3_e6}
\end{table}


\section{Results for $E_7$}\label{sec:E7}

\subsection{Details in the proof of Proposition \ref{dim4+_1}}

In Table \ref{table1_e7} we list the cases where Corollary \ref{slice_corollary} applies. 
Here $\cg(\sg_0)$ is the semisimple part of a Levi subalgebra and has type $D_6$.
The relevant nonzero nilpotent $G$-orbits are those that have non-trivial intersection with $\cg(\sg_0)$. 
In the first several lines of Table \ref{table4_e7}  are the remaining cases where Lemma \ref{kp_prop} applies 
with $x = e + e_0$ for an element $e_0$ in a minimal nilpotent orbit of $\cg(\sg)$. 

\begin{table}[htp]
\caption{${\bf E_7}$:  cases with $e_0 \in \cg(\sg)$ of type $A_1$ in $\gg$}
\begin{center}
\begin{tabular}{|c|c|c|c|}  \hline
$e$  &  $e+e_0 \in \0$  & $\cg(\sg)$   &   Isomorphism \\ 
&&& type of $\SC_{\0, e}$ \\ \hline    
$A_1$  &   $2A_1$ &   $D_6$  &    $d_6$ \\ \hline
$2A_1$  &  $(3A_1)''$ &   $B_4+{\bf A_1}$    &  $A_1$ \\ \hline
             &            $(3A_1)'$  & ${\bf B_4} + A_1$    &  $b_4$   \\ \hline
$(3A_1)''$   &    $4A_1$  &  $F_4$      & $f_4$   \\ \hline
$(3A_1)'$   &  $4A_1$  &   ${\bf C_3}+A_1$   &   $c_3$   \\ \hline
                &  $A_2$  &  $C_3 + {\bf A_1}$   & $A_1$   \\ \hline
$A_2$ & $A_2+A_1$ &  $A_5$ & $a^+_5$   \\ \hline
$4A_1$ &  $5A_1 = A_2+A_1$   & $C_3$   &   $c_3$   \\ \hline
$A_2+A_1$  &  $A_2+2A_1$ & $A_3+T_1$  & $a^+_3$   \\ \hline
$A_2+2A_1$  &  $A_2+3A_1$  &  $A_1+A_1+ {\bf A_1}$ &  $A_1$   \\ \hline
$A_3$ &  $(A_3+A_1)''$   & $B_3+{\mathbf A_1}$   &   $A_1$   \\ \hline
  & $(A_3 + A_1)'$  &   ${\bf B_3}+A_1$   & $b_3$   \\ \hline
$2A_2$   &  $2A_2+A_1$  &   ${\bf G_2}+ A_1$   &  $g_2$   \\ \hline
$(A_3 + A_1)'$  & $A_3+2A_1$    &  $A_1+A_1+{\bf A_1}$  &  $A_1$    \\ \hline
    & $(A_3+2A_1)' = D_4(a_1)$    &  $A_1+{\bf A_1}+A_1$  &  $A_1$    \\ \hline
$(A_3 + A_1)''$  & $A_3+2A_1$    & $B_3$  & $b_3$   \\ \hline
$D_4(a_1)$  &  $D_4(a_1)+A_1$  & ${\bf 3A_1}$  &  $[3A_1]^{++}$  \\ \hline
$A_3+2A_1$    & $A_3+3A_1 = D_4(a_1)+A_1$  & $A_1+{\bf A_1}$  & $A_1$   \\ \hline
$D_4$ &  $D_4+A_1$   &   $C_3$  &  $c_3$   \\ \hline
$D_4(a_1)+A_1$ & $D_4(a_1)+2A_1= A_3+A_2$ &  ${\bf 2A_1}$  &  $[2A_1]^+$   \\ \hline
$A_3+A_2$ & $A_3+A_2+A_1$  & $A_1+T_1$ & $A_1$   \\ \hline
$A_4$ &  $A_4+A_1$  & $A_2+T_1$  &  $a^{+}_2$  \\ \hline
$D_4+A_1$  & $D_4+2A_1 = D_5(a_1)$ & $B_2$ & $b_2$   \\ \hline
$(A_5)''$ &   $A_5+A_1$ & $G_2$ & $g_2$   \\ \hline
$D_5(a_1)$   &  $D_5(a_1) + A_1$  &   $A_1+T_1$   &   $A_1$  \\ \hline
$(A_5)'$ &   $(A_5+A_1)' = E_6(a_3)$ & $A_1+{\bf A_1}$ & $A_1$  \\ \hline
$D_6(a_2)$  &  $D_6(a_2) + A_1 = E_7(a_5)$  &   $A_1$    &   $A_1$  \\ \hline
$D_5$ &  $D_5+A_1$  &  $A_1+{\bf A_1}$  &  $A_1$  \\ \hline
$D_6(a_1)$   &  $D_6(a_1) + A_1 = E_7(a_4)$  &   $A_1$   &   $A_1$  \\ \hline
$D_6$   &  $D_6 + A_1 = E_7(a_3)$  &   $A_1$ &   $A_1$  \\ \hline
\end{tabular}
\end{center}
\label{table1_e7}
\end{table}

The nine remaining cases (all of codimension two), involving $e$ from six different $G$-orbits, are listed in Table \ref{table4_e7}.  
The cases where $e$ is type $A_2+2A_1$ or $2A_2+A_1$ follow by restricting to a subalgebra of type $E_6$.
The case where $e$ is type $A_5+A_1$ proceeds as in Example \ref{mexample}.
The two cases where $e$ is type $D_5(a_1)+ A_1$ are similar to Example \ref{C3_example_1}.
The three minimal degenerations lying above the orbit $A_4+A_2$ and the one above the orbit $A_3+A_2+A_1$ 
satisfy part (2) of Lemma \ref{S_var_sl2}.  Since all the $m_i$ are even for $i \in \EC$, Corollary \ref{A1_or_m} gives that
these four degenerations are $A_1$-singularites and satisfy the proposition.  Still, we carry out 
an explicit computer calculation in GAP to show that both $x_2$ and $x_4$ are nonzero for these degenerations,
so that in each of these cases, $\slice$ takes the form $e + X(2,4,6)$.
The details are omitted.

\begin{table}[htp]
\caption{${\bf E_7}$:  Remaining relevant cases with $e_0$ minimal in $\cg(\sg)$}
\begin{center}
\begin{tabular}{|c|c|l|c|c|c|c|}  \hline
$e$  &  $e_0$   &  $e+e_0$  & $\0$ & $\cg(\sg)$   &  $(m_i, n_i)$ for  &  Isomorphism type  \\
&&&& & $i \in \EC$ &  of $\SC_{\0, e}$  \\ \hline 
$A_2+2A_1$  &  $2A_1$  &  $A_2$  &  $A_2$  & $A_1+ {\bf A_1}+A_1$ &  $\emptyset$  & $A_1$   \\ \hline
$A_2+3A_1$   &  $2A_1$  &  $2A_2+A_1$  &  $2A_2+A_1$  & $G_2$   &  $\emptyset$& $g_2$ \\ \hline
$2A_2$   &    $(3A_1)''$  &  $(A_3+A_1)''$  & $(A_3+A_1)''$  &  $G_2+{\bf A_1}$   &  $\emptyset$ &    $A_1$   \\ \hline
$(A_5)'$  & $(3A_1)''$  & $D_6(a_2)$ &$D_6(a_2)$ & ${\bf A_1}+A_1$ &  $\emptyset$ &  $A_1$  \\ \hline
$D_5$   &    $2A_1$  &  $D_6(a_1)$  &  $D_6(a_1)$  &  ${\bf A_1}+A_1$  &  $\emptyset$  & $A_1$  \\ \hline
$D_5+A_1$   &   $2A_1$  &  $E_7(a_4)$  &  $E_7(a_4)$  & $A_1$  &  $\emptyset$   &  $A_1$  \\ \hline
$A_6$   &    $A_2+3A_1$  &  $E_7(a_4)$  &  $E_7(a_4)$  &  $A_1$  &  $\emptyset$       &   $A_1$  \\ \hline
$E_6(a_3)$  &   $(3A_1)''$  &  $E_7(a_5)$  & $E_7(a_5)$  &  $A_1$   &  $\emptyset$     & $A_1$ \\ \hline
$E_6$   &    $(3A_1)''$  &  $E_7(a_2)$  &  $E_7(a_2)$  & $A_1$  &  $\emptyset$     &   $A_1$  \\ \hline
\hline
$A_2+2A_1$  & $A_2+2A_1$ & $D_4(a_1)$  & $A_3$     &${\bf A_1} +A_1+A_1$ &  $(1,3)^4, {\bf (2,4)}$ &  
$A_1$  \\ \hline
$2A_2+A_1$  & $(3A_1)'$ & $D_4(a_1)$ &    $(A_3+A_1)'$   &$A_1+{\bf A_1}$  & ${\bf (1,3)}$ & 
$m$   \\ \hline
$A_3+A_2+A_1$    &  $A_4+A_2$  &  $E_7(a_5)$  &  $D_4+A_1$  & $A_1$  & ${\bf (2,4)}, (2,8), {\bf (4,6)}$ &  
$A_1$   \\ \hline
$A_4+A_2$  & $A_3+A_2+A_1$ & $E_7(a_5)$ & $A_5+A_1$   &$A_1$  & ${\bf (2,4), (4,6)}$ & 
$A_1$   \\ \hline
&&&  $(A_5)'$  &&${\bf (2,4), (4,6)}$ & 
$A_1$   \\ \hline
&&&  $D_5(a_1)\!+\!A_1$  &&${\bf (2,4), (4,6)}$& 
$A_1$   \\ \hline
$A_5+A_1$  & $(3A_1)'$ & $E_7(a_5)$ & $D_6(a_2)$  &$A_1$ & ${\bf (1,3)}$ & 
$m$   \\ \hline
$ D_5(a_1)+A_1$  &  $2A_2$  &  $E_7(a_5)$   & $E_6(a_3)$      &$A_1$    &  ${\bf (2,4)}$  & 
$A_1$   \\ \hline 
&&&   $D_6(a_2)$   &&  ${\bf (2,4)}$ & 
$A_1$   \\ \hline 
\end{tabular}
\end{center}
\label{table4_e7}
\end{table}

\subsection{Remaining surface singularities} 

The results using \S \ref{section:codimension2} are collected in Table \ref{extra3_e7}.  
We have used the fact that $E_7(a_1)$, $E_7(a_2)$, $E_7(a_3)$, $E_6$, $E_6(a_1)$
have closure which is normal \cite[Section 7.8]{Broer:Decomposition}.   The method from \cite{Sommers:E6}
can be used to show $D_6$ has closure which is normal. 
The entry for $(E_7(a_1), E_7(a_2))$ is from \S \ref{example: E7_subregular}.
For the three degenerations above $E_7(a_5)$, the irreducible components 
of $\slice$ are normal (see \S \ref{additional_E7}).

The remaining six minimal degenerations are unibranch, but either the larger orbit has non-normal closure or it is not known whether the larger orbit has closure which is normal.  In all cases we are able to determine that the slice is normal and hence fully determine the singularity.  
The corresponding action of $A(e)$ is determined using \S \ref{section:codimension2}.
The degeneration $(D_5, E_6(a_3))$ is $C_3$ and $(D_5(a_1),A_4+A_1)$ is $A^+_2$ by restriction to $E_6$, see Table \ref{restriction_subalgebra}.
The other four degenerations follow from Table \ref{full_nilcones}.

\begin{table}[htp]
\caption{Surface singularities using \S \ref{section:codimension2}: $E_7$}
\begin{center}
\begin{tabular}{|c|c|c|c|c|c|c|c|}
\hline Degeneration &  Induced from  & $\sharp \mathbb{P}^1$'s &$A(e)$ & $\sharp$ orbits of $A(e)$ &  $\slice$ \\
\hline
$(E_7(a_1), E_7(a_2))$ & $(A_1, 0)$ & 6 & $1$   & & $D_6$ \\
\hline
$(E_7(a_2), E_7(a_3))$ & $(2A_1, 0)$ & 5 & $S_2$ & 4 &  $C_4$ \\
\hline
$(E_7(a_3),E_6(a_1))$ & $((3A_1)', 0)$ & 5 & $S_2$ & 3 &  $B_3$ \\
\hline
$(E_6, E_6(a_1))$ & $((3A_1)'', 0)$ & 6 & $S_2$ & 4 &  $F_4$ \\
\hline
$(E_6(a_1), E_7(a_4))$ & $(4A_1, 0)$ & 4 & $S_2$ & 3 &  $C_3$ \\ \hline
$(D_6, E_7(a_4))$ & $(D_4, 32^21)$ & 4 & $S_2$ & 3 &  $C_3$ \\
\hline $(A_6, E_7(a_5))$ & $(A_2+3A_1, 0)$ & 4 & $S_3$ & 2 
& $G_2$  \\ \hline
$(D_5+A_1, E_7(a_5))$ & $(2A_2, 0)$ & 4 & $S_3$ & 2 
& $G_2$  \\ \hline
 $(D_6(a_1), E_7(a_5))$ & $(A_3, 0)$ & 12 & $S_3$ & 3 & $3C_3$ \\ \hline
\end{tabular}
\end{center}
\label{extra3_e7}
\end{table}


\subsection{Additional calculations:  three degenerations above $E_7(a_5)$}   \label{additional_E7}


The proofs are similar to the one in \S \ref{4G2} and proceed  by first showing that a larger slice is isomorphic to the whole nilcone of a smaller Lie algebra.

For $(A_6,E_7(a_5))$ and $(D_5+A_1,E_7(a_5))$, we first show
that the degenerations $(A_6,A_5'')$ and $(D_5+A_1,A_5'')$ are both isomorphic to $\NC_{G_2}$. 
Then we use the fact that $E_7(a_5)$ corresponds to the subregular orbit in $G_2$.  
The result follows, as in \S \ref{4G2}, since these singularities are unibranch.  In more detail:
let $e$ be in the orbit $A_5''$.  Then $\cg(\sg)$ is of type $G_2$. Let $e_0$ be a regular nilpotent element in $\cg(\sg)$.  
Then $e + e_0$ lies in the orbit $E_7(a_4)$ and $(m_i,n_i) = (4,6)$ for the unique element in $\EC$.
The simple part of $\gg^h$ is $\so_8(\CM)$.  
Let  $w_i = (\ad f)^2 (M^3)$ with $M = e_0 \in \cg(\sg) \subset \so_7 \subset \so_8$ (\S \ref{powers_of_X}).
Using GAP we showed that there is a unique scalar $b \neq 0$ such that $e + e_0 + bw_i$ is in the orbit $A_6$,
and similarly for $D_5+A_1$. The rest of the proof in \S \ref{F4_non_min} applies to give the result.



For the case of $(D_6(a_1),E_7(a_5))$, we first show that
the degeneration $(D_6(a_1),D_4)$ has one branch which is 
isomorphic to $\NC_{C_3}$.
(There are two branches of $D_6(a_1)$ above $D_4$.)
Let $e$ be in the orbit $D_4$.  Then $\cg(\sg) \cong \sp_6$.   
Let $e_0$ be a regular nilpotent element in $\cg(\sg)$.  
Then $e + e_0$ lies in the orbit $E_7(a_4)$ and $\gg$ decomposes in \eqref{sl2_decomp}
as $$V(0,10) \oplus V(0,6) \oplus V(0,2) \oplus V(2,0) \oplus V(6,4) \oplus   V(6,8) \oplus V(10,0),$$
reflecting that $\cg(\sg)$ decomposes under $\sg_0$ as $V(10) \oplus V(6) \oplus V(2)$ and 
$\gg^f(-6)$ decomposes under $\sg_0$ as $V(4) \oplus V(8)$, which is $14$-dimensional and as
a representation of $\cg(\sg)$ is $V(\omega_2)$.
Also $(m_i,n_i) = (6,8)$ for the unique element in $\EC$.

The semisimple part of $\gg^h$ is isomorphic to $\sl_6(\CM)$. 
If we take $M = e_0$, 
then $M^4 \in \sl_6$ is nonzero since $M$ is regular in $\sl_6$.  It cannot be in $\sp_6$ since only odd powers of $M$ are.
It satisfies $[h_0, M^4] = 8M^4$ and 
so it must be a highest weight vector in $V(8)$ for $\sg_0$ with respect to $e_0$.   
Hence we can choose
$w_i = (\ad f)^3(M^4)$ (\S \ref{powers_of_X}).
We checked using GAP that there is an $x$ in the orbit $D_6(a_1)$ with
$$x = e + e_0 + bw_i$$ with $b \neq 0$.
Since the elements $\overline{C(\sg)\cdot (e_0+bw_i)}$ consist of pairs 
$(M,M^4) \in \sp_6 \oplus V(\omega_2) \cong \sl_6$
with $M \in \sp_6$ nilpotent, 
the slice for $(D_6(a_1),D_4)$ contains an irreducible component isomorphic to $\NC_{C_3}$ (the dimensions match).
Since elements in the slice belonging to the $E_7(a_5)$-orbit correspond to the subregular elements in $\NC_{C_3}$,
one branch of $(D_6(a_1),E_7(a_5))$ is isomorphic to $C_3$, hence the singularity is $3C_3$.

%

\section{Results for $E_8$}  \label{e8}

\subsection{Details in the proof of Proposition \ref{dim4+_1}}\label{E8:slice cases}

In Table \ref{table1_e8} we list the cases where Corollary \ref{slice_corollary} applies. 
The centralizer $\cg(\sg_0)$ is the semisimple part of a Levi subalgebra of type $E_7$.  The nonzero nilpotent $G$-orbits meeting $\cg(\sg_0)$ are those which appear in the table.  The first several lines of Table \ref{table4_e8} contain the remaining cases 
where Lemma \ref{kp_prop} applies 
with $x = e + e_0$ for an element $e_0$ in a minimal nilpotent orbit of $\cg(\sg)$.
  
\begin{table}[htp]
\caption{${\bf E_8}$:  cases with $e_0 \in \cg(\sg)$ of type $A_1$ in $\gg$}
\begin{center}
\begin{tabular}{|c|c|c|c|}  \hline
$e$  &  $e+e_0 \in \0$  & $\cg(\sg)$   &   Isomorphism \\
&&& type of $\SC_{\0, e}$  \\ \hline 
$A_1$  &   $2A_1$ &   $E_7$  &    $e_7$ \\ \hline
$2A_1$  &  $3A_1$ &   $B_6$    &  $b_6$ \\ \hline
$3A_1$   &  $4A_1$  &   ${\bf F_4}+A_1$   &   $f_4$   \\ \hline
                &  $A_2$  &  $F_4 + {\bf A_1}$   & $A_1$   \\ \hline
$A_2$ &         $A_2+A_1$ &  $E_6$ & $e^+_6$   \\ \hline
$4A_1$   &  $5A_1 = A_2+A_1$  &   $C_4$   &   $c_4$   \\ \hline
$A_2+A_1$  &  $A_2+2A_1$ &  $A_5$  &  $a^+_5$   \\ \hline
$A_2+2A_1$  &  $A_2+3A_1$  &  ${\bf B_3}+A_1$ &  $b_3$   \\ \hline
$A_2+3A_1$  &  $A_2+4A_1 = 2A_2$  &  $G_2 + {\bf A_1}$ &  $A_1$   \\ \hline
$A_3$ &  $A_3+A_1$   & $B_5$   &   $b_5$   \\ \hline
$2A_2$   &  $2A_2+A_1$  &   ${\bf 2G_2}$   &  $[2g_2]^+$   \\ \hline
$2A_2+A_1$  &  $2A_2+2A_1$ &  ${\bf G_2}+A_1$  &  $g_2$   \\ \hline
$A_3 + A_1$  & $A_3+2A_1$    &  ${\bf B_3}+ A_1$  &  $b_3$    \\ \hline
                & $(A_3+2A_1)'' = D_4(a_1)$    &  $B_3+{\bf A_1}$  &  $A_1$    \\ \hline
$A_3 + 2A_1$  & $A_3+3A_1 = D_4(a_1) + A_1$    &  ${\bf B_2}+ A_1$  &  $b_2$    \\ \hline
 $D_4(a_1)$  &  $D_4(a_1)+A_1$  & $D_4$  &  $d_4^{++}$  \\ \hline
$D_4(a_1)+A_1$ & $D_4(a_1)+2A_1= A_3+A_2$ &  ${\bf 3A_1}$  &  $[3A_1]^{++}$   \\ \hline
$A_3+A_2$ & $A_3+A_2+A_1$  & $B_2+T_1$ & $b_2$   \\ \hline
$A_3+A_2+A_1$ &  $A_3+A_2+2A_1=D_4(a_1)+A_2$  & $A_1+{\bf A_1}$  & $A_1$  \\ \hline
$A_4$ &  $A_4+A_1$  & $A_4$  &  $a^{+}_4$  \\ \hline
$D_4$ &  $D_4+A_1$   &   $F_4$  &  $f_4$   \\ \hline
$D_4+A_1$  & $D_4+2A_1 = D_5(a_1)$ & $C_3$  & $c_3$   \\ \hline
$A_4+A_1$ &  $A_4+2A_1$  & $A_2+T_1$  &  $a^{+}_2$  \\ \hline
$D_5(a_1)$   &  $D_5(a_1) + A_1$  &   $A_3$   &   $a^{+}_3$  \\ \hline
$A_4+A_2$ &  $A_4+A_2+A_1$  & $A_1+{\bf A_1}$  &  $A_1$  \\ \hline
$D_5(a_1)+A_1$   &  $D_5(a_1) +2A_1 = D_4+A_2$  &   $A_1+{\bf A_1}$   &   $A_1$  \\ \hline
$A_5$ &   $A_5+A_1$ & ${\bf G_2}+A_1$ & $g_2$   \\ \hline
           &   $A_5+A_1 = E_6(a_3)$ & $G_2+{\bf A_1}$ & $A_1$  \\ \hline
$A_5+A_1$ & $A_5+2A_1 = E_6(a_3)+A_1$ & $A_1+{\bf A_1}$ & $A_1$  \\ \hline
$E_6(a_3)$ & $E_6(a_3)+ A_1$  & $G_2$  & $g_2$   \\ \hline
$D_6(a_2)$  &  $D_6(a_2) + A_1 = E_7(a_5)$  &   ${\bf 2A_1}$    &   $[2A_1]^+$  \\ \hline
$D_5$ &  $D_5+A_1$  &  $B_3$  &  $b_3$  \\ \hline
$E_7(a_5)$ & $E_7(a_5)+A_1 = E_8(a_7)$  & $A_1$ & $A_1$  \\ \hline
$D_5+A_1$ &  $D_5+2A_1= D_6(a_1)$  &  $A_1+{\bf A_1}$  &  $A_1$  \\ \hline
$D_6(a_1)$   &  $D_6(a_1) + A_1 = E_7(a_4)$  &   ${\bf 2A_1}$  &   $[2A_1]^+$  \\ \hline
$A_6$   &  $A_6 + A_1$  &   $A_1+{\bf A_1}$  &   $A_1$  \\ \hline
$E_7(a_4)$ & $E_7(a_4)+A_1 = D_5+A_2$  & $A_1$ & $A_1$  \\ \hline
$E_6(a_1)$ & $E_6(a_1)+ A_1$  & $A_2$  & $a^{+}_2$   \\ \hline
$D_6$   &  $D_6 + A_1 = E_7(a_3)$  &   $B_2$ &   $b_2$  \\ \hline
$E_6$ & $E_6+ A_1$  & $G_2$  & $g_2$   \\ \hline
$E_7(a_3)$ &  $E_7(a_3)+A_1 = D_7(a_1)$  & $A_1$ & $A_1$  \\ \hline
$E_7(a_2)$ &  $E_7(a_2)+A_1 = E_8(b_5)$  & $A_1$ & $A_1$  \\ \hline
$E_7(a_1)$ &  $E_7(a_1)+A_1 =E_8(b_4)$  & $A_1$ & $A_1$  \\ \hline
$E_7$ &  $E_7+A_1 = E_8(a_3)$  & $A_1$ & $A_1$  \\ \hline
\end{tabular}
\end{center}
\label{table1_e8}
\end{table}

The remaining cases of the proposition, where $J \neq \emptyset$,
are listed in Table \ref{table4_e8} and include two non-surface cases. 
All cases follow by restriction to a subalgebra or 
by using Lemma \ref{S_var_sl2} and Corollary \ref{A1_or_m},
except for the two degenerations above $A_4+A_3$.
We now discuss those two cases and the two non-surface cases,
but omit the details for the other degenerations.
 
\subsubsection{The degeneration $(A_3+2A_1, 2A_2+2A_1)$}  \label{dim_4_case_1}

Here $e$ is in the orbit $2A_2+2A_1$ and $\cg(\sg) \cong \sp_4(\CM)$.  
Let $e_0$ be in the minimal nilpotent orbit of $\cg(\sg)$.
In this case $\EC$ has one element corresponding to $(m_i,n_i) = (1,3)$.   
Consider the Levi subalgebra $\lg$ of type $E_6+A_1$.  Then without loss of generality
$e \in \lg$ (with nonzero component on the $A_1$ factor) and $e_0 \in \lg$ (contained in the $E_6$ factor).  
By the results for $E_6$, there is an $x$ in the orbit $\0$ of type $A_3+2A_1$ (in $E_8$)
with $x = e + e_0 + e_1$ for a choice of $e_1 \in \gg^f(-1)$ corresponding to $(1,3)$.   
Moreover, writing $\cg(\sg) = V(2\omega_1)$, then $e_1$ is a highest weight vector for a $\cg(\sg)$-module $V(3\omega_1) \subset \gg^f(-1)$.  Hence $\slice = e + X(2\omega_1, 3\omega_1) \cong m'$ since \eqref{dim_equality_full_e0} holds and the singularity is unibranch.

\subsubsection{The degenerations $(A_4+A_1, D_4(a_1)+A_2)$ and $(2A_3, D_4(a_1)+A_2)$}  \label{dim_4_case_2}

Here $\cg(\sg) \cong \sl_3(\CM)$.
All the orbits meet the semisimple subalgebra $\lg$ of $\gg$ of type $D_5 + A_3$:
$\0_e$ meets $\lg$ in the orbit $[3^3, 1] \cup [4]$;
$A_4+A_1$ meets $\lg$ in the orbit $[5,2^2,1] \cup [4]$;
and $2A_3$ meets $\lg$ in the orbit $[4^2,1^2] \cup [4]$.
Then just as in the case $(B_2,  A_2+\tilde{A}_1)$ in \S \ref{F4hard_m}, 
there exists $x \in \0$ with
$x = e + e_0 + e_2$ for some $e_2 \in \gg^f(-2)$ corresponding to the pair $(2,4)$,
for $\0$ either of type $A_4+A_1$ or type $2A_3$.
Identifying $\cg(\sg)$ with $V(\theta)$ where $\theta$ is a highest root of $\cg(\sg)$, 
we have $e_2$ is a highest weight vector for 
a $\cg(\sg)$-module $V(2 \theta) \subset \gg^f(-2)$. 
Hence for both orbits $\slice = e + X(\theta, 2\theta) \cong X(\theta)$, as desired (for two different choices of $e_2$, related by a scalar).

\subsubsection{The degenerations $(A_5+A_1, A_4+A_3)$ and $(D_5(a_1)+A_2, A_4+A_3)$}  \label{e8_hard_m}

Corollary \ref{A1_or_m} applies, but is not sufficient to pin down the singularity, so we carry out an explicit computation.
In both of these cases, $e$ lies in the orbit $A_4+A_3$, for which $\cg(\sg) \cong \sl_2(\CM)$.  
This case is a more complicated version of \S \ref{F4hard_m} in $F_4$.
Using the information in \cite[p. 146]{Lawther-Testerman} (adjusted for sign differences in GAP), let $$e=-(4f_{\alpha_1}+6f_{\alpha_3}+6f_{\alpha_4}+4f_{\alpha_2}+3f_{\alpha_6}+4f_{\alpha_7}+3f_{\alpha_8}),\;\; f=\sum_{i\neq 5}e_{\alpha_i}\;\;\mbox{and}\;\; h=[e,f].$$
A nilpositive element in $\cg(\sg)$ is
$$e_0=2e_{\begin{aligned} 11 &22211 \\[-3mm] &1\end{aligned}}-e_{\begin{aligned} 12 &32110 \\[-3mm] &1\end{aligned}}+2e_{\begin{aligned} 01 &22221 \\[-3mm] &1\end{aligned}}+e_{\begin{aligned} 12 &32100 \\[-3mm] &2\end{aligned}}+e_{\begin{aligned} 12 &22210 \\[-3mm] &1\end{aligned}}+e_{\begin{aligned} 12 &22111 \\[-3mm] &1\end{aligned}},$$ embedded in an $\sl_2$-triple 
$\{ e_0, h_0, f_0\}$ for $\cg(\sg)$.
Then the three elements in $\EC$ correspond to $\{(1,3),(2,4),(3,5)\}$.
The spaces $\gg^f(-1)$, $\gg^f(-2)$ and $\gg^f(-3)$ contain highest weight modules for $\cg(\sg)$ with respective highest weights 3, 4 and 5, and highest weight vectors:
$$e_1=e_{\begin{aligned} 24 &43210 \\[-3mm] &2\end{aligned}}-e_{\begin{aligned} 13 &43211 \\[-3mm] &2\end{aligned}}+e_{\begin{aligned} 12 &43221 \\[-3mm] &2\end{aligned}}-e_{\begin{aligned} 12 &33321 \\[-3mm] &2\end{aligned}},\;\; 
e_2 = e_{\begin{aligned} 24 &54321 \\[-3mm] &2\end{aligned}}+e_{\begin{aligned} 23 &54321 \\[-3mm] &3\end{aligned}},\;\; 
e_3 = e_{\begin{aligned} 24 &65432 \\[-3mm] &3\end{aligned}}.$$
We checked in GAP that 
\begin{eqnarray*}
e+e_0+3e_2 \pm(2e_1+4e_3)\in D_5(a_1)+A_2,\;\;\mbox{and}\;\;  \\
e+e_0-\tfrac{7}{6}e_2\pm\sqrt{\tfrac{8}{27}}(e_1-\tfrac{19}{3}e_3)\in A_5+A_1.
\end{eqnarray*}
Hence in both cases $\slice = e + X(2,3,4,5) \cong X(2,3) = m$.

%

\begin{table}[htp]  
\caption{${\bf E_8}$:  Remaining relevant cases with $e_0$ minimal in $\cg(\sg)$}
\begin{center}
\begin{tabular}{|c|c|l|l|c|c|c|c|}  \hline
$e$  &  $e_0$   &  $e+e_0$  & $\0$ & $\cg(\sg)$   &   $(m_i, n_i)$ for  & Isomorphism \\
&&&& & $i \in \EC$ &  type of $\slice$  \\ \hline 
$2A_3$ &  $2A_1$ & $A_4+2A_1$  & $A_4+2A_1$  & $B_2$   &   $\emptyset$ &  $b_2$  \\ \hline  
$D_4+A_2$   &    $2A_1$  &  $D_5(a_1)+A_2$  &   $D_5(a_1)+A_2$  & $A_2$   &   $\emptyset$ & $a_2^+$  \\ \hline
$A_6$   &  $A_2+3A_1$  &  $E_7(a_4)$ &  $E_7(a_4)$ & ${\bf A_1}+A_1$ & $\emptyset$ & $A_1$  \\ \hline
$A_4+2A_1$   &  $2A_1$  &  $A_4+A_2$&  $A_4+A_2$  & $A_1+T_1$ & $\emptyset$ & $A_1$  \\ \hline
$A_6+A_1$   &  $A_2+3A_1$  &  $D_5+A_2$   &  $D_5+A_2$  & $A_1$ &$\emptyset$ & $A_1$  \\ \hline
$A_7$   &  $4A_1$  &  $E_8(b_6)$ &  $E_8(b_6)$  & $A_1$ & $\emptyset$ & $A_1$ \\ \hline
$D_7$   &  $2A_1$  &  $E_8(a_5)$ &  $E_8(a_5)$ & $A_1$ &  $\emptyset$ & $A_1$ \\ \hline
\hline
$A_2+2A_1$  & $A_2+2A_1$ & $D_4(a_1)$  &   $A_3$  & $B_3+{\bf A_1}$  &  $(1,3)^8, {\bf (2,4)}$ &  
$A_1$  \\ \hline
$2A_2+A_1$  & $3A_1$ & $D_4(a_1)$ &  $A_3+A_1$   &$G_2+ {\bf A_1}$  & ${\bf (1,3)}$ & 
$m$\\ \hline
 $2A_2+2A_1$  & $3A_1$ & $D_4(a_1)+A_1$ & $A_3+2A_1$ & $B_2$     & ${\bf (1,3)}$ & 
$m'$ \\ \hline
$A_3+A_2+A_1$    &  $A_4+A_2$  &  $E_7(a_5)$  &   $D_4+A_1$   & ${\bf A_1} +A_1$  & ${\bf (2,4)}, (2,8), {\bf (4,6)}$ &  
$A_1$\\ \hline

$D_4(a_1)+A_2$  &  $A_2+2A_1$   &    $A_4+2A_1$   &  $2A_3$     &$A_2$  &   ${\bf (2,4)}$  & 
$a_2^+$   \\ \hline
				&   &    &   $A_4+A_1$    &&    ${\bf (2,4)}$  & 
				$a_2^+$   \\ \hline
$A_4+A_2$  & $A_3+A_2+A_1$ & $E_7(a_5)$ &  $A_5$   &${\bf A_1}+A_1$ &  ${\bf (2,4)}, {\bf (4,6)}$ & 
$A_1$ \\ \hline
&&& $D_5(a_1)\!+\!A_1$   && ${\bf (2,4)}, {\bf (4,6)}$ & 
$A_1$\\ \hline
$D_5(a_1)+A_1$  &  $2A_2$  &  $E_7(a_5)$   &  $E_6(a_3)$   & ${\bf A_1}+A_1$    &  ${\bf (2,4)}$  & 
$A_1$\\ \hline
$A_4+A_2+A_1$  & $A_3+A_2+A_1$ & $E_8(a_7)$ & $D_4\!+\!A_2$  &$A_1$  & $(1,5), \bf{(2,4)}, $ & 
$A_1$  \\ 
&&&&& $(3,5), \bf{(4,6)}$  &  \\ \hline

$A_4+A_3$ & $2A_2+2A_1$ & $E_8(a_7)$ &  $A_5\!+\!A_1$   &$A_1$  & ${\bf (1,3), (2,4), (3,5)}$  & 
$m$\\ \hline
&&&  $D_5(a_1)\!+\!A_2$  &&  ${\bf (1,3), (2,4), (3,5)}$ & 
$m$\\ \hline
$A_5+A_1$  & $3A_1$ & $E_7(a_5)$ & $D_6(a_2)$     &${\bf A_1}+A_1$   & ${\bf (1,3)}$ & 
$m$ \\ \hline
$D_5(a_1)+A_2$  &  $A_2+2A_1$  &  $E_8(a_7)$  &  $E_6(a_3)\!+\!A_1$  &$A_1$  & ${\bf (1,3), (2,4)}$  &  
$m$ \\ \hline
		 & && $D_6(a_2)$   && $ (1,3), {\bf (2,4)}$  &  $A_1$ \\ \hline
$E_6(a_3)+A_1$ & $3A_1$ & $E_8(a_7)$ & $E_7(a_5)$    &$A_1$  & ${\bf (1,3)}$  & 
$m$ \\ \hline
$E_6+A_1$ & $3A_1$ & $E_8(b_5)$ &  $E_7(a_2)$   &$A_1$  & ${\bf (1,3)}$  & 
$m$ \\ \hline
\end{tabular}
\end{center}
\label{table4_e8}
\end{table}%

\subsection{Remaining surface singularities, and an exceptional degeneration}  \label{remaining_surfaces}
The results using \S \ref{section:codimension2} are collected in Table \ref{extra3_e8}.
We use the fact that $E_8$, $E_8(a_1)$, $E_8(a_2)$, $E_8(a_3)$, $E_8(a_4)$
have closure which is normal \cite[Section 7.8]{Broer:Decomposition}.   The method from \cite{Sommers:E6} 
can be used to show $E_7, E_8(b_4)$, and $E_8(a_5)$ have closure which is normal. 
The entry for $(E_8(a_1), E_8(a_2))$ is from \S \ref{example: E8_subregular}.
That each irreducible component of $(D_6(a_1), E_8(a_7))$ and $(A_6, E_8(a_7))$ is $G_2$ follows from the fact that the 
degeneration $(E_6, D_4)$ contains a branch isomorphic to the nilpotent cone in $F_4$, and then from the results in $F_4$. 

There are 19 other cases.  For nine of them,
the degenerations are unibranch, but either the larger orbit has non-normal closure or it is not known whether the larger orbit has closure which is normal.   Nevertheless, in these cases we are able to show that the slice is normal and hence fully determine the singularity.  
The action of $A(e)$ is determined using \S \ref{section:codimension2}.   The degeneration $(D_5, E_6(a_3))$ is $C_3$ 
and the degeneration $(D_5(a_1), A_4+A_1)$ is $A^+_2$, both by restriction to $E_6$ (see Table \ref{restriction_subalgebra}).
The other degenerations follow from Table \ref{full_nilcones}.
For the other ten cases, the result is determined up to normalization.  
In four of these cases, the orbit closure is known to be non-normal:
$(E_7(a_1),E_8(b_5)) ,(E_7(a_3),E_6(a_1)+A_1), (D_7(a_2),D_5+A_1), (D_6,D_5+A_2).$ 
The latter three are unibranched.  The orbit closures, and hence the slices, for the other six are expected to be normal.
We use $(Y)$ to denote a singularity with normalization $Y$.

The exceptional degeneration $(A_4 +A_3, A_4 +A_2 +A_1)$ of codimension four is treated in \S \ref{dim4_exceptions}.

\begin{table}[htp]
\caption{Surface singularities using \S \ref{section:codimension2}: $E_8$}
\begin{center}
\begin{tabular}{|c|c|c|c|c|c|c|} \hline 
Degeneration &  Induced from  & $\sharp \mathbb{P}^1$'s &$A(e)$ & $\sharp$ orbits   of $A(e)$  &  $\slice$ \\ \hline
$(E_8(a_1), E_8(a_2))$  & $(A_1, 0)$ &  7 & $1$ &  & $E_7$ \\ \hline
$(E_8(a_2), E_8(a_3))$ & $(2A_1, 0)$ & 7 & $S_2$ & 6 & $C_6$ \\ \hline
$(E_8(a_3), E_8(a_4))$ & $(3A_1, 0)$ & 6 & $S_2$ & 4 & $F_4$ \\ \hline
$(E_8(a_4), E_8(b_4))$ & $(4A_1, 0)$ & 5 & $S_2$ & 4 & $C_4$ \\  \hline
$(E_8(a_5), E_8(b_5))$ & $(A_2+3A_1, 0)$ & 4 & $S_3$ & 2 &  $G_2$ \\ \hline $(E_7(a_1), E_8(b_5))$ & $(A_3, 0)$ & 18 & $S_3$ & 5 & $3(C_5)$ \\ \hline
$(E_8(b_5), E_8(a_6))$ & $(2A_2+A_1, 0)$ & 4 & $S_3$ & 2 & $(G_2)$ \\ \hline
$(E_7(a_3), E_6(a_1)+A_1)$ & $(D_6, 3^22^21^2)$ & 4 & $S_2$ & 2 &  $(A_4^+)$ \\
\hline $(D_7(a_2), D_5+A_2)$ & $(2A_3, 0)$ & 3 & $S_2$ & 2 & $(C_2)$ \\
\hline $(E_7, E_8(b_4))$ & $(D_4, 32^21)$ & 6 & $S_2$ & 4 & $F_4$ \\
\hline $(D_7, E_8(a_6))$ & $(D_4+A_2, 32^21+0)$ & 4 & $S_3$ & 2  & $(G_2)$ \\
\hline $(E_8(b_4), E_8(a_5))$ & $(A_2+2A_1, 0)$ & 4 & $S_2$ & 3 & $C_3$ \\
\hline $(E_7(a_2), D_7(a_1))$ & $(D_5, 32^21^3)$ & 5 & $S_2$ & 3 & $(B_3)$ \\
\hline $(D_7(a_1), E_8(b_6))$ & $(A_3+A_2, 0)$ & 3 & $S_3$ & 2 & $(C_2) = \mu$ \\
\hline $ (E_6+A_1, E_8(b_6))$ & $(E_6, 2A_2+A_1)$ & 4 & $S_3$ & 2& $(G_2)$ \\
\hline $(A_7, D_7(a_2))$ &  $(D_5+A_2, 32^21^3+0)$ & 2 & $S_2$ & 1 & $(A_2^+)$ \\
\hline $ (E_6(a_1)+A_1, D_7(a_2))$ & $(E_7, A_4+A_1)$ &  2 & $S_2$ & 1 & $(A_2^+)$ \\
\hline $(D_6, D_5+A_2)$ & $(D_6, 32^41)$ & 3 & $S_2$ & 2& $(C_2)$ \\
\hline $(D_6(a_1), E_8(a_7))$ & $(A_5, 0)$ & 40 & $S_5$ & 2 
& $10G_2$\\
\hline  $(A_6, E_8(a_7))$ & $(D_4+A_2, 0)$ & 20 & $S_5$ & 2 & 
$5G_2$ \\
\hline
\end{tabular}
\end{center}
\label{extra3_e8}
\end{table}


\begin{table}[htp]
\caption{Some surface cases where Lemma \ref{singequi} can be applied}
\begin{center}
\begin{tabular}{|c|c|c|c|c|c|c|c|}  \hline
$\gg$& $e$  &  $x \in \0$  & subalgebra   &  $\slice$  \\ \hline

$E_7, E_8$  & $E_6(a_3)$ & $D_5$ & $E_6$ & $C_3$ \\ \hline 

$E_7, E_8$  & $A_4+A_1$ & $D_5(a_1)$ & $E_6$ & $A_2^+$\\ \hline
$E_7, E_8$  & $A_3+A_2$ & $A_4$ & $D_6$ &  $C_2$ \\ \hline
\end{tabular}
\end{center}
\label{restriction_subalgebra}
\end{table}

\section{Slices related to entire nilcones} \label{entire_nilcones}

The main goal of the paper was to study $\slice$ for a minimal degeneration.
Many of the same ideas can be used to show that $\slice$ has a familiar description when the degeneration is not minimal.  In particular, there are many cases where
$\slice$ is isomorphic to the closure of a non-minimal orbit in a nilcone for a subalgebra of $\gg$ or is isomorphic to a slice between two orbits in such a nilcone.  
Rather than listing all these cases here, we write down some cases where $\slice$, or one of its irreducible components, is isomorphic to an entire nilcone.   
Some of these were used to show in the surface case that $\slice$, or an irreducible component of $\slice$, is normal (e.g., starting with \S \ref{4G2}).
These examples are relevant for the duality discussed in \S \ref{dualitysubsection}, to be explored in future work.
They are also examples where $C(\sg)$ acts with a dense orbit.

\subsection{Exceptional groups}

The results are listed in Table \ref{full_nilcones}.   The notation $\NC_{X}$ refers to the nilcone in the Lie algebra of type $X$.
The proofs use Lemma \ref{kp_prop}, usually for $x \neq e+e_0$, and often require a computer calculation.

\begin{table}[htp]
\caption{Slices containing a smaller nilcone}
\begin{center}
\begin{tabular}{|c|c|}  \hline
$\gg$  &    Degeneration and nilcone \\ \hline
$F_4$ & $(B_3, \tilde{A}_2) = \NC_{G_2}$ 
\\ \hline
&  $(C_3, \tilde{A}_2) \supset \NC_{G_2}$ 
\\ \hline  \hline
$E_6$  & $(E_6(a_3), D_4) =  \NC_{A_2}$ \\ \hline
& $(A_4, A_3) \supset \NC_{C_2}$  \\ \hline
& $(2A_2, A_2) = [2\NC_{A_2}]^+$  \\ \hline  \hline
$E_7$   &   $(E_7(a_4), D_5) = \NC_{2A_1}$  \\ \hline
& $(D_6(a_1), D_4) \supset \NC_{C_3}$  \\ \hline
  & $(E_7(a_5), A'_5)   = \NC_{2A_1}$  \\ \hline
& $(A_6,A''_5) =  \NC_{G_2}$  \\ \hline
& $(D_5+A_1,A''_5) = \NC_{G_2}$  \\ \hline
& $(A_4+A_2, A_4) =  \NC^+_{A_2}$  \\ \hline 
 & $(D_4,A_2+3A_1) = \NC_{G_2}$  \\ \hline
& $(D_4,2A_2) = \NC_{G_2}$  \\ \hline
& $(D_4(a_1)+A_1, (A_3+A_1)')   = \NC_{2A_1}$  \\ \hline
& $(A''_5,A_3) \supset \NC_{B_3}$ \\ \hline  \hline
$E_8$ &  $(E_8(a_5), E_6) = \NC_{G_2}$ \\ \hline
& $(E_8(a_6), D_6) = \NC_{C_2}$ \\ \hline  
& $(E_6, D_4) \supset \NC_{F_4}$  \\ \hline 
& $(D_5+A_1,A_5) \supset \NC_{G_2}$  \\ \hline
& $(A_6,E_6(a_3)) \supset \NC_{G_2}$  \\ \hline
& $(D_6(a_1), E_6(a_3)) \supset \NC_{G_2}$ \\ \hline
&  $(E_8(b_6), E_6(a_1)) = \NC^+_{A_2}$ \\ \hline 
 & $(A_4,A_3+2A_1) \supset \NC_{C_2}$  \\ \hline
 &  $(D_4,2A_2) = 2 \NC_{G_2}$ \\ \hline 
\end{tabular}
\end{center}
\label{full_nilcones}
\end{table}%

\subsection{Slices isomorphic to entire nilcones:  two slices in $\sl_N$}

These two examples are special cases of isomorphisms discovered by Henderson \cite{Anthony_from_down_under:survey} using Maffei's work on quiver varieties
\cite{maffei:quiver}.   Here we give direct proofs that fit into the framework of Lemma \ref{kp_prop} and \S \ref{extra_eigenvalues}.   
We are grateful to Henderson for bringing these examples to our attention.

\subsubsection{First Slice} 
It is slightly more convenient to work in $\gg = \gl_{nk}$.  Assume $n \geq 2$ and $k \geq 1$.   Consider the nilpotent orbit $\OC'$ with partition $[n^k]$.   
Write  $k = p(n+1) + q$ with $0 \leq q < n+1$, which gives $kn = (pn+q-1)(n+1) + (n+1-q)$ for maximally 
dividing $kn$ by $n+1$ .   Let $\OC$ be the nilpotent orbit with partition $[(n+1)^{pn+q-1}, n+1- q ]$, which is a partition of $kn$.   
Then $\OC' \subset \overline{\OC}$ by the dominance order for partitions.
Moreover, $X \in \cl$ implies $X^{n+1} = 0$ and $\OC$ is maximal for nilpotent orbits in $\gl_{nk}$ with this property.

\begin{proposition}\cite[Corollary 9.5]{Anthony_from_down_under:survey}
Let $e \in \OC'$. The variety $\slice$
is isomorphic to $$\YC := \{ Y \in \gl_k \ | \ Y^{n+1} = 0 \}.$$
In particular, $\slice$ is isomorphic to 
the closure of the nilpotent orbit in $\gl_k$ with partition $[(n+1)^p, q]$, which is the whole nilcone when $k \leq n+1$.
\end{proposition}

\newcommand{\quot}{\ensuremath{/ \hspace{-1mm}/}}

\begin{proof}
Let $I_k$ be the $k \times k$ identity matrix.     Define $e=(e_{ij})$, $h = (h_{ij})$, and $f = (f_{ij})$ to be
$n \times n$-block matrices, with blocks of size $k \times k$, as follows:
\begin{eqnarray*}
   e_{ij} = \begin{cases}
               j(n-j) I_k     & i = j+1  \\
               0           &  \text{else}
                 \end{cases},  &
   h_{ij} = \begin{cases}
               (2i - n -1) I_k        &  i = j  \\
               0           &  \text{else}
           \end{cases},  &
   f_{ij} = \begin{cases}
               I_k        &  j = i+1  \\
               0           &  \text{else}
           \end{cases}
 \end{eqnarray*}
The Jordan type of $e$ and $f$ is $[n^k]$, and so $e, f \in \0'$.  The elements $\{e, h, f\}$ are
a standard basis of an $\sl_2$-subalgebra $\sg$, as in the $k=1$ case.   Also, as in the $k=1$ case, 
the centralizer $\gg^f$ consists of $n \times n$-block matrices $Z=(z_{ij})$ of the form
 \begin{equation*}
   z_{ij} = \begin{cases}
               Y_{j-i}     &  j \geq i \\
               0           &  \text{otherwise}
           \end{cases}
\end{equation*}
for any choice of $Y_0, Y_1, \dots, Y_{n-1} \in \gl_k$.   We abbreviate this matrix by $Z(\{Y_i\})$.
In particular, $\cg(\sg) \cong \gl_k$ consists of the matrices of the form $Z(\{Y, 0, \dots, 0 \})$. 

We are interested in
$$\slice:= \SC_{e} \cap \cl = \SC_e \cap \{ X \in \gg \ | \ X^{n+1} = 0\},$$
where as before $\SC_e = e + \gg^f$.   Let $M = e + Z(\{Y_i\}) \in \SC_e$.
Set $Y_0 = -\frac{1}{n} Y$ for a fixed matrix $Y$ for reasons that will become clear shortly.
Since $M^{n+1} = 0$, we can find constraints on the entries of $M^{n+1}$.
The $(n, 1)$-entry of $M^{n+1}$ is equal to $rY_1 + s Y^2_0$ where $r$ is a sum of products of the coefficient in $e$, hence nonzero.  
Thus $rY_1 + s Y^2_0 = 0$ and $Y_1$ is proportional to $Y^2$.    Given this fact,  
the $(n,2)$-entry of $M^{n+1}$ is equal to $r'Y_2 + s' Y^3_0$ where $r'$ is nonzero.   Hence 
$Y_2$ is proportional to $Y^3$, and so on.  In this way, we conclude that 
$Y_i = c_i Y^{i+1}$ for all $i = 0, 1,2, \dots, n-1$, where the $c_i \in \CM$ are uniquely determined constants (which depend on $n$, but not $k$). 
Consequently $M \in \slice$ takes the form $e + Z(\{c_i Y^{i+1} \})$ for some $Y$.
We were not able to find a general formula for the $c_i$'s, but in all cases that we computed, the $c_i$'s were nonzero, which we expect to be true in general.

Now let $T^n + \sum_{i=1}^{n-1} a_i T^{n-i} \in \CM[T]$ be the characteristic polynomial for the $n \times n$-matrix $e + Z(\{c_i I_1\})$
in the $k=1$ case.  A direct computation  with block matrices then shows that 
 $p(T) := T^n + \sum_{i=1}^{n-1} a_i Y^i T^{n-i}$ is the characteristic polynomial of $M$, viewing $M$ as 
 an $n \times n$-matrix over the commutative ring $\CM[Y]$, where $Y$ acts 
by simultaneous multiplication on each of the block entries of $M$.
By the Cayley-Hamilton Theorem over $\CM[Y]$, it follows that $p(M)=0$.
In fact, $p(T)$ is the minimal polynomial of $M$ over $\CM[Y]$.   Indeed, for $1\leq i\leq n-1$,
the $i$-th block lower diagonal of $M^i$ consists of non-zero scalar matrices while everything below that diagonal is zero.  
Thus $M$ cannot satisfy a polynomial of degree less than $n$ over $\CM[Y]$.

The next step is to show that $Y^{n+1}$ must be the zero matrix.  
Since $p(M) = 0$, 
$$0 = Mp(M)-b_1 Yp(M)=\sum_{i=2}^n (a_i-a_1a_{i-1})Y^iM^{n-i+1}-a_1a_nY^{n+1}.$$
Since the minimal polynomial of $M$ over $\CM[Y]$ has degree $n$, 
it follows that $(a_i - a_1a_{i-1})Y^i = 0$ for $i=2, \dots, n$ and $a_1a_nY^{n+1} = 0$.
Note that $a_1 =1$ by taking the trace of $M$ since $c_0 =  -\frac{1}{n}$.
Now if $Y^{n+1} \neq 0$, then recursively $a_i= a_1^{i} = 1$, but also $a_1a_n = a_n = 0$, a contradiction.
Similarly, if $Y^\ell = 0$ and $Y^{\ell-1} \neq 0$ for some $\ell \leq n+1$, then $a_i = a_1^{i} = 1$ for $i = 1, 2, \dots, \ell-1$. 
We conclude that all elements in $\slice$ take the form $e + Z(\{c_i Y^i \})$ where $Y^{n+1} = 0.$
Hence $\slice$ is isomorphic to a subvariety of $\YC$
via restriction $\pi_0$ of the natural projection $\SC_e \to \cg(\sg)$ by the argument in \S \ref{powers_of_X}.
Now $\slice$ and $\YC$ both have dimension $p^2 n^2+2pqn+p^2n+q^2-q$, and the latter variety
is irreducible; hence $\pi_0$ gives an isomorphism of $\slice$ onto $\YC$.

One consequence is the following: since the $c_i$'s, and hence the $a_i$'s,
are independent of $k$, choosing $k > n$, we deduce that all $a_i =1$, an interesting fact in its own right.
\end{proof}

\begin{remark}
Fix $Y_0 = e_0 \in \cg(\sg)$ in the orbit $[(n+1)^p,q]$ and $Y = -n Y_0$. In the notation of \S \ref{extra_eigenvalues},
the vector $Z( \{ 0, \dots, 0, Y^{i+1}, 0, \dots, 0 \} )$ corresponds to the pair $(i, i+2)$, which lies in $\EC$
when $1 \leq i \leq \min(n, k-1)$.
The proof shows that there is an $x \in \0$ that can be written as in \eqref{x_form1} with $x_i:= c_iZ( \{ 0, \dots, 0, Y^{i+1}, 0, \dots, 0 \} )$
where $0 \leq i \leq \min(n, k-1)$ and such that \eqref{both_conditions} holds.
\end{remark} 

\subsubsection{Second Slice} 

Next, let $\0$ be the orbit in $\gl_{nk}$ with partition $[(n+k-1,(n-1)^{k-1}]$.   Then again $e \in \cl$.  
The elements in $\cl$ correspond to matrices which are nilpotent and which have $\mbox{rank}(M^i) = k(n-i)$ for $i = 1, 2, \dots, n-1$.

\begin{proposition}\cite[Corollary 9.3]{Anthony_from_down_under:survey}
The variety $\slice$ is isomorphic to the nilcone in $\gl_k$.
\end{proposition}

\begin{proof}
Up to smooth equivalence, this result is a consequence of \cite{Kraft-Procesi:GLn}, by cancellation of the first $n-1$ columns of the partitions for $\0$ and $\0'$.   
Here, we show that, in fact, $\slice \cong \NC_{A_{k-1}}$, which also follows from \cite[Corollary 9.3]{Anthony_from_down_under:survey}.

Keep the notation from the proof of the previous proposition.
Let $M \in \SC_e$ satisfying the rank conditions $\mbox{rank}(M^i) = k(n-i)$ for $i = 1, 2, \dots, n-1$.
The last rank condition is $\mbox{rank}(M^{n-1}) = k$.    The bottom, left $2 \times 2$-submatrix of $M^{n-1}$ consists
of 
$\left( \begin{smallmatrix}
 rY_0 & sY_1 \\  tI_k & rY_0 \end{smallmatrix} \right)$, with each of $r, s, t$ positive, since the coefficients of $e$ are positive.
Multiply the last row by $\frac{r}{t}Y_0$ and substract it from the second-to-last row to zero out the $(n-1,1)$-entry.   Then since $\mbox{rank}(tI_k) =k$, it follows that for $\mbox{rank}(M^{n-1}) = k$ to hold, necessarily the second-to-last row must be identically zero.  In particular, the $(n-1,2)$-entry is zero, that is, $Y_1$ is a scalar multiple of $Y_0^2$.    Continuing in this way for the smaller powers of $M$, we conclude that $Y_i = d_i Y_0^i$ for some $d_i \in \CM$, as in the previous proposition.  

Next a direct computation shows that $M^{n+k-1}$ has entry $(n,1)$ which is a scalar multiple of $Y_0^k$ and all other entries are scalar multiples of $Y_0^m$ for $m > k$.  If any of these scalar multiples are nonzero, then since $M^{n+k-1}=0$,
it follows that  $Y_0$ is nilpotent, whence $Y_0^k=0$ since $Y_0 \in \gl_k$.   
These multiples are independent of $k$.  The $k=1$ case implies that
the entries in $M^{n+k-1}$ cannot all be zero unless all $d_i=0$ since $e$
is the only nilpotent element in $\SC_e$.
We have therefore shown that $\slice$ is contained in a variety isomorphic to the nilcone of $\gl_k$.  By dimension reasons, this must be an equality as in the previous proof. 
\end{proof}

\subsubsection{Example}
An example of the first proposition is the degeneration $[2^3]<[3^2]$ and of the second proposition is the degeneration $[2^3] < [4,1^2]$, 
both in $\mathfrak{sl}_6$.  Both slices are isomorphic to the nilcone of $\mathfrak{sl}_3$.
In this setting, the common intermediate orbit $[3,2,1]$ corresponds to the minimal nilpotent orbit in $\mathfrak{sl}_3$.   
Upon restriction to $\sp_6$, the slice becomes isomorphic to the nilcone in $\so_3$, which is of type $A_1$.  This gives another proof of  \S \ref{C3_example_1}, one which does not require knowing that either $[3^2]$ or $[4,1^2]$ have closures which are unibranch at $[2^3]$.
 
\section{The remaining additional singularities}  \label{dim4_exceptions}

The singularities $\mu$ and $\ag_2/{\mathfrak S}_2$ will be discussed in subsequent work.  Here we discuss the minimal degenerations
$(2A_2+A_1,A_2+2A_1)$ in $E_6$ and $(A_4 +A_3, A_4 +A_2 +A_1)$ in $E_8$
and show that they are singularities of type $\tau$ and $\chi$, respectively.
Both cases are related to showing that a larger slice is the {\it cover} of the nilcone in a smaller Lie algebra
(compare this with the cases in \S \ref{entire_nilcones}).
For the case in $E_6$, we show for the degeneration 
$(2A_2+A_1, A_2)$ that the slice is isomorphic to the affinization of a $3$-fold cover 
of the regular nilpotent orbit in $\sl_3(\CM) \oplus \sl_3(\CM)$.
For the case in $E_8$, we show for the degeneration 
$(A_4+A_3,A_4)$ that the slice is isomorphic to the affinization of the universal cover 
of the regular nilpotent orbit in $\sl_5(\CM)$.

\subsection{Preliminaries}
We start with a lemma that extends the results in \S \ref{dim_condition}.  The lemma 
introduces an alternative transverse slice to some orbits, slightly different from the Slodowy slice.
This alternative slice will facilitate the determination of the singularities of the two degenerations in this section.  
It will also be used in subsequent work for other, non-mimimal degenerations.
Since this slice is different from the nilpotent Slodowy slice, we are not able to determine the isomorphism type of the 
nilpotent Slodowy slice, 
and thus the results 
in Theorem \ref{main_theorem}
are stated only up to smooth equivalence.


\begin{lemma}\label{slicelemma}
Let $e$ be a nilpotent element in $\gg$, and let $\sg := \langle e,h,f\rangle$ be an $\mathfrak{sl}_2$-subalgebra
containing it. Next, let $e_0$ be a nilpotent element in $\cg(\sg)$, and let $\sg_0 := \langle e_0, h_0, f_0 \rangle$
be an $\mathfrak{sl}_2$-subalgebra of $\cg(\sg)$.
Suppose condition \eqref{dim_equality_full_e0} is satisfied:
$$\dim C(\sg)\cdot e_0 = \codim_{\cl_{e+e_0}} \0_e.$$
Then $$\SC'_{e+e_0}:=e+e_0+\cg(\sg)^{f_0}\oplus\sum_{i<0}{\mathfrak g}^f(i)$$ is a
transverse slice in $\gg$ to $\0_{e+e_0}$ at $e+e_0$, where
 $\gg^f(i)$ denotes the $\ad h$-eigenspace for the eigenvalue $i$ in $\gg^f$.
\end{lemma}

\begin{proof}
Decompose $\gg$ under $\sg \oplus \sg_0$ as in \eqref{sl2_decomp}.  Then by Proposition \ref{prop:codimensions_match},
the dimension hypothesis ensures that the summands $V^{(i)}_{m_i, n_i}$ satisfy $m_i \geq n_i$ whenever $m_i > 0$. 
 
 Let $V(m,n)$ be such a summand with $m \geq n$ and $m>0$, and consider the action of $\sg \oplus \sg_0$ on $V(m,n)$.
 Then $\dim \ker f = n+1$ and $\dim \ker f_0 = m+1$.
As discussed in \S \ref{dim_condition}, 
$V(m,n)$ decomposes into $n+1$ irreducible representations under the action of the $\sl_2$-subalgebra $\langle e+e_0, h+h_0, f+f_0 \rangle$.
Therefore $\dim \ker (f+f_0) = n+1$ and so $\dim \ker f = \dim \ker( f+f_0)$.
Now $\ker f \cap \im (e+e_0) = \{0\}$ on $V(m,n)$.  Indeed, if $[e+e_0, y] \in \ker f$, then 
write $y = \sum_{i,j} y_{i,j}$ in the common eigenbasis for $h$ and $h_0$, where $i, j \in \mathbb Z$. 
If $y_{-m,-n} \neq 0$, then  $[e+e_0, y]$ has nonzero component on the $(-m+2,-n)$-eigenspace since $m > 0$.
This contradicts $[e+e_0, y] \in \ker f$, since $\ker f$ coincides with the $(-m)$-eigenspace of $h$; hence $y_{-m,-n} = 0$.
Repeating this argument for $y_{-m+2,-n}$ and then $y_{-m,-n+2}$ shows that they are both zero.  Continuing inductively
along the diagonals, we get $y_{-m, i} = 0$ for all $i$.  Thus $[e+e_0,y] \in \ker f$ only if $[e+e_0,y] = 0$, as desired.
It follows that $\im (e+e_0) \oplus \ker f$ is a direct sum decomposition of $V(m,n)$ since $\dim \ker f = \dim \ker (f+f_0)$.

On $\cg(\sg)$, which is the direct sum of those $V^{(i)}_{m_i, n_i}$ with $m_i=0$, 
we clearly have $\cg(\sg) = \im (e+e_0) \oplus \cg(\sg)^{f_0}$ since $\sg$ acts trivially.
Therefore, ${\cg(\sg)}^{f_0}\oplus\sum_{i<0}{\mathfrak g}^f(i)$ is  a complementary subspace to $[e+e_0,{\mathfrak g}]$ in ${\mathfrak
g}$, and we are done.
\end{proof}

Let $(\0,\0')$ be either $(2A_2+A_1,A_2+2A_1)$ in type $E_6$ or $(A_4+A_3,A_4+A_2+A_1)$ in type $E_8$.
Let $\0''$ be the $A_2$ orbit in the $E_6$ case and the $A_4$ orbit in the $E_8$ case.
Let $e \in \0''$.   

Our strategy to study the singularity of $\cl$ along $\0'$
is to first describe $\slice$.  In both cases, 
 there exists $x \in \0$ of the form in \eqref{x_form1} 
 such that \eqref{both_conditions} holds with $x_0 \in \cg(\sg)$ regular nilpotent.
 Hence, $\slice$ has a dense $C(\sg)$-orbit, and this allows us to describe $\slice$ in a concrete way. 
 Of independent interest, $\slice$ is the affinization of a cover of the $C(\sg)$-orbit through $x_0$, so unlike
 in \S \ref{entire_nilcones}, the projection to $\cg(\sg)$ of a branch of $\slice$ is not an isomorphism.
The next step is to show for $e_0$ in the unique $C(\sg)$-orbit of codimension four in the nilcone of $\cg(\sg)$
that Lemma \ref{slicelemma} applies.
This allows for the singularity in question to be studied by studying 
  $\cl \cap \SC'_{e+e_0}$, which is manageable since
  $\SC'_{e+e_0} \subset \SC_e = e +\gg^f$, and therefore
  $\cl \cap \SC'_{e+e_0} = \slice \cap \SC'_{e+e_0}$, so it is enough to work 
  completely inside the concrete $\slice$.

Set $Z=C(\sg)$ and $\zg=\cg(\sg)$.
Having found $x \in \0$ of the form $e+x_0+x_1+\ldots +x_m$ as above, our approach then consists of the following series of steps:

\vspace{0.1cm}
\noindent
{\it 1. Describe the (closure of the) set of elements in $Z \cdot x_0$ which are in $e_0+{\mathfrak z}^{f_0}$.} 

\vspace{0.1cm}
\noindent
{\it 2. For each $y_0 \in Z \cdot x_0$ found in step 1, find an element $z\in Z$ such that $z \cdot x_0 = y_0$.}

\vspace{0.1cm}
\noindent
{\it 3. With $z$ as in step 2, determine the values of $z \cdot x_1$, $z\cdot x_2$ etc.}

\vspace{0.1cm}
\noindent
Then since $Z \cdot x$ is dense in $\slice$, we arrive at a parametrization of $\cl \cap \SC'_{e+e_0}$.

\subsection{$\mathbf {(2A_2+A_1,A_2+2A_1)}$ in $\mathbf {E_6}$}
\label{hard_E6}

Recall $\0''$ is of type $A_2$.  We choose $e \in \0''$ and the rest of $\sg$ as follows:
$$e=e_{\alpha_2}+e_{\begin{aligned}12&321 \\[-3mm]&1\end{aligned}},\;\;
f=2f_{\alpha_2}+2f_{\begin{aligned}12&321 \\[-3mm]&1\end{aligned}},\;\; h=[e,f].$$
Then $\zg \cong \mathfrak{sl}_3\oplus\mathfrak{sl}_3$, with basis of simple roots
$\{\alpha_1,\alpha_3,\alpha_5,\alpha_6\}$.
Let ${\mathfrak l}_1$ be the subalgebra of ${\mathfrak z}$ with 
simple roots $\{\alpha_1,\alpha_3\}$ and let ${\mathfrak l}_2$, 
with simple roots $\{\alpha_5,\alpha_6\}$, so that $\zg=\lg_1\oplus\lg_2$.
Similarly, $Z^\circ=L_1\times L_2\cong\SL_3\times\SL_3$, where
$\Lie(L_1)={\mathfrak l}_1$ and $\Lie(L_2)={\mathfrak l}_2$,
and $Z/Z^\circ$ is cyclic of order 2, generated by an
element which interchanges $L_1$ and $L_2$.

The $Z$-orbit structure of ${\mathcal N}(\zg)$ is therefore as follows: there is a unique open orbit, which is also connected. We call this the regular orbit. Its complement in ${\mathcal N}(\zg)$ has two irreducible components permuted transitively by $Z/Z^\circ$, and a unique open $Z$-orbit, which we call the subregular orbit, consisting of pairs $(x,y)$ where one of $x,y$ is regular nilpotent, and the other is subregular, in $\mathfrak{sl}_3$.
The closure of this orbit contains the $Z$-orbit of all pairs $(x,y)$ where both $x$ and $y$ are subregular nilpotent elements of $\mathfrak{sl}_3$.
There are three further $Z$-orbits with representatives
$(x,0)$, as $x$ ranges over all Jordan types in $\sg\lg_3$.

We recall \cite[p. 81]{Lawther-Testerman} that $\gg^f(-2)={\mathbb C}f\oplus V\oplus W$ where $V$ is isomorphic to the tensor product of the natural representation of $L_1$ with the dual of the natural representation of $L_2$, and $W\cong V^*$.
The only other non-trivial space $\gg^f(-i)$ is $\gg^f(-4)$, which is one-dimensional.
Moreover, $v_1 := 3f_{\beta}$ where $\beta = {\begin{aligned} 01 & 210 \\[-3mm] &1 \end{aligned}}$ is a highest weight vector in $V$ and 
$w_1:= 3f_{\alpha_2+\alpha_4}$ is a highest weight vector in $W$, 
relative to the choice of simple roots above.
With respect to the $Z^\circ$-action, we identify $V$ (respectively, $W$) with the space of $3\times 3$ matrices, on which $(g,h)\in L_1\times L_2$ acts via 
$$(g,h)\cdot M=gMh^{-1} \text{ (respectively, } (g,h)\cdot M=hMg^{-1}),$$ 
and we identify $v_1$ and $w_1$ with the matrix with $1$ in the top right entry, and zero everywhere else.


Let
$e_1=e_{\alpha_1+\alpha_3}$, $e_2=e_{\alpha_5+\alpha_6}$,
$\tilde{e}_1=e_{\alpha_1}+e_{\alpha_3}$,
$\tilde{e}_2=e_{\alpha_5}+e_{\alpha_6}$.
Let $x_0 := \tilde{e}_1+\tilde{e}_2$, which is a regular nilpotent element in $\zg$
and let $e_0: = e_1 + e_2$.   Then $e_0$ satisfies the dimension hypothesis \eqref{dim_equality_full_e0} 
and so we can apply Lemma \ref{slicelemma} to it.  On the other hand,
for $x_0$ the situation in \S \ref{extra_eigenvalues}
 applies:

\begin{lemma}\label{anelement}
The element 
$$x:= e+x_0 + v_1+ w_1$$ 
lies in $\slice \cap \0$.  Thus $S_{\0,e} = \ov{Z^\circ \cdot x}.$
\end{lemma}

\begin{proof}
We verified by computer that $x \in \0$.  
The last part follows, as in \S \ref{C_closures},
since both $\slice$ and $Z^\circ \cdot x_0$ have dimension $12$,
and $\slice$ is irreducible since $\cl$ is unibranch at $e$.
\end{proof}

We note that $x_0$ is in the regular nilpotent $Z$-orbit in $\zg$ and $e_1+\tilde{e}_2$ and $\tilde{e}_1+e_2$ both lie in the subregular nilpotent $Z$-orbit so that $\overline{Z\cdot x_0}\supset\overline{Z\cdot (e_1+\tilde{e_2})}\supset\overline{Z\cdot e_0}$.
Moreover, we observe that
 $e + e_1+\tilde{e}_2$ and $e + \tilde{e}_1+e_2$ both belong to $\OC \cap \SC_{\OC,e}$.
This fact can be used to give a conceptual proof of the previous lemma.
It is also useful for the next proposition.

The centralizer $(Z^\circ)^{x_0}$ of $x_0$ in $Z^\circ$ is generated by its identity component, 
a unipotent group of dimension four, and the nine scalar matrices in the center of $Z^{\circ} \cong \SL_3\times\SL_3$.
Let $U$ be the index $3$ subgroup of 
this centralizer containing the central cyclic group $\{ (\omega^i I, \omega^i I) \, | \, i \in \{ 0, 1, 2\} \}$,
where $\omega=e^{2\pi i/3}$.
Let $p: \slice \to \NC(\zg)$ be the restriction of the $Z$-equivariant projection of $e+\gg^f$ onto $\zg$.  
By the previous lemma, $p$ is surjective onto the nilpotent cone $\NC(\zg)$ in $\zg$.
The next proposition is not needed in the proof of the main result, but is of independent interest.

\begin{proposition}\label{dimsoforbits}
The slice $\slice$ is isomorphic to the affinization of the $3$-fold cover $Z^\circ/ U$ of
the regular nilpotent orbit $Z^\circ \cdot x_0$ in $\sl_3(\CM) \oplus \sl_3(\CM)$,
and hence is a normal variety.
Moreover, $p$ is finite and is an isomorphism when restricted to 
the complement of $Z^{\circ} \cdot x$.
Finally, $\slice$ (and hence the affinization) is smooth at points over the subregular  $Z$-orbit in $\NC(\zg)$.
\end{proposition}

\begin{proof}
For dimension reasons the identity component of $U$ acts trivially on $v_1$ and $w_1$.
A pair of scalar matrices $( \omega^{i} I, \omega^{j} I)$ acts on $V$ and $W$ 
by the scalars $\omega^{i-j}$ and  $\omega^{j-i}$, respectively.  
Hence the subgroup of $(Z^\circ)^{x_0}$ that acts trivially on 
$x$ is exactly $U$.
This shows that $\widetilde{Y}:=Z^\circ \cdot x$ identifies with the $3$-fold cover 
$Z^\circ /U$ of the regular orbit $Y:= Z^\circ \cdot x_0$ in $\zg$.

Now the regular functions $\CM[\slice]$ 
on $\slice = \ov{Z^\circ \cdot x}$ embed in 
$\CM[\widetilde{Y}]$, since $\widetilde{Y}$ is dense in $\slice$.
Since $p$ is surjective onto $\ov{Y} = \NC(\zg)$,
we then have the inclusions $\CM[\ov{Y}] \subset \CM[\slice] \subset \CM[\widetilde{Y}]$.
Also $\CM[Y] \cong \CM[ \NC(\zg)]$ since $\NC(\zg)$ is normal.
Now from \cite{Graham}, the ring $\CM[\widetilde{Y}]$
is generated as a module over $\CM[Y]$ by the unique copies of $V$ and $W$
in $\CM[\widetilde{Y}]$.  But $\CM[\slice]$ contains a copy of both $V$ and $W$, via the $Z^\circ$-equivariant projection 
of $\slice$ onto the $V$ and $W$ factors in $\gg^f$, respectively.  Hence
$\CM[\slice] = \CM[\widetilde{Y}]$.
This shows in particular that $\slice$ is normal and $p$ is finite.

For any non-regular element in $\NC(\zg)$, its centralizer in $Z^\circ$ will contain a torus that acts non-trivially on any line in $V$ and $W$.  Thus, since $p$ is finite, $\slice$ must be zero on the $V$ and $W$ components over such elements.  It follows that 
$p$ is an isomorphism over such elements, that is, when restricted to the complement of $Z^{\circ} \cdot x$
in $\slice$.  

Moreover, $\0 \cap \slice$ consists of exactly two $Z$-orbits, 
corresponding to points over the regular and subregular $Z$-orbits in $\zg$.  Since
$\0 \cap \slice$ is smooth, it follows that $\slice$ is smooth at points over the subregular orbit.  
Alternatively, this follows from the fact that transverse slice of $\NC(\zg)$ at a subregular element is $\CM^2/\Gamma'$, where $\Gamma' \subset\SL_2$ is cyclic of order three.  The preimage under $p$ of this transverse slice must then be $\CM^2$.


\end{proof}

Before continuing, we make some observations about transverse slices in $\mathfrak{sl}_3$.
Following up on our identification of $V$ and $W$ with $3 \times 3$ matrices, we 
identify $\lg_1$ and $\lg_2$ with $\sl_3$ so that
$e_1$ and $e_2$ correspond to $\left(\begin{smallmatrix} 0 & 0 & 1 \\ 0 & 0 & 0 \\ 0 & 0 & 0\end{smallmatrix}\right)$.


\begin{lemma} \label{A2_slice}
With the above identification of $\lg_1$ with $\sl_3$, we have that
$$e_1 =\begin{pmatrix} 0 & 0 & 1 \\ 0 & 0 & 0 \\ 0 & 0 & 0
\end{pmatrix},\;h_1=\begin{pmatrix} 1 & 0 & 0 \\ 0 & 0 & 0 \\ 0 &
0 & -1\end{pmatrix},\; f_1=\begin{pmatrix} 0 & 0& 0 \\ 0 & 0 & 0
\\ 1 & 0 & 0 \end{pmatrix}.$$
is an $\sl_2$-triple through $e_1$.
The intersection of
$e_1 +{\mathfrak{sl}_3}^{f_1}$ with
the nilpotent cone in $\sl_3$ is the set of elements of the form
$$X_{st}:=\begin{pmatrix} \frac{1}{2}st & 0 & 1
\\ s^3 & -st & 0 \\ -\frac{3}{4}s^2t^2 & t^3 & \frac{1}{2}st \end{pmatrix}$$
for $s,t\in{\mathbb C}$.
\end{lemma}

\begin{proof}
The ideal of the nilpotent cone in $\sl_3$ is generated by the determinant and the sum of the three diagonal $2\times 2$
minors.   The zero set in $e_1+{\mathfrak{sl}_3}^{f_1}$ of these two functions  
is exactly the elements $X_{st}$ for $s,t\in{\mathbb C}$.
\end{proof}


Continuing the identification of $\lg_1$ and $\lg_2$ with $\sl_3$, we have 
$\tilde{e}_1$ and $\tilde{e}_2$ correspond to
$\left(\begin{smallmatrix} 0 & 1 & 0 \\ 0 & 0 & 1 \\ 0 & 0 & 0\end{smallmatrix}\right)$.

\begin{lemma}\label{conjugating}
If $s\neq 0$ then
$g_{st}\tilde{e}_1 g_{st}^{-1}=X_{st}$ where
$$g_{st}:=\begin{pmatrix} -t & -1/s & 0 \\ -s^2 &
0 & 0 \\ st^2/2 & -t/2 & -1/s \end{pmatrix}$$
Moreover, $g_{st} \in L_1$
and $g_{st}^{-1}=\begin{pmatrix} 0 & -1/s^2 & 0 \\ -s & t/s & 0 \\ s^2t/2 & -t^2 & -s \end{pmatrix}$.
\end{lemma}

\begin{proof}
It is easy to check that $\det g_{st}=1$ (hence lies in $L_1$) 
and that $g_{st}^{-1}$ is as described.  The columns $c_1,c_2,c_3$ of $g_{st}$ satisfy 
$X_{st}c_1=0$, 
$X_{st}c_2=c_{1}$,
and $X_{st}c_3=c_{2}$, from which it follows that
$g_{st} \tilde{e}_1g_{st}^{-1}=X_{st}$.
\end{proof}

As noted above, Lemma \ref{slicelemma} applies to $e_0 = e_1+e_2$.
Furthermore, $e+e_0 \in \0'$.
Thus the affine linear space $\SC'_{e+e_0}=e+e_0+\lg_1^{f_1}+\lg_2^{f_2}+\gg^f(-2)+\gg^f(-4)$ is transverse to $\0'$, and 
hence $\Sing(\0,\0')$ can be determined by describing the intersection 
$\cl \cap \SC'_{e+e_0}$.

\begin{theorem}\label{2A2A1}
 The intersection $\cl \cap \SC'_{e+e_0}$ 
 consists of all elements of the form:
$$e+\left(X_{st},X_{uv},\begin{pmatrix} -\frac{1}{2}tu^2v & tv^2 & tu \\ -\frac{1}{2}s^2u^2v & s^2v^2 & s^2u \\ \frac{1}{4}st^2u^2v & -\frac{1}{2}st^2v^2 & -\frac{1}{2}st^2u \end{pmatrix}, \begin{pmatrix} -\frac{1}{2}s^2tv & t^2v & sv \\ -\frac{1}{2}s^2tu^2 & t^2u^2 & su^2 \\ \frac{1}{4}s^2tuv^2 & -\frac{1}{2}t^2uv^2 & -\frac{1}{2}suv^2 \end{pmatrix}\right)\in e+{\mathfrak l}_1\oplus{\mathfrak l}_2\oplus V\oplus W$$
where $s,t,u,v\in{\mathbb C}$.
\end{theorem}

\begin{proof}
Suppose $s,u\neq 0$.
Consider the action of  the element $(g_{st},g_{uv}) \in Z^\circ$ on $x$.
From Lemmas \ref{A2_slice}
and \ref{conjugating} (also for the $\lg_2$ version),
we have 
\begin{eqnarray}
\nonumber
(g_{st},g_{uv}) \cdot x  & = & e + (g_{st},g_{uv}).(\tilde{e}_1,\tilde{e}_2,v_1,w_1)  \\& = &  
\nonumber
e + \left(g_{st}\tilde{e}_1 g_{st}^{-1}, g_{uv} \tilde{e}_2 g_{uv}^{-1} ,g_{st}
 \left( \begin{smallmatrix} 0 & 0 & 1 \\ 0 & 0 & 0 \\ 0 & 0 & 0 \end{smallmatrix} \right)
g_{uv}^{-1},
g_{uv}  \left( \begin{smallmatrix} 0 & 0 & 1 \\ 0 & 0 & 0 \\ 0 & 0 & 0 \end{smallmatrix}  \right)
g_{st}^{-1} \right) \\
&= & 
e + \left(X_{st},X_{uv},
\begin{pmatrix} -\frac{1}{2}tu^2v & tv^2 & tu \\ -\frac{1}{2}s^2u^2v & s^2v^2 & s^2u \\ \frac{1}{4}st^2u^2v & -\frac{1}{2}st^2v^2 & -\frac{1}{2}st^2u \end{pmatrix}, \begin{pmatrix} -\frac{1}{2}s^2tv & t^2v & sv \\ -\frac{1}{2}s^2tu^2 & t^2u^2 & su^2 \\ \frac{1}{4}s^2tuv^2 & -\frac{1}{2}t^2uv^2 & -\frac{1}{2}suv^2 \end{pmatrix}   
\right). \label{explicit_parametrization}
\end{eqnarray}
By Lemma \ref{anelement}, we have $x \in \0$, so the elements in \eqref{explicit_parametrization}
lie in $\0$.  They also clearly are in $\SC'_{e+e_0}$, 
and so this  set of elements, of dimension four, 
lies in $\cl \cap \SC'_{e+e_0}$.  But the latter is of dimension four since 
this is the codimension of $\0'$ in $\cl$.
Moreover, $\cl$ is unibranch at points in ${\mathcal
O}'$.  Hence $\cl \cap \SC'_{e+e_0}$ is irreducible of dimension four,
and must be the closure of the set of elements in \eqref{explicit_parametrization} 
with $s,u\neq 0$.  The closure of this latter set is evidently those in \eqref{explicit_parametrization} where $s,t,u,v$ are unrestricted.
\end{proof}

Let $\Gamma$ be the subgroup of
$\Sp_4(\CM)$ generated by
$\diag(\omega,\omega^{-1},\omega,\omega^{-1})$.

\begin{corollary}
The singularity $\Sing(\CM^4/\Gamma, 0)$ is equal to $\Sing(2A_2+A_1,A_2+2A_1)$.
\end{corollary}

\begin{proof}
By the theorem, the variety $\cl \cap \SC'_{e+e_0}$
is isomorphic to the variety
with coordinate ring ${\mathbb C}[st,s^3,t^3,uv,u^3,v^3,sv,tu,s^2u,su^2,t^2v,tv^2]$. It is
straightforward to see that this is the invariant subring of
${\mathbb C}[s,t,u,v]$ for the induced action of $\Gamma$.
Also $e+e_0$ corresponds to the point $s=t=u=v=0$.
Since  $\Sing(\cl \cap \SC'_{e+e_0}, e+e_0) = \Sing(2A_2+A_1,A_2+2A_1)$, the result follows.
\end{proof}

We note an interesting consequence of the above description.
The closed subset given by setting $s=v$, $t=u$ has coordinate ring ${\mathbb C}[s^3,t^3,st,st^2,s^2t,s^2,t^2]$, which is exactly the coordinate ring of the singularity $m$. This amounts to taking fixed points in $\cl \cap \SC'_{e+e_0}$ under an appropriate outer involution of ${\mathfrak g}$, giving us another proof that the singularity $(\tilde{A}_2+A_1,A_2+\tilde{A}_1)$ in $F_4$ is 
smoothly equivalent to $m$.

\subsection{$\mathbf{(A_4+A_3, A_4+A_2+A_1)}$ in ${\mathbf E_8}$}
\label{hard_E8}

\subsubsection{}
We begin by describing a concrete model for the singularity.

Let $\Delta=\langle \sigma,\tau:\sigma^5=\tau^2=(\sigma\tau)^2=1\rangle$ be a dihedral group of order 10, acting on $V={\mathbb C}^4$ by: $\tau(u,v)=(v,u)$ and $\sigma(u,v)=(\zeta u,\zeta^{-1}v)$, where $\zeta=e^{\frac{2\pi i}{5}}$ and $(u,v)\in{\mathbb C}^2\oplus{\mathbb C}^2={\mathbb C}^4$.
 
Denote by $p,q$ (resp. $s,t$) the coordinate functions on the first (resp. second) copy of ${\mathbb C}^2$.
In particular, ${\mathbb C}[V]={\mathbb C}[p,q,s,t]$.
It is easy to show that the ring of invariants ${\mathbb C}[V]^\Delta$ is generated by $A=pt+qs$, $B=-2ps$, $C=2qt$ and the functions $F_i=p^{5-i}q^i+s^{5-i}t^i$ for $0\leq i\leq 5$.
We note that $A^2+BC=(pt-qs)^2$.
Since none of the elements of $\Delta$ act as complex reflections on $V$, it follows that the singular points of the quotient $V/\Delta$ are the $\Delta$-orbits of points with non-trivial centralizer, hence are the images in $V/\Delta$ of the points of the form $(u,u)$ (or equivalently, $(u,\zeta^i u)$) for $u\in{\mathbb C}^2$.
Thus the singular locus is properly contained in the zero set of $(A^2+BC)$ in $V/\Delta$.
Let $D=A^2+BC$ and for $0\leq i\leq 5$ let $G_i=(p^{5-i}q^i-s^{5-i}t^i)/(pt-qs)\in\Frac({\mathbb C}[V]^\Delta)={\mathbb C}(V)^\Delta$.
It is easy to see that for $0\leq i\leq 5$, $DG_i\in{\mathbb C}[V]^\Delta$  vanishes on the singular locus of $V/\Delta$, and that $F_i=AG_i+BG_{i+1}$ for $i\leq 4$ (resp. $F_i=CG_{i-1}-AG_i$ for $i\geq 1$), whence the $G_i$ satisfy: $2AG_i-CG_{i-1}+BG_{i+1}=0$ for $1\leq i\leq 4$.
(These equations are also satisfied by the $F_i$'s.)

Let $Y=\Spec({\mathbb C}[A,B,C,G_0,\ldots ,G_5])$.

\begin{remark}
a) The singularity $Y$ can be obtained by blowing up $V/\Delta$ at its singular locus, as follows.
It is not hard to show that the ideal of elements of ${\mathbb C}[V]^\Delta$ which vanish at the singular points is generated by $D$ and $DG_0,\ldots ,DG_5$.
Thus the blowup of $V/\Delta$ can be described as the subset of ${\mathbb A}^9\times{\mathbb P}^6$ which is the closure of the set of elements of the form $(A,B,C,F_0,\ldots ,F_5,[D:DG_0:\ldots :DG_5])$ with at least one of $D,DG_0,\ldots ,DG_5\neq 0$.
Clearly, the affine open subset given by $D\neq 0$ has affine coordinates $A,B,C,F_i,G_i$, and hence is isomorphic to $Y$.
An immediate consequence of this description is that $Y$ is birational to $V/\Delta$.

b) It can be shown that the ideal of relations satisfied by $A,B,C,G_0,\ldots ,G_5$ is generated by the expressions $2AG_i+BG_{i+1}-CG_{i-1}=0$ together with ten identities of the form $G_iG_j-G_{i+1}G_{j-1}-p(A,B,C)=0$, where $p$ is a cubic polynomial.
For example, $G_iG_{i+2}-G_{i+1}^2=\frac{(-1)^i}{8}B^{3-i}C^i$ for $i\leq 3$ and $G_iG_{i+3}-G_{i+1}G_{i+2}=\frac{(-1)^{i+1}}{4}AB^{2-i}C^i$ for $i\leq 2$.


c) It can be shown that all of the remaining affine open subsets of the blow-up given by $DG_i\neq 0$ are smooth, in fact are isomorphic to ${\mathbb A}^4$.
For example, the open subset given by $DG_0\neq 0$ is the affine variety with coordinate ring $R={\mathbb C}[A,B,C,F_0,\ldots, F_5,1/G_0,G_1/G_0,\ldots , G_5/G_0]$.
It is an easy calculation (using the identities for the $G_i$ mentioned above) to check that this ring is generated by $B,F_0,{1}/{G_0}$ and ${G_1}/{G_0}$, hence by dimensions is a polynomial ring of rank four.
Thus the point of $Y$ corresponding to the maximal ideal $(A,B,C,G_i)$ is the unique singular point of the blow-up of ${\mathbb C}^4/\Delta$.
This justifies the more succinct description of $\Sing(\0,\0')$ given in the introduction.

d) In general, a blow-up of a symplectic singularity is not a symplectic singularity.
In our case, $\overline\0$ inherits a symplectic structure from that of $\gg$, and so (subject to our claim) $Y$ is a symplectic singularity.
More generally, it can be shown that the blow-up (at the singular locus) of the quotient of ${\mathbb C}^4$ by any dihedral group (with ${\mathbb C}^4$ identified with two copies of its reflection representation) is a symplectic singularity.
\end{remark}

We will next show that $\Sing(\overline\0,\0')$ is equivalent to $Y$.

\subsubsection{}
Let $e=e_{\alpha_1}+e_{\alpha_3}+e_{\alpha_4}+e_{\alpha_2}$, $f=4f_{\alpha_1}+6f_{\alpha_3}+6f_{\alpha_4}+4f_{\alpha_2}$, $h=[e,f]$. 
Then $e \in \0''$, the orbit of type $A_4$,
and  $\zg \cong\mathfrak{sl}_5(\CM)$ with basis of simple roots $\{\beta_1,\beta_2,\beta_3,\beta_4\}:=\{\alpha_8,\alpha_7,\alpha_6,
\begin{aligned} 24 &65321 \\[-3mm] &3\end{aligned}\}$, 
and $Z$ is isomorphic to the semidirect product of $\SL_5({\mathbb C})$ by an outer involution.
Let $x_0$ belong to the regular nilpotent orbit in $\zg$.

For the purposes of calculation we identify $Z^\circ$ with $\SL_5(\CM)$ by identifying the basis of simple roots 
$\{ \beta_1,\beta_2,\beta_3,\beta_4\}$ with the basis of simple roots of $\SL_5(\CM)$ coming from the choice
of diagonal maximal torus and upper triangular Borel subgroup, with the usual ordering of simple roots.
Let $W$ be the natural module for $Z^\circ$, corresponding to the defining representation of $\SL_5$.
The $Z^\circ$-module structure of ${\mathfrak g}^f$ includes the following spaces:
$${\mathfrak g}^f(-2)\cong W\oplus W^*\oplus{\mathbb C},\;\; {\mathfrak g}^f(-4)\cong \Lambda^2(W)\oplus\Lambda^2(W^*)\oplus{\mathbb C}$$
The following vectors are highest weight vectors, relative to the simple roots $\{\beta_i\}$:
$$w_1=3e_{\begin{aligned} 00 &11111 \\[-3mm] &1\end{aligned}}-2e_{\begin{aligned} 01 &11111 \\[-3mm] &0\end{aligned}}\in W,\;\; u_1=2e_{\begin{aligned} 13 &54321 \\[-3mm] &2\end{aligned}}-3e_{\begin{aligned} 23 &44321 \\[-3mm] &2\end{aligned}}\in W^*,$$ $$y_1=e_{\begin{aligned} 12 &33321 \\[-3mm] &1 \end{aligned}}\in\Lambda^2(W),\;\; z_1=e_{\begin{aligned} 01 &22221 \\[-3mm] &1 \end{aligned}}\in\Lambda^2(W^*).$$
Let $x_0 := e_{\beta_1}+e_{\beta_2}+e_{\beta_3}+e_{\beta_4}$, a regular nilpotent element in $\zg$.
Then we verified that 
$$x=e+ x_0 -w_1+u_1+10y_1-10z_1\in\0,$$
where recall $\0$ is of type $A_4+A_3$,
and so it follows that $\slice = \ov{Z^\circ \cdot x}$ since both sides are dimension $20$ and $\cl$ is unibranch at $e$.
This leads to a description of $\slice$, whose details, which we omit, are similar to those in Proposition \ref{dimsoforbits}.
Recall that $p: \slice \to \NC(\zg)$ is the $Z$-equivariant projection.

\begin{proposition}\label{A4_cover}
The slice $\slice$ is isomorphic to the affinization of the universal cover of the regular nilpotent orbit in $\sl_5(\CM)$,
and hence is a normal variety.
Moreover, $p$ is finite and is an isomorphism when restricted to 
the complement of $Z^{\circ} \cdot x$.
Finally, $\slice$ (and hence the affinization) is smooth at points over the subregular $Z^{\circ}$-orbit in $\NC(\zg)$.
\end{proposition}


We also note that $\0 \cap \slice$ is the union of two $Z^\circ$-orbits, one of which projects under $p$ to the regular orbit
and the other, to the subregular orbit in $\NC(\zg)$.

Let $e_0 \in \zg$ be an element in the orbit with partition type $[3,2]$,
which is codimension four in $\NC(\zg)$.   Then $e+e_0 \in \0'$.
Moreover, $e_0$ satisfies the condition in Lemma \ref{slicelemma}
and so we can study the singularity $(\0, \0')$ by
studying  $\cl  \cap \SC'_{e+e_0}$

\begin{lemma}
The intersection $\cl \cap \SC'_{e+e_0}$ is isomorphic to the closure of the set of all $(M,w'_1,w'_1\wedge w'_2,w'_1\wedge w'_2\wedge w'_3,w'_1\wedge w'_2\wedge w'_3\wedge w'_4)\in (e_0+{\mathfrak z}^{f_0})\times (W\oplus\Lambda^2(W)\oplus\Lambda^3(W)\oplus\Lambda^4(W))$ such that there is a basis $\{ w'_1,w'_2,\ldots ,w'_5\}$ for $W$ with $w'_1\wedge \ldots \wedge w'_5=1$ and $Mw'_i=w'_{i-1}$ ($i\geq 2$), $Mw'_1=0$.
\end{lemma}

\begin{proof}
We can describe $\slice$ in the following way:
let $\{ w_1,w_2,w_3,w_4,w_5\}$ be the standard basis for ${\mathbb C}^5$ and
let $$M_0=\begin{pmatrix} 0 & 1 & 0 & 0 & 0 \\ 0 & 0 & 1 & 0 & 0 \\ 0 & 0 & 0 & 1 & 0 \\  0 & 0 & 0 & 0 & 1 \\ 0 & 0 & 0 & 0 & 0 \end{pmatrix}$$
so that $M_0w_1=0$ and $M_0w_i=w_{i-1}$ for $2\leq i\leq 5$.
Then $\slice$ is isomorphic to the closure in $\mathfrak{sl}_5\oplus W\oplus\Lambda^2(W)\oplus\Lambda^3(W)\oplus\Lambda^4(W)$ of the $\SL_5$-orbit of $\tilde{M}_0:=(M_0,w_1,w_1\wedge w_2,w_1\wedge w_2\wedge w_3,w_1\wedge w_2\wedge w_3\wedge w_4)$.

To describe the subvariety $\cl \cap \SC'_{e+e_0} = \slice \cap \SC'_{e+e_0}$, we note that
if $M\in e_0+\zg^{f_0}$ is nilpotent then generically $M$ is regular and therefore there exists a basis ${\mathcal B}=\{ w'_1,w'_2,w'_3,w'_4,w'_5\}$ of ${\mathbb C}^5$ such that $Mw'_1=0$ and $Mw'_i=w'_{i-1}$ for $i\geq 2$.
After scaling, we may assume that $g_{\mathcal B} :=(w'_1\; w'_2\; w'_3\; w'_4\; w'_5)$ has determinant one.
Then the tuple $(M,w'_1,w'_1\wedge w'_2,w'_1\wedge w'_2\wedge w'_3,w'_1\wedge w'_2\wedge w'_3\wedge w'_4)=g_{\mathcal B}\cdot\tilde{M}_0$.
\end{proof}

Next, we concretely describe the variety $\NC(\zg) \cap (e_0+\zg^{f_0})$. 
Let $$e_0=\begin{pmatrix} 0 & 0 & 0 & 0 & 0 \\ 0 & 0 & 0 & 0 & 0 \\ 2 & 0 & 0 & 0 & 0 \\ 0 & 1 & 0 & 0 & 0 \\ 0 & 0 & 2 & 0 & 0 \end{pmatrix},\;\; f_0=\begin{pmatrix} 0 & 0 & -1 & 0 & 0 \\ 0 & 0 & 0 & -1 & 0 \\ 0 & 0 & 0 & 0 & -1 \\ 0 & 0 & 0 & 0 & 0 \\ 0 & 0 & 0 & 0 & 0\end{pmatrix}$$
and then $e_0+{\mathfrak z}^{f_0}$ consists of all matrices of the form $$\begin{pmatrix} 2a & b & c & d & g \\ 0 & -3a & h & k & l \\ 2 & 0 & 2a & b & c \\ 0 & 1 & 0 & -3a & h \\ 0 & 0 & 2 & 0 & 2a\end{pmatrix},$$ where $a,b,c,d,g,h,k,l\in{\mathbb C}$.
For the purposes of our calculation, we consider the matrices in $e_0+{\mathfrak z}^{f_0}$ of the form:
$$M=\begin{pmatrix} 2a & b & c-6a^2 & d-2ab & 40a^3-10ac-\frac{5}{4}bh \\ 0 & -3a & h & 9a^2-4c & l-2ah \\ 2 & 0 & 2a & b & c-6a^2 \\ 0 & 1 & 0 & -3a & h \\ 0 & 0 & 2 & 0 & 2a \end{pmatrix}$$
A calculation confirms that any such matrix satisfies $\Tr M^2=\Tr M^3=0$, and that the conditions $\Tr M^4=0$ and $\Tr M^5=0$ are expressed in terms of the coordinates $a,b,c,d,h,l$ as:
\begin{equation}\label{nilpeqns} dh+bl+\frac{8}{3}c^2 = a(9bh-216a^3+72ac),\;\; dl=c(9bh-216a^3+48ac).\end{equation}
Since every irreducible component of the set of $(a,b,c,d,h,l)$ satisfying these two equations has dimension at least four, it follows that the set of matrices given by the coordinates satisfying (\ref{nilpeqns}) is equal to the set of nilpotent elements of $e_0+{\mathfrak z}^{f_0}$ (and is therefore irreducible).

It is easy to verify that the rational functions $a={A}/{6}$, $b=-{G_0}/{3}$, $c=-{BC}/{16}$, $d={BG_1}/{4}$, $h={G_5}/{3}$, $l={CG_4}/{4}$ in ${\mathbb C}(p,q,s,t)$ satisfy (\ref{nilpeqns}).
Since $A,BC,G_0,G_5,BG_1,CG_4$ are regular functions on $Y$, we have therefore constructed a morphism from $Y$ to ${\mathcal N}({\mathfrak z})\cap(e_0+{\mathfrak z}^{f_0})$, corresponding to the inclusion ${\mathbb C}[A,BC,G_0,G_5,BG_1,CG_4]\subset{\mathbb C}[Y]$.
In fact, this morphism corresponds to quotienting $Y$ by the action of a group of order five, as follows: let $\rho$ be the automorphism of order five of $V$ which sends $(p,q,s,t)$ to $(\zeta p,\zeta^{-1}q,\zeta s,\zeta^{-1}t)$.
Then $\rho$ normalizes $\Gamma$ and has an induced action on $Y$ satisfying ${\mathbb C}[Y]^\rho = {\mathbb C}[A,BC,G_0,G_5,BG_1,CG_4]$.
(The invariants $B^2G_2$ and $C^2G_3$ are contained in this ring, since $BG_2=CG_0-2AG_1$ and $CG_3=2AG_4+BG_5$.)
It follows that the coordinates $a=A/6$ etc. given above define an isomorphism from $Y/\langle\rho\rangle$ to ${\mathcal N}(\mathfrak{sl}_5)\cap (e_0+{\mathfrak z}^{f_0})$.

\begin{remark}
The above discussion indicates an interesting way to view the singularity $([5],[3,2])$ in $\mathfrak{sl}_5$, as an affine open subset of the blow up of the quotient of ${\mathbb C}^4$ by a group of order 50.
Indeed, the group generated by $\Gamma$ and $\rho$ is isomorphic to the complex reflection group $G(5,1,2)$, acting on ${\mathbb C}^4=U\oplus U^*$ where $U$ is the defining representation for $G(5,1,2)$.
Blowing up the quotient at the set of orbits of points of the form $(u,u)$, and restricting to the affine open subset given by $D\neq 0$, one obtains the variety $Y/\langle\rho\rangle$.
\end{remark}

We will first give an ad hoc justification that $\SC'_{\0,e+e_0} := \cl \cap \SC'_{e+e_0}$ is isomorphic to $Y$, 
and then a more rigorous proof.
Fix a matrix $M$ as above with coordinates $a=A/6$, etc., which we think of as depending on the point $(A,B,C,G_0,\ldots ,G_5)\in Y$.
The space of (column) vectors in $W$ which are annihilated by $M$ is generically of dimension one, spanned by $$w'_1=\begin{pmatrix} -\frac{1}{6}G_0G_4+\frac{1}{9}A^2B+\frac{1}{32}B^2C \\ -\frac{1}{4}CG_3-\frac{1}{12}BG_5 \\ -\frac{1}{6}AB \\ -G_4 \\ B \end{pmatrix}$$
Similarly, the space of (row) vectors in $W^*$ which are annihilated by $M$ is also generically of dimension one, spanned by $u'_1=(C,\; -G_1,\; -\frac{1}{6}AC,\; \frac{1}{3}CG_0-\frac{1}{2}AG_1,\; \frac{1}{9}A^2C+\frac{1}{32}BC^2+\frac{1}{6}G_1G_5)$.
It follows that if $z\in Z^\circ=\SL_5$ is such that $\Ad z(M_0)=M$, then $zw_1$ is a scalar multiple of $w'_1$, and $u_1z$ is a scalar multiple of $u'_1$.
Our more rigorous argument below will (essentially) consist of showing that these scalars, up to multiplication by a fifth root of unity, are independent of $p,q,s,t$.
Thus the ring of regular functions on $\SC'_{\0,e+e_0}$ also contains elements which naturally correspond to $B$, $C$, $G_1$ and $G_4$.
To continue along this line, we would have to find a vector $w'_2\in W$ such that $Mw'_2=w'_1$, and similarly for $u'_1$.
Then it turns out that the coordinates of $w'_1\wedge w'_2$ and $u'_1\wedge u'_2$ are contained in ${\mathbb C
}[Y]$, and include scalar multiples of $G_2$ and $G_3$.
Thus one obtains a morphism $\varphi:Y\rightarrow \SC'_{\0,e+e_0}$, which (since all of the generators $A,B,C,G_i$ appear somewhere in the coordinates describing $\varphi$) is evidently a closed immersion, hence an isomorphism by equality of dimensions and reducedness.

For a more careful analysis, we note that finding a basis $\{ w'_1,\ldots ,w'_5\}$ for ${\mathbb C}^5$ such that $Mw'_i=w'_{i-1}$ for $i\geq 2$ and $Mw'_1=0$ is essentially equivalent to finding an element $w'_5\in{\mathbb C}^5$ such that $M^4w'_5\neq 0$.
Moreover, any transformation of the form $w'_5\mapsto w'_5+\alpha w'_4+\beta w'_3+\gamma w'_2+\delta w'_1$ preserves the elements $w'_1$, $w'_1\wedge w'_2$, $w'_1\wedge w'_2\wedge w'_3$, $w'_1\wedge w'_2\wedge w'_3\wedge w'_4$.
Thus, to find $z\cdot\tilde{M}_0$ where $\Ad z(M_0)=M$, it suffices to choose an element $w'_5$ such that $M^4w'_5\neq 0$, and then to multiply $w'_5$ by an appropriate scalar such that $\det (w'_1\; w'_2\; w'_3\; w'_4\; w'_5)=1$.
For this purpose, we first choose $$w'_5=\begin{pmatrix} 1 \\ 0 \\ 0 \\ 0 \\ 0 \end{pmatrix}\,\;\mbox{; then}\; w'_4=\begin{pmatrix} \frac{1}{3}A \\ 0 \\ 2 \\ 0 \\ 0 \end{pmatrix},\; w'_3=\begin{pmatrix} -\frac{2}{9}A^2-\frac{1}{8}BC \\ \frac{2}{3}G_5 \\ \frac{4}{3}A \\ 0 \\ 4 \end{pmatrix},\; w'_2=\begin{pmatrix} \frac{4}{9}A^3+\frac{7}{24}ABC+\frac{1}{3}G_0G_5 \\ CG_4-\frac{1}{3}AG_5 \\ -\frac{2}{3}A^2-\frac{1}{2}BC \\ 2G_5 \\ 4A\end{pmatrix}$$ and finally $$w'_1=\begin{pmatrix} A^4+\frac{23}{36}A^2BC+\frac{1}{32}B^2C^2+AG_0G_5+\frac{1}{2}BG_1G_5-\frac{1}{3}CG_0G_4 \\ \frac{1}{2}ACG_4+\frac{1}{3}BCG_5 \\ \frac{1}{6}ABC \\ CG_4 \\ -BC\end{pmatrix}=-C\begin{pmatrix} -\frac{1}{6}G_0G_4+\frac{1}{9}A^2B+\frac{1}{32}B^2C \\ -\frac{1}{4}CG_3-\frac{1}{12}BG_5 \\ -\frac{1}{6}AB \\ -G_4 \\ B \end{pmatrix}$$
Then one can show that the determinant of the matrix $(w'_1\; w'_2\; w'_3\; w'_4\; w'_5)$ is $-C^5$.
Thus we replace each of $w'_i$, $1\leq i\leq 5$ by $w''_i=-w'_i/C$, which is well-defined whenever $C\neq 0$.
In other words, whenever $C\neq 0$ we can construct a matrix $g_{\mathcal B}=(w_1''\; w_2''\; w_3''\; w_4''\; w_5'')$ of determinant 1 such that $g_{\mathcal B}M_0g_{\mathcal B}^{-1}=M$.

It is now a routine computer calculation to verify that, relative to the obvious basis for $\Lambda^2(W)$, we have $$w''_1\wedge w''_2 = \begin{pmatrix} \frac{1}{18}G_0G_4^2-\frac{5}{288}C^3G_0+\frac{11}{288}AC^2G_1-
\frac{5}{144}A^2CG_2+\frac{1}{18}A^3G_3 \\ \frac{1}{36}A^2B^2+\frac{1}{64}B^3C-\frac{1}{18}A G_0G_3-\frac{1}{12}B G_0G_4 \\ -\frac{2}{9}A^2 G_3-\frac{7}{24}AB G_4-\frac{1}{16}B^2 G_5 \\ \frac{1}{6} AB^2+\frac{1}{3} G_0G_3 \\ -\frac{1}{12}C(AG_2+2BG_3) \\ \frac{2}{3}G_3G_5-G_4^2 \\ AG_3+BG_4 \\ \frac{1}{3}AG_3+\frac{1}{2}BG_4 \\ -\frac{1}{2}B^2 \\ 2G_3 \end{pmatrix}$$
and similarly $$w''_1\wedge w''_2\wedge w''_3 = \begin{pmatrix} -\frac{1}{18}A^3G_2+\frac{1}{36}A^2CG_1+\frac{5}{288}ABCG_2-\frac{1}{96}BC^2G_1-\frac{1}{18}G_0G_3G_4 \\ -\frac{1}{9}AG_3G_4+\frac{1}{24}BG_3G_5-\frac{1}{8}BG_4^2
 \\ \frac{1}{3}A^2G_2+\frac{1}{12}ABG_3+\frac{1}{4}BCG_2-\frac{1}{12}C^2G_0 \\ -\frac{1}{24}ABG_3+\frac{1}{16}B^2G_4-\frac{1}{9}A^2G_2 \\
-\frac{1}{8}B^3+\frac{1}{3}G_0G_2 \\ \frac{2}{3}AG_2+\frac{1}{2}BG_3 \\ -\frac{1}{3}G_3G_4-\frac{1}{12}AC^2 \\ AG_2+BG_3 \\ \frac{1}{2}C^2 \\ -2G_2 \end{pmatrix}.$$
Finally, it is straightforward to show using the identification of $\Lambda^4(W)$ with $W^*$ that $$w''_1\wedge w''_2\wedge w''_3\wedge w''_4 = \begin{pmatrix} -C & G_1 & \frac{1}{6}AC & \frac{1}{3}CG_0-\frac{1}{2}AG_1 & \frac{1}{9}A^2C+\frac{1}{32}BC^2+\frac{1}{6}G_1G_5\end{pmatrix}.$$

What these computations amount to is the following:

\begin{theorem}
There is a morphism from the open subset of $Y$ given by $C\neq 0$ to $\0 \cap \SC'_{e+e_0}$, 
given by letting the matrix $g_{\mathcal B}$ act on $\tilde{M}_0$.
Moreover, this morphism extends to an isomorphism from $Y$ 
to $\SC'_{\0,e+e_0}$.
\end{theorem}

\begin{proof}The first part follows from the above discussion, since $g_{\mathcal B}$ has coordinates in the localized ring ${\mathbb C}[Y]_C$.
But we can see by our calculations that in fact, the coordinates of $M$ and $w''_1,\ldots ,w''_1\wedge w''_2\wedge w''_3\wedge w''_4$ all lie in ${\mathbb C}[Y]$.
Thus we can extend the morphism to a morphism $\varphi$ from $Y$ to $\SC'_{\0,e+e_0}$.
Since each of the generators $A,B,C,G_0,\ldots ,G_5$ appears (up to multiplication by a scalar) as a coordinate of the map $\varphi$, it follows that $\varphi$ is a closed immersion, and hence by dimensions, irreducibility and reducedness, is an isomorphism onto $\SC'_{\0,e+e_0}$.
\end{proof}


\section{Graphs}  \label{sec:graphs}

Capital letters are used to denote simple singularities and lower-case letters to denote singularities of closures of minimal nilpotent orbits.
The notation $m$, $m'$, $\mu$, $\chi$, $a_2/\mathfrak{S}_2$ and $\tau$ are explained in \S \ref{unexpected_singularities}.
The intrinsic symmetry action induced from $A(e)$ is explained in \S \ref{section:split_and_intrinsic} and the notation is explained in \S \ref{full_intrinsic_action}.
We use $(Y)$ to denote a singularity with normalization $Y$.


  \begin{center}
\begin{tikzpicture}
  \matrix (G2all)
    [matrix of math nodes,%
     nodes in empty cells,
     nodes={outer sep=0pt,minimum size=0pt},
     column sep={1.7cm,between origins},
     row sep={.8cm,between origins}]
  {
& | (G2) | G_2  \\
& | (G2a1) | G_2(a_1) \\
&  | (tA1) |  \tilde{A}_1 \\
& | (A1) | A_1 \\
&  | (0) | 0 \\
  };

\begin{scope}[]
\draw (G2) -- (G2a1) node[midway, right] {$G_2$};
\draw (G2a1) -- (tA1) node[midway, right] {$A_1$};
\draw (tA1) -- (A1) node[midway, right]{$m$};   
\draw (A1) -- (0) node[midway,right] {$g_2$};
\end{scope}
\end{tikzpicture}
\end{center}

  \begin{center}
\begin{tikzpicture}

  \matrix (F4all)
    [matrix of math nodes,%
     nodes in empty cells,
     nodes={outer sep=0pt,minimum size=0pt},
     column sep={1.7cm,between origins},
     row sep={.7cm,between origins}]
  {
&&& | (F4) | F_4  \\
&&& | (F4a1) | F_4(a_1) \\
&&& | (F4a2) | F_4(a_2) \\
&& | (B3) | B_3 && | (C3) | C_3  \\
&&& | (F4a3) | F_4(a_3) \\
&&& | (C3a1) | C_3(a_1) \\
&&  | (tA2+A1) | \tilde{A}_2 \! +\!A_1  && | (B2) | B_2  \\
&&&&  | (A2+tA1) | A_2  \! +\! \tilde{A}_1 \\
&&  | (tA2) | \tilde{A}_2  && | (A2) | A_2  \\
&&&  | (A1+tA1) | A_1 \! +\! \tilde{A}_1 \\
&&&  | (tA1) |  \tilde{A_1} \\
&&& | (A1) | A_1 \\
&&&  | (0) | 0 \\
  };

 \begin{scope}[]

\draw (F4) -- (F4a1) node[midway, right] {$F_4$};
\draw (F4a1) -- (F4a2) node[midway, right] {$C_3$};
\draw (F4a2) -- (B3) node[midway, above]{$A_1$};   
\draw (F4a2) -- (C3) node[midway, above] {$A_1$}; \draw (C3) --
(F4a3) node[midway,  above=2.5pt, left=.3pt] {$4 G_2$}; \draw (B3)
-- (F4a3) node[midway, above=2pt, right] {$G_2$}; \draw (F4a3) --
(C3a1) node[midway, right] {$A_1$}; \draw (C3a1) -- (tA2+A1)
node[midway, above left] {$m$}; \draw (C3a1) -- (B2) node[midway,
above right] {$[2A_1]^+$};
\draw (B2) -- (A2+tA1) node[midway, right] {$A_1$};  

\draw (A2+tA1) -- (A2) node[midway, right] {$a^{+}_2$};

\draw (tA2+A1) -- (A2+tA1) node[midway, above=3pt, right] {$m$};
\draw (tA2+A1) -- (tA2) node[midway, left] {$g_2$};

\draw (tA2) -- (A1+tA1) node[midway, below] {$A_1$};    
\draw (A2) -- (A1+tA1) node[midway, below] {$A_1$};
\draw (A1+tA1) -- (tA1)  node[midway, right] {$a^{+}_3$};
\draw (tA1) -- (A1)  node[midway, right] {$c_3$};
\draw (A1) -- (0) node[midway,right] {$f_4$};

\end{scope}

\end{tikzpicture}
\end{center}

\begin{tikzpicture}
  \matrix (E6all)
    [matrix of math nodes,%
     nodes in empty cells,
     nodes={outer sep=0pt,minimum size=0pt},
     column sep={1.5cm,between origins},
     row sep={.68cm,between origins}]
  {
&&& | (E6) | E_6  \\
&&& | (E6a1) | E_6(a_1) \\
&&& | (D5) | D_5 \\
&&& | (E6a3) | E_6(a_3)  \\
&&  |(A5) | A_5 \\
&&& | (D5a1) | D_5(a_1) \\
&&  | (A4+A1) | A_4 \! +\!A_1 \\
&&&& | (D4) | D_4 \\
&& | (A4) | A_4 \\
&&& | (D4a1) | D_4(a_1)  \\
&&& | (A3+A1) | A_3+A_1 \\
&& | (A3) | A_3  && | (2A2+A1) | 2A_2+A_1 \\ \\
&& | (A2+2A1) | A_2+2A_1 && | (2A2) | 2A_2 \\
&&&  | (A2+A1) | A_2+A_1 \\
&&&   | (A2) | A_2 \\
&&&  | (3A1) | 3A_1 \\
&&& | (2A1) | 2A_1 \\
&&& | (A1) | A_1 \\
&&&| (0) | 0 \\
  };
  \begin{scope}[]

\draw (E6) -- (E6a1) node[midway, right] {$E_6$};
\draw (E6a1)--(D5) node[midway, right] {$A_5$};    
\draw (D5) --   (E6a3) node[midway,right] {$C_3$};   
\draw (E6a3) -- (A5) node[midway,above] {$A_1$};
\draw (E6a3)  --  (D5a1) node[midway,right] {$A_2$};   

\draw (A5) -- (A4+A1) node[midway,left] {$A_2$};
\draw (D5a1) -- (A4+A1) node[midway, above] {$A_2$};
\draw (D5a1) -- (D4) node[midway, right] {$a_2$};   

\draw (A4+A1) -- (A4) node[midway, left] {$A_1$}; \draw (D4) --
(D4a1) node[midway, above=2pt,left=1pt] {$G_2$}; \draw (A4) -- (D4a1)
node[near start, above=2pt, right=1pt] {$3C_2$}; \draw (D4a1)  -- (A3+A1)
node[midway, left] {$A_1$};

\draw (A3+A1)  -- (A3) node[midway,above left] {$b_2$};   
\draw (A3+A1)  -- (2A2+A1) node[midway,above right] {$m$};

\draw (A3) -- (A2+2A1) node[midway,left] {$A_1$};    
\draw (2A2+A1) -- (2A2) node[midway, right] {$g_2$};
\draw (2A2+A1) -- (A2+2A1) node[midway,above left ] {$\tau$};

\draw (2A2) -- (A2+A1) node[midway, below right] {$A_2$};     
\draw (A2+2A1) -- (A2+A1) node[midway,below left] {$a_2$};   
\draw (A2+A1) -- (A2) node[midway,right] {$[2a_2]^+$}; \draw (A2)
-- (3A1) node[midway, right] {$A_1$}; \draw (3A1) -- (2A1)
node[midway, right] {$b_3$}; \draw (2A1) -- (A1)
node[midway,right] {$a_5$}; \draw (A1) -- (0) node[midway,right]
{$e_6$};

\end{scope}
\end{tikzpicture}

\begin{tikzpicture}
  \matrix (E7all)
    [matrix of math nodes,%
     nodes in empty cells,
     nodes={outer sep=0pt,minimum size=0pt},
     column sep={1.5cm,between origins},
     row sep={.68cm,between origins}]
  {
&&&| (E7) | E_7 \\
&&& | (E7a1) | E_7(a_1)  \\
&&& | (E7a2) | E_7(a_2) \\
&& | (E6) | E_6 &&  | (E7a3) | E_7(a_3) \\
&&&| (E6a1) | E_6(a_1) & | (D6) | D_6 \\
&&& | (E7a4) | E_7(a_4)  \\
& | (A6) | A_6  &&  | (D5+A1) | D_5+A_1  && | (D6a1) | D_6(a_1)\\
& | (E7a5) | E_7(a_5) &&&& | (D5) | D_5 \\
& | (D6a2) | D_6(a_2) &&&& | (E6a3) | E_6(a_3)  \\
& | (A5+A1) | A_5 \! + \! A_1 && | (A5') | A_5' &&  | (D5a1+A1) | D_5(a_1)\! +\!A_1  \\
\\
& | (A5'') | A^{''}_5 && | (D5a1) | D_5(a_1) && | (A4+A2) | A_4 \! +\!A_2 \\
&&& | (D4+A1) | D_4 \! +\!A_1  &&  | (A4+A1) | A_4 \! +\!A_1 \\
& | (D4) | D_4 && | (A4) | A_4 && | (A3+A2+A1) | A_3 \! +\!A_2 \! +\!A_1 \\
&&&&& | (A3+A2) | A_3 \! +\!A_2 \\
&&&&& | (D4a1+A1) | D_4(a_1) \! +\!A_1 \\
& | (D4a1) | D_4(a_1) &&&&  | (A3+2A1) | A_3 \! +\!2A_1 \\
&&& | (A3+A1') | (A_3+A_1)'  && | (A3+A1'') | (A_3+A_1)''  \\
&& | (2A2+A1) | 2A_2+A_1 \\
& | (A2+3A1) | A_2+3A_1 && | (2A2) | 2A_2  && | (A3) | A_3\\
&&& | (A2+2A1) | A_2+2A_1 \\
&&& | (A2+A1) | A_2+A_1 \\
&& | (4A1) | 4A_1 \\
& | (3A1'') | (3A_1)'' &&&& | (A2) | A_2 \\
&&&& | (3A1') | (3A_1)' \\
&&& | (2A1) | 2A_1 \\
&&& | (A1) | A_1 \\
&&&| (0) | 0 \\
  };
  \begin{scope}[]

\draw (E7) -- (E7a1) node[midway, right] {$E_7$};
\draw (E7a1) -- (E7a2) node[midway,right] {$D_6$};
\draw (E7a2) -- (E6) node[midway,above] {$A_1$};
\draw (E7a2) -- (E7a3) node[midway, above] {$C_4$};
\draw (E6) -- (E6a1) node[midway, below] {$F_4$};  
\draw (E7a3) -- (D6) node[midway,right] {$A_1$}; \draw (E7a3) --(E6a1) node[midway,above] {$B_3$}; 
\draw (D6)--(E7a4) node[midway,above] {$C_3$}; \draw (E6a1)--(E7a4)
node[midway,right] {$C_3$}; \draw (E7a4) --   (A6)
node[midway,above] {$A_1$}; \draw (E7a4) --   (D5+A1)
node[midway,right] {$A_1$}; \draw (E7a4) --   (D6a1)
node[midway,above right]{$A_1$}; \draw (D6a1) --   (D5)
node[midway, right] {$A_1$}; \draw (D6a1) --   (E7a5) node[midway,below]
{$3 C_3$}; \draw (A6) --   (E7a5) node[midway,left] {$G_2$};
\draw (D5+A1) --  (D5) node[midway,below right] {$A_1$}; \draw
(D5+A1) -- (E7a5)  node[midway,above] {$G_2$};

\draw (D5) --   (E6a3) node[midway,right] {$C_3$};
\draw (E7a5) --  (E6a3) node[midway,below] {$A_1$};
\draw (E7a5) --  (D6a2) node[midway,left] {$A_1$};
\draw (D6a2) --  (A5+A1) node[midway, left] {$m$};
\draw (D6a2) --  (A5') node[midway, below] {$A_1$};
\draw (D6a2) --  (D5a1+A1) node[near start, above=-1pt] {$A_1$};

\draw (E6a3) -- (A5') node[midway,above=-1pt] {$A_1$};
\draw (E6a3)  --  (D5a1+A1) node[midway,right] {$A_1$};

\draw (A5+A1) -- (A5'') node[midway,left] {$g_2$};
\draw (A5+A1) -- (A4+A2)  node[near start,below] {$A_1$};

\draw (A5') -- (A4+A2) node[near start,above=4pt, right=0pt] {$A_1$};
\draw (A5'') -- (A4) node[midway,above] {$B_3$};

\draw (D5a1+A1) -- (A4+A2) node[midway,right] {$A_1$}; \draw
(D5a1+A1) -- (D5a1) node[near start,below=2pt,right=2pt] {$A_1$}; \draw (D5a1) --
(D4+A1) node[midway, left] {$b_2$}; \draw (D5a1) -- (A4+A1)
node[midway,above] {$A^{+}_2$}; \draw (A4+A2) -- (A4+A1)
node[midway,right] {$A^{+}_2$};

\draw (D4+A1) -- (D4) node[midway, below] {$c_3$};
\draw (D4+A1) -- (A3+A2+A1) node[near start,above] {$A_1$};
\draw (A4+A1) -- (A3+A2+A1) node[midway,right] {$a_2/\mathfrak{S}_2$};
\draw (A4+A1) -- (A4) node[midway,below=7pt,left=-3pt] {$a^{+}_2$};

\draw (A3+A2+A1)  -- (A3+A2) node[midway, right] {$A_1$}; \draw
(A4) -- (A3+A2) node[midway,below] {$C_2$}; \draw (A3+A2)  --
(D4a1+A1) node[midway,right] {$[2A_1]^+$};

\draw (D4) -- (D4a1) node[midway,left] {$G_2$};  

\draw (D4a1+A1) -- (D4a1) node[midway,above] {$[3A_1]^{++}$};
\draw (D4a1+A1) -- (A3+2A1) node[midway, right] {$A_1$};

\draw (D4a1)  -- (A3+A1') node[midway, below left] {$A_1$};
\draw (A3+2A1)  -- (A3+A1') node[midway,above] {$A_1$};
\draw (A3+2A1)  -- (A3+A1'') node[midway,right] {$b_3$};

\draw (A3+A1')  -- (A3) node[midway,below right] {$b_3$};
\draw (A3+A1')  -- (2A2+A1) node[midway,above=2pt, left=1pt] {$m$};

\draw (A3+A1'')  -- (A3) node[midway,right] {$A_1$};
\draw (A3+A1'')  -- (2A2) node[near end,above] {$A_1$};

\draw (2A2+A1) -- (2A2) node[midway,above=2pt, right=1pt] {$g_2$};
\draw (2A2+A1) -- (A2+3A1) node[midway,above=2pt, left=1pt ] {$g_2$};

\draw (2A2) -- (A2+2A1) node[midway, right] {$A_1$};
\draw (A3) -- (A2+2A1) node[midway,below] {$A_1$};
\draw (A2+3A1) -- (A2+2A1) node[midway,below] {$A_1$};

\draw (A2+2A1) -- (A2+A1) node[midway,right] {$a^{+}_3$};
\draw (A2+A1) -- (A2) node[midway,above] {$a^{+}_5$};
\draw (A2+A1) -- (4A1) node[midway,above] {$c_3$};
\draw (4A1) -- (3A1') node[midway,above] {$c_3$};
\draw (4A1) -- (3A1'') node[midway,above] {$f_4$};
\draw (A2) -- (3A1') node[midway,below] {$A_1$};
\draw (3A1') -- (2A1) node[midway,below] {$b_4$};
\draw (3A1'') -- (2A1) node[midway, above] {$A_1$};
\draw (2A1) -- (A1) node[midway,right] {$d_6$};
\draw (A1) -- (0) node[midway,right] {$e_7$};

\end{scope}
\end{tikzpicture}

\begin{tikzpicture}
\tiny
  \matrix (galois)
    [matrix of math nodes,%
     nodes in empty cells,
     nodes={outer sep=0pt,minimum size=0pt},
     column sep={1.5cm,between origins},
     row sep={.56cm,between origins}]
  {
&&&&   | (E8) | E_8 \\
&&&&  | (E8a1) | E_8(a_1) & & &\\
&&&& | (E8a2) | E_8(a_2) \\
&&&& |  (E8a3) | E_8(a_3) \\
&&&| (E7) | E_7 & | (E8a4) | E_8(a_4) \\
&&&& | (E8b4) | E_8(b_4) \\
&&& | (E7a1) | E_7(a_1) && | (E8a5) | E_8(a_5) \\
&&&&| (E8b5) | E_8(b_5) & | (D7) | D_7 \\
&&& | (E7a2) | E_7(a_2) & | (E8a6) | E_8(a_6) \\
&&| (E6+A1) | E_6 \! + \! A_1 &&| (D7a1) | D_7(a_1) \\
& | (E6) | E_6 &&  | (E7a3) | E_7(a_3) && | (E8b6) | E_8(b_6) \\
&&&| (D6) | D_6 & | (E6a1+A1) | E_6(a_1)\!+\! A_1 && | (A7) | A_7 \\
&&&&& | (D7a2) | D_7(a_2) \\
&&& | (E6a1) | E_6(a_1) && | (D5+A2) | D_5+A_2 \\
&&&& | (E7a4) | E_7(a_4) && | (A6+A1) | A_6+A_1 \\
&&& | (D6a1) | D_6(a_1) &&| (A6) | A_6 \\
&& | (D5+A1) | D_5+A_1 && | (E8a7) | E_8(a_7) \\
&| (D5) | D_5 && | (E7a5) | E_7(a_5) \\
&&& | (E6a3+A1) | E_6(a_3) \! +\! A_1 &&& | (D6a2) | D_6(a_2) \\
&& | (E6a3) | E_6(a_3) &&| (A5+A1) | A_5 \! + \! A_1 && | (D5a1+A2) | D_5(a_1) \! + \! A_2 \\
&&&& | (A5) | A_5 && | (A4+A3) | A_4 \! + \! A_3 && | (D4+A2) | D_4 \! + \! A_2 \\
&&&&| (D5a1+A1) | D_5(a_1)\! +\!A_1  &&&& | (A4+A2+A1) | A_4\! +\!A_2\! +\!A_1 \\
&&& | (D5a1) | D_5(a_1) &&& | (A4+A2) | A_4 \! +\!A_2 \\
&& | (D4+A1) | D_4 \! +\!A_1 &&& | (A4+2A1) | A_4 \! +\!2A_1 \\
& | (D4) | D_4 &&& | (A4+A1) | A_4 \! +\!A_1 && | (2A3) | 2A_3 \\
&&& && | (D4a1+A2) | D_4(a_1) \! +\! A_2 \\
&&& | (A4) | A_4 && | (A3+A2+A1) | A_3 \! +\!A_2 \! +\!A_1 \\
&&&& | (A3+A2) | A_3 \! +\!A_2 \\
&&&& | (D4a1+A1) | D_4(a_1) \! +\!A_1 \\
&&&& | (D4a1) | D_4(a_1) & | (A3+2A1) | A_3 \! +\!2A_1 \\
&&&&& | (A3+A1) | A_3+A_1 & | (2A2+2A1) | 2A_2+2A_1 \\
&&| (A3) | A_3 &&& | (2A2+A1) | 2A_2+A_1 \\
&&&&& | (2A2) | 2A_2 \\
&&&&| (A2+3A1) | A_2+3A_1 \\
&&& | (A2+2A1) | A_2+2A_1 \\
&&& | (A2+A1) | A_2+A_1 \\
&&| (4A1) | 4A_1 & | (A2) | A_2 \\
&&& | (3A1) | 3A_1 \\
&&& | (2A1) | 2A_1 \\
&&& | (A1) | A_1 \\
&&&| (0) | 0 \\
  };
  \begin{scope}[]
\draw (E8) -- (E8a1) node[midway,right] {$E_8$};
\draw (E8a1) -- (E8a2) node[midway,right] {$E_7$};
\draw (E8a2) -- (E8a3) node[midway,right] {$C_6$};
\draw (E8a3) -- (E8a4) node[midway,right] {$F_4$};
\draw (E8a3) -- (E7) node[midway,above] {$A_1$};   
\draw (E7) -- (E8b4) node[midway,above=4pt,right=-7pt] {$F_4$}; \draw (E8a4) --
(E8b4) node[midway,right] {$C_4$};
\draw (E8b4) -- (E7a1) node[midway, above] {$A_1$};   
\draw (E8b4) -- (E8a5) node[near end,above=-1pt] {$C_3$}; \draw
(E7a1) -- (E8b5) node[midway,right=5pt, above=-2pt] {$3 (C_5)$}; \draw
(E8a5) -- (E8b5) node[near end,above=-1pt] {$G_2$};
\draw (E8a5) -- (D7) node[midway,right=6pt, below=-5pt] {$A_1$};   
\draw (E8b5) -- (E7a2) node[midway, above=2pt,left=-3pt] {$A_1$};    
\draw (E8b5) -- (E8a6) node[midway,right=-3.5pt] {$(G_2)$}; \draw (D7) --
(E8a6) node[midway, right=3pt, below=-1.5pt] {$(G_2)$};

\draw (E7a2) -- (E6+A1) node[midway,above] {$m$};   
\draw (E7a2) -- (D7a1) node[near start,above=5pt,right=-8pt] {$(B_3)$}; \draw (E8a6)
-- (D7a1) node[midway,right=-1pt] {$C_2$};

\draw (E6+A1) -- (E6) node[midway, above] {$g_2$};
\draw (E6+A1) -- (E8b6) node[near start,right, above=-1pt] {$(G_2)$};

\draw (D7a1) -- (E7a3) node[midway,above=5pt,left=-10pt] {$A_1$};    
\draw (D7a1) -- (E8b6) node[midway,above=-1pt] {$\mu$};

\draw (E6) -- (E6a1) node[midway,right] {$F_4$};  
\draw (E7a3) -- (D6) node[midway,left] {$b_2$};
\draw (E7a3) -- (E6a1+A1) node[midway,above=5pt, right=-4pt] {$(A^{+}_4)$};  
\draw (E8b6) -- (E6a1+A1) node[midway, right=8pt,  below=-5.5pt] {$A^{+}_2$};   
\draw (E8b6) -- (A7) node[midway,right] {$A_1$};   

\draw (D6)--(D5+A2) node[midway,right] {$(C_2)$};
\draw (E6a1+A1)--(E6a1) node[near end,above=1pt] {$a^{+}_2$};  
\draw (E6a1+A1)--(D7a2) node[midway, right=7pt, above=-4pt]
{$(A_2^+)$}; \draw (A7) -- (D7a2) node[midway, right=4pt, below=-2pt]
{$(A_2^+)$};
\draw (D7a2) --  (D5+A2) node[midway,right] {$(C_2)$};  

\draw (E6a1)--(E7a4) node[midway,above] {$C_3$}; 
\draw (D5+A2)--(E7a4) node[midway, above] {$A_1$};    
\draw (D5+A2)--(A6+A1) node[midway,above=-1pt, right] {$A_1$};    

\draw (E7a4) --   (D6a1)   node[midway,below=2pt, right] {$[2A_1]^+$};   
\draw (E7a4) --   (A6)  node[midway,right] {$A_1$};  
\draw (A6+A1) --  (A6) node[midway,below] {$A_1$};  
\draw (D6a1) --   (D5+A1) node[midway,below right] {$A_1$};   
\draw (D6a1) --   (E8a7) node[midway,right] {$10G_2$}; \draw
(A6) --   (E8a7) node[midway,below] {$5G_2$};
\draw (D5+A1) --  (D5) node[midway,below right] {$b_3$};   
\draw (D5+A1) -- (E7a5)  node[midway,right] {$G_2$};
\draw (E8a7) --  (E7a5) node[midway, above left] {$A_1$};    

\draw (E7a5) --  (E6a3+A1) node[midway,right] {$m$};     
\draw (E7a5) --  (D6a2) node[midway,above] {$[2A_1]^+$};   
\draw (D6a2) --  (A5+A1) node[near start, above] {$m$};
\draw (D6a2) --  (D5a1+A2) node[midway,right] {$A_1$};  

\draw (D5) --   (E6a3) node[midway,right] {$C_3$};

\draw (E6a3+A1) --  (E6a3) node[midway, above left] {$g_2$};   
\draw (E6a3+A1) -- (A5+A1) node[midway,right] {$A_1$};   
\draw (E6a3+A1) -- (D5a1+A2) node[near start,above=-1pt] {$m$};   

\draw (E6a3) -- (A5) node[midway,right] {$A_1$};   
\draw (E6a3)  --  (D5a1+A1) node[midway,right] {$A_1$};

\draw (A5+A1) -- (A5) node[midway,right] {$g_2$};
\draw (A5+A1) -- (A4+A3)  node[midway, above=-1pt] {$m$};  
\draw (D5a1+A2) -- (A4+A3) node[midway,right] {$m$}; \draw
(D5a1+A2) -- (D4+A2) node[midway,above] {$a^{+}_2$};

\draw (A4+A3) -- (A4+A2+A1)node[midway,below] {$\chi$};
\draw (D4+A2) -- (A4+A2+A1)node[midway,right] {$A_1$};     
\draw (D4+A2) -- (D5a1+A1) node[midway,below right] {$A_1$};   

\draw (A5) -- (A4+A2) node[near start,above=-2pt] {$A_1$};
\draw (A4+A2+A1) -- (A4+A2) node[midway,below right] {$A_1$};   
\draw (D5a1+A1) -- (A4+A2) node[midway,below right] {$A_1$};   
\draw (D5a1+A1) -- (D5a1) node[midway,above left] {$a^{+}_3$};
\draw (D5a1) -- (D4+A1) node[midway,above] {$c_3$}; \draw (D5a1)
-- (A4+A1) node[midway,right] { $A_2^+$};
\draw (A4+A2) -- (A4+2A1) node[midway,right] {$A_1$};  

\draw (A4+2A1) -- (2A3) node[midway,above right] {$b_2$};
\draw (A4+2A1) -- (A4+A1) node[midway,above left] {$a^{+}_2$};

\draw (D4+A1) -- (D4) node[midway,above] {$f_4$};
\draw (D4+A1) -- (A3+A2+A1) node[near start,above] {$A_1$};
\draw (A4+A1) -- (D4a1+A2) node[midway,above right] {$a^{+}_2$};
\draw (A4+A1) -- (A4) node[near end, below, right=2pt] {$a^{+}_4$};
\draw (2A3) -- (D4a1+A2) node[midway,right] {$a^{+}_2$};
\draw (D4a1+A2) -- (A3+A2+A1) node[midway,right] {$A_1$};
\draw (A3+A2+A1)  -- (A3+A2) node[midway, below right] {$b_2$};

\draw (A4) -- (A3+A2) node[midway,below] {$C_2$};   

\draw (A3+A2)  -- (D4a1+A1) node[midway,right] {$[3A_1]^{++}$};
\draw (D4) -- (D4a1) node[midway,left] {$G_2$};  

\draw (D4a1+A1) -- (D4a1) node[midway,right] {$d^{++}_4$};
\draw (D4a1+A1) -- (A3+2A1) node[midway, above right] {$b_2$};

\draw (D4a1)  -- (A3+A1) node[midway, below left] {$A_1$};
\draw (A3+2A1)  -- (A3+A1) node[midway,right] {$b_3$};
\draw (A3+2A1)  -- (2A2+2A1) node[midway,above right] {$m'$};

\draw (A3+A1)  -- (A3) node[midway,above] {$b_5$};
\draw (A3+A1)  -- (2A2+A1) node[midway,right] {$m$}; 

\draw (2A2+2A1) -- (2A2+A1) node[midway, below right] {$g_2$};
\draw (2A2+A1) -- (2A2) node[midway,right] {$[2g_2]^+$};

\draw (2A2) -- (A2+3A1) node[midway,above left] {$A_1$};
\draw (A3) -- (A2+2A1) node[midway,above left] {$A_1$};  
\draw (A2+3A1) -- (A2+2A1) node[midway,above left] {$b_3$};
\draw (A2+2A1) -- (A2+A1) node[midway,right] {$a^{+}_5$};
\draw (A2+A1) -- (4A1) node[midway,above] {$c_4$};
\draw (A2+A1) -- (A2) node[midway,right] {$e^{+}_6$};
\draw (A2) -- (3A1) node[midway,right] {$A_1$};
\draw (4A1) -- (3A1) node[midway,below] {$f_4$};
\draw (3A1) -- (2A1) node[midway,right] {$b_6$};
\draw (2A1) -- (A1) node[midway,right] {$e_7$};
\draw (A1) -- (0) node[midway,right] {$e_8$};

\end{scope}
\end{tikzpicture}

\bibliographystyle{myalpha}
\bibliography{fjls}

%
%
%
%
%
%
%
%
%
%
%
%
%

\end{document}